\documentclass{article}
\usepackage[T1]{fontenc}
\usepackage{titlesec}
\usepackage[utf8]{inputenc}
\usepackage{amsmath}
\usepackage{amssymb}
\usepackage{amsfonts}
\usepackage{amsthm}
\usepackage{geometry}
\usepackage{babel}
\usepackage{xcolor}
\newtheorem{lemma}{Lemma}
\newtheorem{theorem}{Theorem}

\newtheorem*{corollary}{Corollary}
\numberwithin{equation}{section}
\title{PhD}
\author{}
\date{January 2021}

\begin{document}

\begin{center}
    \textbf{Universality for Random Matrices}
\end{center}

%\begin{center}
%     Simona Diaconu\footnote{Department of Mathematics, Stanford University, sdiaconu@stanford.edu}
%\end{center}
\begin{center}
     Simona Diaconu\footnote{Courant Institute, New York University, simona.diaconu@nyu.edu}
\end{center}

\begin{abstract}
    Traces of large powers of real-valued Wigner matrices are known to have Gaussian fluctuations: for \(A=\frac{1}{\sqrt{n}}(a_{ij})_{1 \leq i,j \leq n}\in \mathbb{R}^{n \times n}, A=A^T\) with \((a_{ij})_{1 \leq i \leq j \leq n}\) i.i.d., symmetric, subgaussian, \(\mathbb{E}[a^{2}_{11}]=1,\) and \(p=o(n^{2/3}),\) as \(n,p \to \infty,\) \(\frac{\sqrt{\pi}}{2^{p}}(tr(A^p)-\mathbb{E}[tr(A^p)]) \Rightarrow N(0,1).\) This work shows the entries of \(A^{2p},\) properly scaled, also have asymptotically normal laws when \(n \to \infty, p=n^{o(1)}:\) the normalizations of the diagonal entries depend on \(\mathbb{E}[a_{11}^4],\) contributions that become negligible as \(p \to \infty,\) whereas their counterparts in \(A^{2p+1}\) depend on all the moments of \(a_{11}\) when \(p\) is bounded or the moments grow fast relatively to \(p.\) This result demonstrates large powers of Wigner matrices are roughly Wigner matrices with normal entries when \(a_{11} \overset{d}{=} -a_{11},\mathbb{E}[a^{2}_{11}]=1, \mathbb{E}[|a_{11}|^{8+\epsilon_0}] \leq C(\epsilon_0),\) providing another perspective on eigenvector universality, which until now has been justified primarily via local laws. The last part of this paper finds the first-order terms of traces of Wishart matrices in the random matrix theory regime, rendering yet another connection between Wigner and Wishart ensembles, as well as an avenue to extend the results herein for the former to the latter. The primary tools employed %behind the entry CLTs 
    are the method of moments and a simple identity the Catalan numbers satisfy. %\textcolor{red}{8+epsilon/large moment}
\end{abstract}

%\tableofcontents

\section{Introduction}\label{intro}

Random matrices are employed in a plethora of disciplines, including statistics, physics, genetics, and computer science. Two of the most commonly studied families are Wigner and Wishart ensembles: the former are named after Eugene Wigner, who proposed them  in \(1955\) as a model for the organization of heavy nuclei (\cite{wigner}), while the latter were introduced by John Wishart in \(1928,\) his main motivation being multivariate populations, i.e., observations or feature measurements (\cite{wishart}). Initially the entries of these matrices were assumed to be normally distributed (real- or complex-valued), and there has been a great interest in understanding the sensitivity of these results %on their eigenstructure 
to their laws, in particular, how much Gaussianity could be weakened. 
%these phenomena hold. 
Mehta~\cite{mehta} conjectured in \(1991\) that the asymptotic behavior of large random matrices is universal: for each family of ensembles, it should depend on few moments of the entry distributions. Ever since, considerable progress has been made in this direction, primarily with three tools, \((i)\) orthogonal polynomials (bulk eigenvalues), \((ii)\) method of moments (edge eigenvalues), and \((iii)\) local laws (most of the eigenspectrum and its corresponding eigenvectors): a work that does not fit within these categories is Johansson's~\cite{johansson}, where the author deals with matrices having a Gaussian component and computes eigenvalue densities exactly. Inasmuch as the results of this paper rely on the method of moments and are concerned primarily with eigenvectors, the reader is referred to \cite{bleher}, \cite{deift1}, \cite{deift2}, \cite{pasturscherbina} for \((i),\) while \((ii)\) and \((iii)\) are looked at next in more detail in the context of Wigner matrices, Wishart ensembles being discussed subsequently. Consider solely real-valued Wigner matrices, i.e., \(A=\frac{1}{\sqrt{n}}(a_{ij})_{1 \leq i,j \leq n} \in \mathbb{R}^{n \times n}, A=A^T\) with \((a_{ij})_{1 \leq i \leq j \leq n}\) i.i.d. random variables (results for their complex-valued counterparts exist as well), and denote by \(\lambda_1(M) \geq \lambda_2(M) \geq ... \geq \lambda_n(M)\) the eigenvalues of a symmetric matrix \(M \in \mathbb{R}^{n \times n}\) with corresponding eigenvectors \((u_i)_{1 \leq i \leq n},\) chosen to form an orthonormal basis of \(\mathbb{R}^m,\) %if need be, 
and \(u_{ik}=(u_i)_{k1}\) for \(1 \leq i,k \leq n.\)
\par
The second technique \((ii),\) combinatorial in nature, is based on the seminal work of Sinai and Soshnikov~\cite{sinaisosh}, where the authors computed the order of \(\mathbb{E}[tr(A^p)]\) for \(p=o(n^{1/2})\) under the assumption that the distribution of \(a_{11}\) is symmetric and subgaussian of bounded norm (i.e., there exists \(C>0\) such that \(\mathbb{E}[a_{11}^{2k}] \leq (Ck)^{k}\) for all \(k \in \mathbb{N}\)). When \(\mathbb{E}[a^2_{11}]=1, p=o(n^{1/2}),\) as \(n,p \to \infty\)
\begin{equation}\label{tracesinaisosh}
    \mathbb{E}[tr(A^p)]=\begin{cases}
    0, & \hspace{0.4cm}p=2s-1, s \in \mathbb{N},\\
    \frac{4^s}{\sqrt{\pi}} \cdot \frac{n}{s^{3/2}}(1+o(1)), & \hspace{0.4cm} p=2s, s \in \mathbb{N}, 
    \end{cases}
\end{equation} 
a convergence employed to derive the asymptotic behavior of \(tr(A^p)\) in the aforesaid regime, %as \(n,p \to \infty, p=o(n^{1/2}),\)
\begin{equation}\label{cltsinaisosh}
    \frac{\sqrt{\pi}}{2^{p}}(tr(A^p)-\mathbb{E}[tr(A^p)]) \Rightarrow N(0,1)
\end{equation}
(see main theorem in \cite{sinaisosh}).
In an ensuing paper \cite{sinaisosh2}, they extended these results to \(p=o(n^{2/3}),\) and Soshnikov~\cite{sosh} further exploited the counting device underlying (\ref{tracesinaisosh}) to justify the joint distribution of the largest \(k\) eigenvalues of \(A\) is universal for \(k\) fixed: it depends exclusively on the first two moments of \(a_{11}.\)% (\(k\) fixed). 
%it does not depend on any moments of \(a_{11}\) but its first two.% (\(k\) fixed). 
\par
The third technique \((iii),\) analytical in nature, was introduced by Erdös, Schlein, and Yau in \cite{locallaws1}, and has ever since been widely employed to derive the asymptotic behavior of both eigenvalues and eigenvectors of random matrices (\cite{locallaws3}, \cite{locallaws4}, \cite{locallaws2}, \cite{sinekernels}, \cite{erdospeche}, etc.). Its gist is an interpolation between \(A\) and \(A^{G},\) a real-valued Wigner matrix whose entries have Gaussian distributions, that preserves the asymptotic laws of the eigenvalues and eigenvectors, the crucial ingredient being an entry-wise control of \((A-zI)^{-1}\) for \(z \in \mathbb{C}^{+},\) \(||z||\) bounded, and \(Im(z) \geq Cn^{-1}.\) A more recent work \cite{bourgadeyau} by Bourgade and Yau considered the stochastic differential equations describing the evolution of the eigenvalues and eigenvectors of perturbations of \(A\) by a multi-dimensional Brownian motion, which the authors call the Dyson Brownian motion and vector flow, respectively. 
\par
One striking application of the method in \cite{bourgadeyau} is connected to quantum unique ergodicity, a stronger version of unique ergodicity, conjectured by Rudnick and Sarnak~\cite{rudnicksarnak} for arithmetic hyperbolic manifolds: the measures \(\mu_j\) induced by the eigenstates \(\psi_j\) of the Laplace-Beltrami operator of such negatively curved and compact manifolds converge weakly to the volume measure as \(j \to \infty.\) In the context of random matrices, this amounts to a strong notion of eigenvector delocalization (see also the introduction of \cite{yau2}). A local version of it is shown in \cite{bourgadeyau}: there exist \(\epsilon_0, \delta_0>0\) such that for \(\delta>0, \alpha_1, \alpha_2, \hspace{0.05cm} ... \hspace{0.05cm}, \alpha_n \in [-1,1], \alpha_1+\alpha_2+...+\alpha_n=0,\) 
\begin{equation}\label{bouryau}
    \mathbb{P}(\frac{n}{|\{i: 1 \leq i \leq n, \alpha_i \ne 0\}|} \cdot |\sum_{1 \leq i \leq n}{\alpha_iu^2_{ik}}|>\delta) \leq C \cdot (n^{-\epsilon}+|\{i: 1 \leq i \leq n, \alpha_i \ne 0\}|^{-1})
\end{equation}
for \(k \in T_N=([1,n^{1/4}] \cup [n^{1-\delta_0},n-n^{1-\delta_0}] \cup [n-n^{1/4},n])\cap \mathbb{N},\) and \(C>0\) universal (corollary \(1.4\)). Their method has been not only latterly used for other models (e.g., \cite{benigni}, \cite{benigni2}), but also improved by Cipolloni, Erdös, and Schröder~\cite{cipolini} to justify a strong form of quantum unique ergodicity: for any \(\epsilon,C>0,\) and \(M \in \mathbb{R}^{n \times n}, ||M|| \leq C,\)
\begin{equation}\label{cipo}
    \mathbb{P}(\max_{1 \leq i,j \leq n}{|u^T_iMu_j-\frac{1}{n}tr(M) \cdot \chi_{i=j} |} \leq n^{\epsilon-1/2})=1-o(1).
\end{equation}
It must be mentioned the assumptions on the entries of \(A\) for both (\ref{bouryau}) and (\ref{cipo}) are not stringent (finite moments of all orders), with the model in \cite{bourgadeyau} even allowing for some variability in the variances of the entries (the so-called generalized Wigner matrices).
\par
Consider now Wishart matrices, a family of ensembles closely connected to Wigner models: \(A=\frac{1}{N} XX^T,\) with \(X=(x_{ij})_{1 \leq i \leq n, 1 \leq j \leq N} \in \mathbb{R}^{n \times N}\) containing i.i.d. entries. Oftentimes results for the latter have been extended to the former in the random matrix theory regime (i.e., \(\lim_{n \to \infty}{\frac{n}{N}}=\gamma\)) by a symmetrization trick: employ \(H=\begin{pmatrix}
0 & X  \\
X^T & 0  \\
\end{pmatrix}\) instead of \( A.\) However, eigenvalue universality for alike Wishart matrices with \(\gamma \ne 1\) has been justified exclusively through local laws, and a primary reason underlying this discrepancy between these two families is related to combinatorics. The key in the method of moments, when applied to a random matrix \(A \in \mathbb{R}^{n \times n},\) is computing 
\[\mathbb{E}[tr(A^p)]\]
for \(p \in \mathbb{N}\) growing with \(n.\) On the one hand, %For Wigner matrices, %\(p= \lfloor tn^{2/3} \rfloor\) for \(t>0\) is needed; these 
the traces of Wigner matrices can be understood using the structure of cycles of length \(p:\)
\begin{equation}\label{trace}
    \mathbb{E}[tr(A^p)]=\sum_{(i_0,i_1, \hspace{0.05cm}... \hspace{0.05cm},i_{p-1})}{\mathbb{E}[a_{i_0i_1}a_{i_1i_2}...a_{i_{p-1}i_0}]},
\end{equation}
and due to the i.i.d. behavior, the expectations of such products are easy to compute as they factor into moments of \(a_{11}.\) On the other hand, when \(A=\frac{1}{N} XX^T\) for \(X \in \mathbb{R}^{n \times N}\) with i.i.d. entries, the contributions of individual cycles in (\ref{trace}) are more cumbersome. Opening the parentheses in
\begin{equation}\label{trcov}
    \mathbb{E}[tr(A^p)]=\sum_{(i_0,i_1, \hspace{0.05cm}... \hspace{0.05cm},i_{p-1})}{\mathbb{E}[(\sum_{1 \leq k \leq N}{x_{i_0k}x_{i_1k}})(\sum_{1 \leq k \leq N}{x_{i_1k}x_{i_2k}})...(\sum_{1 \leq k \leq N}{x_{i_{p-1}k}x_{i_0k}})]},
\end{equation}
produces \(N^p\) products for each cycle of length \(p,\) %with some of
%many products are created, 
%the graphs underlying them being disconnected, and additionally, there is 
and additionally, there is an implicit asymmetry created by \(n,N,\) the dimensions of \(X.\) %inherent to two dimensions %to deal with, 
%\(n,N.\) 
\par
%Historically, 
Traces as in (\ref{trcov}) have been handled in two contexts, almost sure limits of the edge eigenspectrum of \(A\) and edge universality. Assuming \(\lim_{n \to \infty}{\frac{n}{N}}=\gamma \in (0,\infty),\) Yin et al.~\cite{baiyinkrish} showed \(\lambda_1(A) \xrightarrow[]{a.s.} (1+\sqrt{\gamma})^2,\) for which the following was essential: when \(z>(1+\sqrt{\gamma})^2,\)
\begin{equation}\label{trineq}
    \sum_{n \geq 1}{z^{-k_n}\mathbb{E}[tr(A_n^{k_n})]}<\infty
\end{equation}
for \(k_n \in \mathbb{N}\) growing slightly faster than \(\log{n},\) and \(A_n:=\frac{1}{N}XX^T.\) Subsequently, Bai and Yin~\cite{baiyin} proved \(\lambda_1(A) \xrightarrow[]{a.s.} (1-\sqrt{\gamma})^2\) when \(\gamma \in (0,1)\) (for \(\gamma>1,\) the same result holds but for \(\lambda_{n-N+1}(A):\) this is due to the duality between \(XX^T\) and \(X^TX\)) by justifying the analog of (\ref{trineq}) for a family of matrices \(T(l),\) whose linear combinations yield, roughly speaking, the powers of \(A-(1+\gamma)I.\) 
\par
The authors of the aforementioned works considered the directed graphs underlying (\ref{trcov}), which could be visualized as two (parallel) lines \(L_1,L_2,\) each containing \(p\) points, with \(2p\) segments between them: Soshnikov~\cite{sosh} led another analysis of such traces. He found their asymptotic behavior for \(N=n+o(n^{1/3})\) and \(p=o(n^{2/3}),p=\lfloor tn^{2/3} \rfloor\) by an entirely different approach than the one just described, based on a comparison between traces of powers of \(A\) and those of a  Wigner matrix (the latter ranges of \(p\) demonstrate universality at the right edge of \(A\)). At a high level, \(N\) and \(n\) must be roughly equal in \cite{sosh} for the transition to a Wigner matrix to be tight: in such ensembles, a key feature of the cycles underlying (\ref{trace}) is all the vertices belonging to the same set, \(\{1,2, \hspace{0.05cm} ... \hspace{0.05cm},n\},\) whereas in the graphs underlying (\ref{trcov}), there are two types, elements of \(\{1,2, \hspace{0.05cm} ... \hspace{0.05cm},n\}\) and \(\{1,2, \hspace{0.05cm} ... \hspace{0.05cm},N\},\) respectively. This suggests a close connection between these families of random matrices when the two sets are roughly equal, whereas it is not a priori clear what occurs when this condition is violated. 
\par
This paper relies on the technique introduced by Sinai and Soshnikov in \cite{sinaisosh}, used for universality results (\cite{sosh}, \cite{feralpeche}) as well as upper bounds on eigenvalues (\cite{auffinger}, \cite{benaychpeche}), and a simple observation from \cite{oldpaper} allowing to compute the leading terms in analogs of (\ref{trcov}), originally employed for deriving the asymptotic law of the largest eigenvalue of Wigner matrices with entries at the boundary between light- and heavy-tailed regimes. Concretely, this work contains three types of results: CLTs in the spirit of (\ref{cltsinaisosh}), asymptotic Haar behavior of the matrix consisting of eigenvectors of a Wigner ensemble, and trace expectations for Wishart matrices, which translate into CLTs, revealing a novel connection between Wishart and Wigner ensembles and providing a means to extend the results herein to the former (despite this not being pursued here). Loosely speaking, the eigenvector behavior comes forth in this case through traces: it is yet another product of the methodology in \cite{sinaisosh}, originally devised for the latter and up to this point exploited primarily for eigenspectra. 
%In the present case, what comes forth is the eigenvector behavior, not that of the eigenvalues as it occurs in the aforementioned references. 
\vspace{0.4cm}
\par
\textbf{Acknowledgments:} The bulk of this work has been completed in the author's graduate studies at Stanford University: this version corrects mistakes the original version had in the proofs of the auxiliary lemmas with the primary results remaining almost unchanged. She is particularly grateful to George Papanicolaou and Lenya Ryzhik for their feedback on earlier versions of this paper.

\section{Main Results}

The first step towards comprehending eigenvectors of Wigner matrices is comprised of the following convergences, closely connected to (\ref{cltsinaisosh}).

\begin{theorem}\label{th4}
For \(\epsilon_0>0\) fixed, let \(A=\frac{1}{\sqrt{n}}(a_{ij})_{1 \leq i,j \leq n} \in \mathbb{R}^{n \times n}, A=A^T,\) \((a_{ij})_{1 \leq i \leq j \leq n}\) i.i.d. with 
\begin{equation}\label{condd0}\tag{W1}
    a_{11} \overset{d}{=} -a_{11}, \hspace{0.2cm} \mathbb{E}[a^2_{11}]=1, \hspace{0.2cm}, \mathbb{E}[|a_{11}|^{8+\epsilon_0}] \leq C(\epsilon_0),
\end{equation}
and \(e_1,e_2, \hspace{0.05cm} ... \hspace{0.05cm},e_n\) the standard basis in \(\mathbb{R}^n.\) Then for \(p \in \mathbb{N}, i,j \in \{1, \hspace{0.05cm} ... \hspace{0.05cm},n\}, i \ne j,\) \(p=n^{o(1)},\) as \(n \to \infty,\) 
\[(a) \hspace{0.2cm} \sqrt{\frac{n}{c(p,\mathbb{E}[a^4_{11}])}}\cdot (e_i^TA^{2p}e_i-\frac{1}{n}\mathbb{E}[tr(A^{2p})]) \Rightarrow N(0,1),\]
%\frac{\sqrt{n}}{c(2p,\mathbb{E}[a^4_{11}])}
\[(b) \hspace{0.2cm} \sqrt{\frac{n}{C_{p}-C^2_{p/2}\chi_{2|p}}}\cdot e_i^TA^{p}e_j \Rightarrow N(0,1),\]
where \(c:\mathbb{N}\times [1,\infty) \to [1,\infty),\)
\[c(p,1+\delta)=\mathcal{A}_1(p)+\mathcal{A}_2(p) \cdot (\delta-1)\]
where \(\mathcal{A}_1(p),\mathcal{A}_2(p)\) are given by (\ref{ref1}). Furthermore, %for an absolute \(\sigma>0,\)
\begin{equation}\label{limm}
    \lim_{p \to \infty}{\frac{c(p,1+\delta)}{C_{p-1}}}=\frac{\sigma(\sigma+6)}{2},
\end{equation}
where \(\sigma>0\) is defined by (\ref{sigmadef}).% (\ref{formm}), and (\ref{keyeq})).
\end{theorem}

Some remarks on the statement of Theorem~\ref{th4} are in order. The condition \(p=n^{o(1)}\) is likely suboptimal, and given \cite{sosh}, \(p=o(n^{2/3})\) is necessary: the proof shows the diagonal entries of \(A^{2p+1}\) have heavier tails than Gaussian, their \((2l)^{th}\) moment being a linear combination of \((\mathbb{E}[a^{2k}_{11}])_{0 \leq k \leq l},\) with weights decaying polynomially in \(p.\) This implies that when \(p\) is bounded, the behavior is not universal (it depends on all moments of \(a_{11}:\) this is illustrated also by the simple case \(p=1\)) and is asymptotically Gaussian when the moments \(a_{11}\) do not grow too fast relatively to it,
\begin{equation}\label{pcondie}
    \lim_{n \to \infty}{(\log{p}-\frac{2\log{\mathbb{E}[a^{2l}_{11}\chi_{|a_{11}| \leq n^{\delta}}]}}{3l})}=\infty,
\end{equation}
%\[\hspace{0.3cm} \mathbb{E}[a^{2l}_{11}\chi_{|a_{11}| \leq n^{\delta}}] =O_l(n^{1-4\delta}),\]
for some \(\delta \in (\frac{2}{8+\epsilon_0},\frac{1}{4})\) and 
all \(l \in \mathbb{N}\): e.g., \(p \to \infty\) and \(\mathbb{E}[a^{2l}_{11}] \leq C(l)\) suffice (see end of subsection~\ref{treigen22}). The moment condition in (\ref{condd0}) is a weakening of subgaussianity, the assumption behind (\ref{cltsinaisosh}), and comes at the cost of a lower range for \(p\) due to the trade-off between the range of \(p\) and tails of \(a_{11}:\) the larger the former, the more first-order terms in the quantities of interest (see discussion at the beginning of section~\ref{treigen}). The function \(c\) appears when computing the variance of \(e_i^TA^{p}e_i-\frac{1}{n}\mathbb{E}[tr(A^{p})]\) (once it is expressed as a sum over cycles of length \(p\)), whose leading components are identified with a simple property underlying the Catalan numbers ((\ref{easycatalan}) and the description preceding it): an interesting feature of this result is the fourth moment appearing in the normalization. This statistic is known to separate the heavy- and light-tailed regimes: when the law of \(a_{11}\) is regularly varying with index \(\alpha \in (0,4)\) (in particular, \(\mathbb{E}[|a_{11}|^{\alpha+\epsilon}]=\infty\) for all \(\epsilon>0\)), its (right) edge eigenvalues, properly normalized (roughly by \(n^{2/\alpha},\) not \(\sqrt{n}\)), follow a Poisson point process and their corresponding eigenvectors are completely delocalized (equidistributed on two entries: see \cite{sosh2}, \cite{auffinger}), standing in stark contrast with the case \(\mathbb{E}[a^4_{11}]<\infty.\)
\par
A consequence of \(a_{11} \overset{d}{=} -a_{11}\) is \(\mathbb
{E}[(\Tilde{A}_p)_{ij} \cdot (\Tilde{A}_p)_{i'j'}]=0\) when \(i<j,i'<j', \{i,j\} \ne \{i',j'\},\) where 
\[\Tilde{A}_p=\sqrt{\frac{n}{C_p-C^2_{p/2}\chi_{2|p}}}(A^p-\frac{1}{n}\mathbb{E}[tr(A^{p})]I),\] %\sqrt{\frac{n}{C_p-C^2_{p/2}\chi_{2|p}}}
thereby suggesting
\begin{equation}\label{heuristic}
    \Tilde{A}_p \approx A^{G}
\end{equation}
when \(2|p\) for \(A^{G}=(a_{ij})_{1 \leq i,j \leq n}\) a symmetric matrix with \((a_{ij})_{1 \leq i \leq j \leq n}\) independent,  
normally distributed, %and %centered, and 
\[
\mathbb{E}[a_{ij}]=0, \hspace{0.4cm} \mathbb{E}[a_{ij}^{2}]=\begin{cases}
    \frac{\sigma(\sigma+6)}{8}=\lim_{p \to \infty}{\frac{c(p,\mathbb{E}[a^4_{11}])}{C_p}},\hspace{0.95cm} i=j,\\
    1, \hspace{4.9cm} i \neq j.
\end{cases}
\]
%\[\mathbb{E}[a_{ii}^{2}]=\frac{1}{2}=\lim_{p \to \infty}{\frac{2C_{p-1}}{C_p-C^2_{p/2}\chi_{2|p}}}, \hspace{0.8cm} \mathbb{E}[a^2_{ij}]=1, i \ne j.\] 
To make (\ref{heuristic}) rigorous, a pseudo-distance between the metrics induced by \(\Tilde{A}_p\) and \(A^{(G)}\) %respectively, must be
is considered as the dimensions of these matrices increase with \(n.\) 
\par
Let \(m=\frac{n(n+1)}{2}, \mathbb{R}^{m}=\{(a_{ij})_{1 \leq i \leq j \leq n},a_{ij} \in \mathbb{R}\}:\) denote by \(\mu_{n,p}\) the joint distribution of \(((\Tilde{A}_p)_{ij})_{1 \leq i \leq j \leq n},\) and \(\nu_{n,p}\) the law of a multivariate Gaussian distribution of dimension \(m,\) whose entries \((x_{ij})_{1 \leq i \leq j \leq n}\) are independent, centered and
\begin{equation}\label{dij}
    \mathbb{E}[x^2_{ij}]=d_{ij}=\begin{cases}
    \frac{c(p,\mathbb{E}[a^4_{11}])}{C_{p}},\hspace{1cm} i=j,\\
    1, \hspace{2.2cm} i \neq j.
    \end{cases}
\end{equation}
By convention,
%the law given by
%\[\mathbb{E}_{\mu_{n,p}}[f(X)]=\int_{\mathbb{R}^m}{f(X)\prod_{1 \leq i \leq j \leq n}{da_{ij}}}, \hspace{0.8cm} X=((\Tilde{A}_p)_{ij})_{1 \leq i \leq j \leq n},\]
\begin{equation}\label{munp}
   \mathbb{E}_{\mu_{n,p}}[f(X)]=\int{fd\mu_{n,p}}=\mathbb{E}[f(X)], \hspace{0.8cm} X \overset{d}{=}((\Tilde{A}_p)_{ij})_{1 \leq i \leq j \leq n}, 
\end{equation}
\begin{equation}\label{nunp}
    \hspace{1.1cm} \mathbb{E}_{\nu_{n,p}}[f(X)]=\int{fd\nu_{n,p}}=\mathbb{E}[f(X)], \hspace{0.8cm} X \overset{d}{=} N(0,diag(d_{ij})_{1 \leq i \leq j \leq n}),
\end{equation}
for Borel measurable functions \(f:\mathbb{R}^m \to \mathbb{R}_{\geq 0}.\) %is Borel measurable, and \(\nu_{n,p}\) a multivariate Gaussian distribution of dimension \(m,\) whose entries \((x_{ij})_{1 \leq i \leq j \leq n}\) have diagonal covariance with
%\[\mathbb{E}[x^2_{ij}]=\begin{cases}
%    \frac{c^2(p,\mathbb{E}[a^4_{11}])}{C_p},\hspace{0.8cm} i=j,\\
%    1, \hspace{2.2cm} i \neq j.
%    \end{cases}\]
%\[\mathbb{E}[x^2_{ii}]=\frac{c^2(p,\mathbb{E}[a^4_{11}])}{C_p}, \hspace{0.8cm} \mathbb{E}[x^2_{ij}]=1, i \ne j.\] 
Additionally, take 
\[\mathbb{E}_{\mu_{n}}[f(X)]=\int{fd\mu_{n}}=\mathbb{E}[f(X)], \hspace{0.8cm} X \overset{d}{=}(A_{ij})_{1 \leq i \leq j \leq n},\]
\[\hspace{1.4cm} \mathbb{E}_{\nu_{n}}[f(X)]=\int{fd\nu_{n}}=\mathbb{E}[f(X)], \hspace{0.8cm} X \overset{d}{=} N(0,diag(l_{ij})_{1 \leq i \leq j \leq n}),\]
with
\[l_{ij}=\begin{cases}
    \frac{\sigma(\sigma+6)}{8},  \hspace{1cm} i=j,\\
    1, \hspace{1.9cm} i \ne j,
    \end{cases}\]
(\(\mu_n,\nu_n\) are, up to an extent, limits of \(\mu_{n,p},\nu_{n,p},\) respectively, in light of (\ref{limm})).
%\[d\mu_n:=\prod_{1 \leq i \leq j \leq n}{da_{ij}},\hspace{0.3cm}  d\nu_n=\prod_{1 \leq i \leq j \leq n}{da^{G}_{ij}},\] 
%with \((a_{ij}^G)_{1 \leq i \leq j \leq n}\) independent Gaussian random variables with \(\mathbb{E}[a^G_{ij}]=0, \mathbb{E}[(a^G_{ij})^2]=\frac{1}{2}(1+\chi_{i \ne j}).\)% for \(1 \leq i \leq j \leq n.\) 
%\(\mathbb{E}[(a^G_{ii})^2]=\frac{1}{2}(1+\chi_{i \ne j}), \mathbb{E}[(a^G_{ij})^2]=1\) when \(1 \leq i<j \leq n.\) 
\par
Recall the distance giving weak convergence: 
\[d(\eta_1,\eta_2)=\sup_{||f||_{\infty}=1, f \in C(\mathbb{R}^{m})}{|\int{fd\eta_1}-\int{fd\eta_2}|},\]
where \(||f||_{\infty}=\sup_{x \in \mathbb{R}^m}{|f(x)|},\) and \(C(\mathbb{R}^{m})\) is the set of continuous functions \(f:\mathbb{R}^m \to \mathbb{R}.\) A rigorous interpretation \(\mu_{n,p} \approx \nu_{n,p}\) would be
\[d(\mu_{n,p},\nu_{n,p})\xrightarrow[]{n \to \infty} 0.\]
The following result is a weaker version of this ideal convergence.

\begin{theorem}\label{th55}
    Under the assumptions of Theorem~\ref{th4}, there exist %universal constants 
    \(c_0(\epsilon_0),c_1,c_2 \in (0,1)\) such that for \(p \in \mathbb{N}, \newline p \leq c_0(\epsilon_0) (\log{n})^{1/45},\) %\textcolor{red}{REVIEW STATEMENT}
    \begin{equation}\label{weakconv}
        d_{n,p}(\mu_{n,2p},\nu_{n,2p}):=\sup_{f \in \mathcal{R}_{n,p}}{|\int{fd\mu_{n,2p}}-\int{fd\nu_{n,2p}}|} \leq c^m_1+2C(\epsilon_0)n^{-\epsilon_0/8},
    \end{equation}
    where \(B_m(0,R)=\{y \in \mathbb{R}^m, ||y|| \leq R\}\) and
    \[\mathcal{R}_{n,p}=\{f \in C(\mathbb{R}^m), \hspace{0.1cm} supp(f) \subset B_m(0,c_2\sqrt{mp}), \hspace{0.1cm} ||f||_{\infty} \leq 1\}.\] 
\end{theorem}

Since the compactly supported elements of \(C(\mathbb{R}^m),\) %\(C_b(\mathbb{R}^m)=C(\mathbb{R}^m) \cap \{f:\mathbb{R}^m \to \mathbb{R}, ||f||_{\infty}<\infty\},\) 
a metric space equipped with the \(L^{\infty}\) norm, are dense in it,  Theorem~\ref{th55} restricts the sizes of the domains of the functions underlying the pseudo-distance \(d_{n,p},\) from arbitrary compact sets to balls centered at the origin of radius at most \(c_2\sqrt{mp}.\) Despite this condition being likely suboptimal, it already renders nontrivial results: by the law of large numbers, \(\nu_{n,p}\) concentrates around the boundary of the ball centered at the origin of radius \(\Theta(\sqrt{m}),\) suggesting most functions of interest in \(C(\mathbb{R}^n)\) with compact support are captured by Theorem~\ref{th55}. A consequence of this result is stated next.

\begin{corollary}\label{cor1}
Under the assumptions in Theorem~\ref{th4}, there exists \(c_3>0\) such that for any \((f_l)_{l \in \mathbb{N}} \subset C(\mathbb{R}^m)\) with \(\lim_{l \to \infty}{||f-f_l||_{\infty}}=0,\) \(f_l \geq 0\) for \(l \in \mathbb{N},\) and \(p \geq \frac{16}{c^2_2},\)
\[|\int{fd\mu_{n,2p}}-\int{fd\nu_{n,2p}}| \leq ||f||_{\infty}\cdot [6c^m_1+4C(\epsilon_0) n^{-\epsilon_0/8}+\frac{c_3}{mp^2}].\]
\end{corollary}

Because for any \(\lambda \in \mathbb{R},\)  \(A\) and \(A^p-\lambda I\) share eigenvectors, a notion of asymptotically Haar distributed for them ensues from the two results above.

\begin{theorem}\label{th7}
   Let \(\mathcal{R}_n\) be the set of functions \(f:Sym(n):=\{M \in \mathbb{R}^{n \times n}, M=M^T\} \to \mathbb{R}\) with \(||f||_{\infty} \leq 1,\) the restriction 
   \[f|_{Sym^d(n):=\{M \in \mathbb{R}^{n \times n}, M=M^T, \lambda_1(M)>...>\lambda_n(M)\}}\] 
   continuous, given by \(f(M)=h(u_1, u_2, \hspace{0.05cm} ... \hspace{0.05cm},u_n),\) for \(u_1,u_2, \hspace{0.05cm} ... \hspace{0.05cm}, u_n\) unit eigenvectors of \(M,\) \(h:(\mathbb{S}^{n-1})^{n} \to \mathbb{R}\) even in each of its components, and Borel measurable. Under the assumptions in Theorem~\ref{th4}, \(A \in Sym^d(n),\) and
    \[\sup_{f \in \mathcal{R}_n}{|\int{fd\mu_{n}}-\int{fd\nu_{n}}|}\to 0\]
    with high probability as \(n \to \infty.\) 
\end{theorem}

Theorem~\ref{th7} offers a new perspective on universality: eigenvectors of a Wigner matrix \(A \in \mathbb{R}^{n \times n}\) behave like their counterparts for \(A^G=(a^{G}_{ij})_{1 \leq i, j \leq n},\) and up to an extent, \(\Tilde{A}_p \approx A^G\) is a central limit theorem. Although this is justified by the same combinatorial device as (\ref{cltsinaisosh}), a result that requires growth conditions on all moments (out of which subgaussianity is the most common), the current context permits a relaxation of such constraints primarily because eigenvectors are considerably more invariant than eigenvalues: properties of the former for \(A^p-\lambda I\) stream down to \(A,\) whereas the latter are evidently highly dependent on \(p.\) Furthermore, edge eigenvalue universality using the method of moments relies on the asymptotic behavior of \(\mathbb{E}[tr(A^p)]\) for \(p \approx tn^{2/3}\) and all \(t>0\) being universal: this forces the tails of \(a_{11}\) to be light (else the trace expectations might depend on all of its moments). Regarding tail conditions, an \(8+\epsilon_0\) moment is required chiefly for a truncation: a fourth is necessary as discussed above, and comparatively, \(\mathbb{E}[a_{11}^{2k}] \leq C(k)\) for all \(k \in \mathbb{N}\) is usually assumed when employing local laws (sometimes this can be relaxed to \(k \leq K\) for a sufficiently large \(K\)). The symmetry condition is mainly technical and might be dispensable, despite this relaxation not being pursued here (several works employ the method of moments and deal with distributions that are only centered: e.g., \cite{feralpeche}, \cite{oldpaper1}, \cite{oldpaper2}; it must also be pointed out symmetry plays no role in local laws). 
\par
For context, %the sake of completeness, 
Haar measures on the orthogonal group \(O(n):=\{O \in \mathbb{R}^{n \times n}, OO^T=I\}\) and their connection with Gaussian distributions are discussed next: this is meant to shed light into why matrices with normal entries are easier to analyze than generic ones, even when it comes to eigenvectors. One definition of the normalized Haar measure on \(O(n)\) is the distribution of \(W_1,\) where \(Z \in \mathbb{R}^{n \times n}\) has i.i.d. centered standard normal entries, and \(Z=W_1DW_2^T\) is an SVD-decomposition. Several parameterizations of \(O(n)\) have been discovered: for instance, if \(O \in O(n), det(O)=1,\) then \(O=e^T, O=(I-S)(I+S)^{-1}\) (the latter is called Cayley's transform) for \(T,S\) skew-symmetric. Another construction, generalized Eulerian angles, was introduced by Raffenetti and Ruedenberg~\cite{raffanettiruedenberg}; it is a recurrent procedure in which the elements of \(O(n)\) are decomposed into products of \(\frac{n(n-1)}{2}\) factors, each depending solely on one parameter: the building blocks are of the form \(a_{pq}(\alpha) \in \mathbb{R}^{n \times n}, \alpha \in [0,2\pi],\)
\[(a_{pq}(\alpha))_{ij}=\begin{cases}
    1, \hspace{1.4cm} i=j, i \not \in \{p,q\} \\
    \cos{\alpha}, \hspace{0.8cm} i=j, i \in \{p,q\} \\
    \sin{\alpha}, \hspace{0.8cm} i=p, j=q \\
    -\sin{\alpha}, \hspace{0.5cm} i=q, j=p\\
    0,  \hspace{1.4cm} else
\end{cases}\]
for \(1 \leq p<q \leq n;\) notwithstanding the inherent elegance of this decomposition, neither obtaining these parameters for a given matrix nor defining  a measure on them leading to left or right invariance on \(O(n),\) the primary feature of Haar measures, is clear (see \cite{review}). 
\par
A myriad of definitions can be nevertheless concocted for the Haar measures on \(O_n\) using multivariate Gaussian distributions, whose equivalence up to multiplication by a scalar, guaranteed by the uniqueness of such objects up to constant factors, is far from evident. One recipe is choosing a well-behaved function (e.g., \(F\) nonzero and continuous) \(F:Sym(n) \to [0,\infty)\) with \(F(M)=F(VMV^T)\) for all \(M \in Sym(n), V \in O(n),\) and letting for any Borel measurable set \(S \subset O_n,\)
\begin{equation}\label{haar1}
    \int_{S}{d\tilde{\mu}}:=\int_{Z \in S_\alpha}{dZ},
\end{equation}
%\begin{equation}\label{haar1}
%    \int_{S}{d\tilde{\mu}}:=(2\pi)^{-\frac{n^2}{2}}\int_{z \in S_\alpha}{e^{-\frac{1}{2}\sum_{1 \leq i,j \leq n}{z^2_{ij}}}\prod_{1 \leq i,j \leq n}{dz_{ij}}},
%\end{equation}
where \(Z \in \mathbb{R}^{n \times n}\) has i.i.d. centered standard normal entries (\(dZ:=(2\pi)^{-\frac{n^2}{2}}e^{-\frac{1}{2}\sum_{1 \leq i,j \leq n}{z^2_{ij}}}\prod_{1 \leq i,j \leq n}{dz_{ij}}),\) %and  
\[S_{\alpha}:=\{M \in \mathbb{R}^{n \times n}: \exists M_0 \in S, F((M-M_0)^T(M-M_0)) \leq \alpha\}\]
for some \(\alpha>0\) (note \(\tilde{\mu}\) is both left and right invariant under multiplication by elements of \(O(n)\)). Concentration and a careful selection of \(F\) ensure most random matrices \(Z \in \mathbb{R}^{n \times n}\) belong to \((O(n))_{\alpha_n}\) for a deterministic \(\alpha=\alpha_n,\) producing a probability measure from the right-hand side of (\ref{haar1}) upon scaling. 
\par
Consider a slightly different approach: in (\ref{haar1}) let
\begin{equation}\label{salpha}
    S_\alpha=\{X \in \mathbb{R}^{n \times n}: \forall i \in \{1,2, \hspace{0.05cm} ... \hspace{0.05cm},n\}, \exists M_i \in S, ||X_i-M_i|| \leq \alpha\},
\end{equation}
where \(Y_i \in \mathbb{R}^n\) is the \(i^{th}\) row of \(Y \in \mathbb{R}^{n \times n}.\) This is right-invariant because \((YV)_i=V^TY_i;\) since Haar measures (both left and right) are unique up to constant factors, and there exists one both right- and left- invariant (e.g., \(F(M)=tr(M^2)\)), (\ref{haar1}) and (\ref{salpha}) provide the Haar measures on \(O(n).\) This definition can be employed to derive delocalization properties of the eigenvectors of Wigner matrices: loosely speaking, one expects orthogonal matrices to be close to random matrices of the type \(\frac{1}{\sqrt{n}}Z\) (by the law of large numbers), entailing any \(S \subset O(n)\) with positive measure has elements small perturbations of some \(\frac{1}{\sqrt{n}}Z,\) whose largest entries are of order \(\sqrt{\frac{\log{n}}{n}}.\) Subsection~\ref{subsechaar2} presents in further detail how this representation renders, for instance,
\begin{equation}\label{haarprob}
    \mathbb{P}(U \in O(n): \max_{1 \leq i,j \leq n}{|u_{ij}|} \geq t \cdot \sqrt{\frac{\log{n}}{n}}) \leq n^2(1-\exp(-cn^{1-ct^2}))
\end{equation}
where \(c>0\) is universal, as well as (\ref{bouryau}) and (\ref{cipo}) (this discussion is included %for completeness, 
primarily because no reference with explicit justifications for Gaussian ensembles seems available).
\par
The last part of this paper is concerned with Wishart ensembles. Let \(X=(x_{ij})_{1 \leq i \leq n, 1 \leq i \leq N}\) have i.i.d. entries with \(\mathbb{E}[x_{11}]=0,\mathbb{E}[x^2_{11}]=1,\) and \(A=\frac{1}{N}XX^T.\) Silverstein~\cite{silverstein} extended the results of the seminal paper \cite{marchenkopastur} by Marchenko and Pastur and showed the almost sure weak convergence of the empirical spectral distribution of \(A=\frac{1}{N}XX^T\) to a probability distribution as \(n, N \to \infty\) with \(\frac{n}{N} \to \gamma \in (0,\infty):\) for \(F_n(x)=\frac{1}{N}\sum_{1 \leq i \leq N}{\chi_{x \geq \lambda_i(A)}},\)
\[F_n \xrightarrow[]{a.s.} F_\gamma,\]
where \(F_\gamma\) is the cdf of the probability distribution with density
\(f_{\gamma}(x)dx,\) and
\[f_\gamma(x)=\begin{cases} 
\frac{\sqrt{(b(\gamma)-x) \cdot (x-a(\gamma))}}{2\pi \gamma x}, & x \in [a(\gamma),b(\gamma)],\\
0, & x \not \in [a(\gamma),b(\gamma)],
\end{cases}\] 
for \(a(\gamma)=(1-\sqrt{\gamma})^2,b(\gamma)=(1+\sqrt{\gamma})^2.\) Lemma \(3.1\) in \cite{baisilvbook} gives its \(k^{th}\) moment is
\begin{equation}\label{betadef}
    \beta(k,\gamma):=\int_{a(\gamma)}^{b(\gamma)}{x^kf_\gamma(x)dx}=\sum_{0 \leq r \leq k-1}{\frac{1}{r+1}\binom{k}{r} \binom{k-1}{r}\gamma^{r}}.
\end{equation}
An elementary result unveils the recurrence satisfied by these moments, similar in spirit to the one satisfied by the Catalan numbers (which corresponds to \(\gamma=1\)).

\begin{lemma}\label{l1}
For \(k \in \mathbb{N},\) \(k \geq 2,\) and \(\gamma>0,\) %\textcolor{blue}{make range to k-1 and 1+gamma 1?}
\begin{equation}\label{momrec}
    \beta(k,\gamma)=(1+\gamma)\beta(k-1,\gamma)+\gamma \sum_{1 \leq a \leq k-2}{\beta(a,\gamma) \cdot \beta(k-a-1,\gamma)}.
\end{equation}
\end{lemma}

Lastly, the analog of (\ref{tracesinaisosh}) for Wishart matrices is comprised below.

\begin{theorem}\label{th1}
Let \(X=(x_{ij})_{1 \leq i \leq n, 1 \leq i \leq N}\) have i.i.d. entries with \(x_{11} \overset{d}{=} -x_{11},\) subgaussian of finite norm, i.e., \(\mathbb{E}[b^{2k}_{11}] \leq (Ck)^k\) for all \(k \in \mathbb{N}\) and some \(C>0,\) \(\mathbb{E}[x^2_{11}]=1,\) \(A=\frac{1}{N}XX^T,\) and \(\lim_{n \to \infty}{\frac{n}{N}}=\gamma \in (0,\infty).\) Then for \(p \in \mathbb{N},p=o(n^{1/2}),\gamma_n=\frac{n}{N},\) 
\begin{equation}\label{tracecovarnew}
    \mathbb{E}[tr(A^p)]=n\beta(p,\gamma_n) \cdot (1+O(\frac{p^2}{n})).
\end{equation}
\end{theorem}
With (\ref{tracecovarnew}) under the belt, a rationale analogous to the one employed by Sinai and Soshnikov~\cite{sinaisosh} to infer (\ref{cltsinaisosh}) from (\ref{tracesinaisosh}) can be used a trace CLT for \(tr(A^p).\) As previously mentioned, extending the results above for Wigner matrices to Wishart ensembles seems feasible, despite this extension not being pursued here; furthermore, the arguments below can be utilized for complex-valued matrices as well (the structure of the sets whose elements are the leading cycles, \(\mathcal{C}(l),\) makes this transparent).
\par
The remainder of the paper is organized as follows: Theorem~\ref{th4}, Theorem~\ref{th55} with its two consequences (Corollary~\ref{cor1}, Theorem~\ref{th7}), and  Theorem~\ref{th1} are presented in sections~\ref{treigen}, \ref{weakconvv}, and \ref{traccov}, respectively. 

\section{Entry CLTs}\label{treigen}

This section consists of the proof of Theorem~\ref{th4}: as mentioned in the introduction, the primary ingredients behind it are the counting device introduced by Sinai and Soshnikov~\cite{sinaisosh}, and the dominant cycles underlying the traces below, looked at in detail in \cite{oldpaper}. The crux of the former is encompassed by
\[\mathbb{E}[tr(B^{2q})]=n(\mathbb{E}[b^{2}_{11}])^q C_q \cdot (1+o(1)),\]
for a random matrix \(B=(b_{ij})_{1 \leq i,j \leq n} \in \mathbb{R}^{n \times n},B=B^T,\) under growth conditions on \(q\) relatively to \(n\) and entry moments %the moments of the entries of \(B\)
(e.g., \(q=o(n^{1/2}), o(1)=O(q^2n^{-1})\) when \(B\) has i.i.d. entries with  \(b_{11} \overset{d}{=} -b_{11},\) \(b_{11}\) subgaussian with bounded norm). Another version, proved in detail in subsection \(2.1\) of \cite{oldpaper} and employed here, is:
\par
\vspace{0.4cm}
Suppose \(p \in \mathbb{N}\) and \(B=(b_{ij})_{1 \leq i,j \leq n} \in \mathbb{R}^{n \times n},B=B^T,(b_{ij})_{1 \leq i \leq j \leq n}\) i.i.d., \(b_{11} \overset{d}{=} -b_{11},\) \(\mathbb{E}[b_{11}^2] \leq 1, \newline \mathbb{E}[b_{11}^{2l}] \leq L(n)n^{\delta(2l-4)}, 2 \leq l \leq p, \delta>0,\) and \(L:\mathbb{N} \to [1,\infty), L(n)<n^{2\delta}.\) Then for \(\delta_1=\frac{1}{4}-\delta,\)
\begin{equation}\label{subas}
    \mathbb{E}[tr(B^{2p})] \leq C_pn^{p+1}+C_pn^{p+1}L(n) \cdot Cp^2n^{-2\delta_1}
\end{equation}
when \(p \leq n^{\delta_1}, n \geq n(\delta).\)
\vspace{0.5cm}
\par
Let
\begin{equation}\label{deltadef}
    \delta=\delta(\epsilon_0)=\frac{2}{8+\epsilon_0}+\frac{\epsilon_0}{8(8+\epsilon_0)} \in (0,\frac{1}{4}), \hspace{0.5cm} A_s=\frac{1}{\sqrt{n}}(a_{ij}\chi_{|a_{ij}| \leq n^{\delta}}).
\end{equation}
%(\frac{2}{8+\epsilon_0},\frac{1}{4})
A union bound entails \(A=A_s\) with high probability since
\begin{equation}\label{approxx}
    \mathbb{P}(A \ne A_s)=\mathbb{P}(\max_{1 \leq i \leq j \leq n}{|a_{ij}|}>n^{\delta}) \leq n^2 \cdot \mathbb{P}(|a_{11}|>n^{\delta}) \leq n^2 \cdot \frac{C(\epsilon_0)}{n^{\delta(8+\epsilon_0)}}=C(\epsilon_0)n^{-\epsilon_0/8}=o(1).
\end{equation}
Note (\ref{subas}) holds for \(\sqrt{n}A_s:\) its entries are i.i.d., symmetric, and 
\begin{equation}\label{decaymomm}
    \mathbb{E}[a^2_{11}\chi_{|a_{11}| \leq n^{\delta}}] \in [1-n^{-2\delta},1], \hspace{0.3cm} \mathbb{E}[a^{2l}\chi_{|a_{11}| \leq n^{\delta}}] \leq C(\epsilon_0) n^{\delta(2l-4)}, l \in \mathbb{N}, l \geq 2:
\end{equation}
when \(2 \leq l \leq 4+\epsilon_0/2,\) the last claim follows from Hölder inequality, 
\[\mathbb{E}[a^{2l}_{11}] \leq (\mathbb{E}[|a_{11}|^{8+\epsilon_0}])^{\frac{2l-2}{6+\epsilon_0}} \cdot (\mathbb{E}[a^{2}_{11}])^{\frac{8+\epsilon_0-2l}{6+\epsilon_0}} \leq C(\epsilon_0),\] 
and for \(l>4+\epsilon_0/2,\) 
\begin{equation}\label{boundbigl}
    \mathbb{E}[a_{11}^{2l}\chi_{|a_{11}| \leq n^{\delta}}] \leq \mathbb{E}[|a_{11}|^{8+\epsilon_0} \cdot n^{\delta(2l-8-\epsilon_0)}\chi_{|a_{11}| \leq n^{\delta}}] \leq C(\epsilon_0) n^{\delta(2l-8-\epsilon_0)}.
\end{equation}
Since Slutsky's lemma and Carleman condition entail
\begin{equation}\label{moments}
    \lim_{n \to \infty}{\frac{n^{l/2}}{c^{l/2}(p,\mathbb{E}[a_{11}^4])}\mathbb{E}[(e_i^TA_s^{2p}e_i-\frac{1}{n}\mathbb{E}[tr(A_s^{2p})])^l]}=\begin{cases}
    0, & l=2l_0-1,l_0 \in \mathbb{N}\\
    (l-1)!!, & l=2l_0,l_0 \in \mathbb{N}
    \end{cases},
\end{equation}
\begin{equation}\label{moments2}
    \lim_{n \to \infty}{\frac{n^{l/2}}{(C_{p}-C^2_{p/2}\chi_{2|p})^{l/2}}\mathbb{E}[(e_i^TA_s^{p}e_j)^l]}=\begin{cases}
    0, & l=2l_0-1,l_0 \in \mathbb{N}\\
    (l-1)!!, & l=2l_0,l_0 \in \mathbb{N}
    \end{cases}
\end{equation}
suffice to deduce Theorem~\ref{th4} (see, for instance, lemmas \(B.1, B.2\) in \cite{baisilvbook}), and they are the subject of the rest of this section. By a slight abuse of notation, \(A=A_s\) in what follows: in particular, \(\mathbb{P}(|a_{11}| \leq n^{\delta})=1.\) 
\par
The goal of this section is justifying (\ref{moments}) and (\ref{moments2}): before proceeding, a summary of the technique developed in \cite{sinaisosh}, upon which these convergences rely, is in order. Specifically, its outcome is a change of summation in 
\begin{equation}\label{genexp}
    \mathbb{E}[tr(B^{q})]=\sum_{(i_0,i_1, \hspace{0.05cm} ... \hspace{0.05cm}, i_{q-1})}{\mathbb{E}[b_{i_0i_1}b_{i_1i_2}...b_{i_{q-1}i_0}]}:=\sum_{\mathbf{i}:=(i_0,i_1, \hspace{0.05cm} ... \hspace{0.05cm}, i_{q-1},i_0)}{\mathbb{E}[b_{\mathbf{i}}]},
\end{equation}
from cycles \(\mathbf{i}:=(i_0,i_1, \hspace{0.05cm} ... \hspace{0.05cm}, i_{q-1},i_0)\) to tuples of nonnegative integers \((n_1,n_2, \hspace{0.05cm} ... \hspace{0.05cm},n_q).\) %employed to infer the main contributors in (\ref{genexp}). 
Recall terminology and notation from \cite{sinaisosh}, necessary in what is to come. Interpret \(\mathbf{i}:=(i_0,i_1, \hspace{0.05cm} ...\hspace{0.05cm} ,i_{q-1},i_0)\) as a directed cycle with vertices among \(\{1,2, \hspace{0.05cm} ... \hspace{0.05cm}, n\},\) call \((i_{k-1},i_k)\) its \(k^{th}\) edge and \(i_{k-1}\) its \(k^{th}\) vertex for \(1 \leq k \leq q,\) where \(i_{q}:=i_0.\) By convention, for \(u,v \in \{1,2, \hspace{0.05cm} ... \hspace{0.05cm}, n\},\) \((u,v)\) denotes a directed edge from \(u\) to \(v,\) whereas \(uv\) is undirected (the former are the building blocks of the cycles underlying the trace in (\ref{genexp}), while the latter determine their corresponding expectations): consequently \(uv=vu,\) and \(m(uv)\) denotes its multiplicity in \(\mathbf{i},\) i.e., 
\[m(uv)=|\{r:1 \leq r \leq q, i_{r-1}i_r=uv\}|.\] 
Say \(\mathbf{i}\) is an \textit{even cycle} if each undirected edge has even multiplicity in it, i.e., \(\forall uv, 2|m(uv)\) (otherwise \(\mathbf{i}\) is \textit{odd}). 
Because the entries of \(A\) are symmetric, solely even cycles have non-vanishing contributions in (\ref{genexp}) (such trace expectations have also been analyzed when the entries do not have symmetric distributions: e.g., \cite{feralpeche}, \cite{oldpaper1}, \cite{oldpaper2}).
\par
The desired change of summation is then achieved as follows. %by mapping cycles \(\mathbf{i}\) to tuples of nonnegative integers \((n_1,n_2, \hspace{0.05cm} ... \hspace{0.05cm},n_q)\) and bounding the sizes of the preimages of this transformation as well as the expectations of their elements. 
For \(\mathbf{i},\) call an edge \((i_k,i_{k+1})\) and its right endpoint \(i_{k+1}\) \textit{marked} if an even number of copies of \(i_ki_{k+1}\) precedes it, 
\[|\{t \in \mathbb{Z}: 0 \leq t \leq k-1, i_ti_{t+1}=i_ki_{k+1}\}| \in 2\mathbb{Z},\] 
and for \(k \geq 0,\) denote by \(N_{\mathbf{i}}(k)\) the set of \(j \in \{1,2, \hspace{0.05cm} ... \hspace{0.05cm}, n\}\) marked exactly \(k\) times in \(\mathbf{i},\) and \(n_k:=|N_{\mathbf{i}}(k)|.\) This provides a mapping from cycles \(\mathbf{i}\) to tuples of nonnegative integers \((n_0,n_1, \hspace{0.05cm} ... \hspace{0.05cm},n_q):\) in particular, any vertex \(j \in \{1,2, \hspace{0.05cm} ... \hspace{0.05cm}, n\}\) of \(\mathbf{i},\) apart perhaps from \(i_0,\) is marked at least once (the first edge of \(\mathbf{i}\) containing \(j\) is of the form \((i,j)\) because \(i_0 \ne j,\) and no earlier edge is incident with \(j\)), and each even cycle \(\mathbf{i}\) has \(q/2\) marked edges, whereby
\begin{equation}\label{tuplecond}
    \sum_{0 \leq k \leq q}{n_k}=n, \hspace{0.2cm} \sum_{1 \leq k \leq q}{kn_k}= q/2.
\end{equation} 
%For \(0 \leq k \leq q/2,\) denote by \(N_{\mathbf{i}}(k)\) the set of \(j \in \{1,2, \hspace{0.05cm} ... \hspace{0.05cm}, n\}\) marked exactly \(k\) times in \(\mathbf{i},\) and \(n_k:=|N_{\mathbf{i}}(k)|.\) Then
%\begin{equation}\label{tuplecond}
%    \sum_{0 \leq k \leq q}{n_k}=n, \hspace{0.2cm} \sum_{1 \leq k \leq q}{kn_k}= q/2.
%\end{equation} 
Having constructed \((n_0,n_1, \hspace{0.05cm} ... \hspace{0.05cm},n_{q}),\) it remains to bound the number of cycles mapped to a given tuple and their expectations, a task steps \(1-5\) summarized below undertake (see subsection \(2.1\) in \cite{oldpaper} for details).
Fix \((n_0,n_1, \hspace{0.05cm} ... \hspace{0.05cm}, n_p),\) and let \(\mathbf{i}\) be an even cycle of length \(2p\) mapped to it by the procedure just described.
\vspace{0.3cm}
\par
\underline{Step \(1.\)} Map \(\mathbf{i}\) to a Dyck path \((s_1,s_2, \hspace{0.05cm} ... \hspace{0.05cm} ,s_{2p}),\) where \(s_k=+1\) if \((i_{k-1},i_k)\) is marked, and \(s_k=-1\) if \((i_{k-1},i_k)\) is unmarked: the number of such paths is 
\[C_p=\frac{1}{p+1}\binom{2p}{p}.\]
\par
\underline{Step \(2.\)} Once the positions of the marked edges in \(\mathbf{i}\) are chosen (i.e., a Dyck path), establish the order of their marked vertices. The number of possibilities is
\[\frac{p!}{\prod_{k \geq 1}{(k!)^{n_k}}} \cdot \frac{1}{\prod_{k \geq 1}{n_k!}}.\]
\par
\underline{Step \(3.\)} Select the vertices appearing in \(\mathbf{i},\)
\[V(\mathbf{i}):=\{i_k,0 \leq k \leq 2p\},\] 
one at a time, by reading the edges of \(\mathbf{i}\) in ascending order. %starting with the first, \((i_0,i_1).\)
\par
\underline{Step \(4.\)} Choose the remaining vertices of \(\mathbf{i}\) from \(V(\mathbf{i}),\) by reading anew the edges of \(\mathbf{i}\) in order, beginning at \((i_0,i_1)\) (step \(3\) only established the first appearance of each element of \(V(\mathbf{i})\) in \(\mathbf{i}\)). Solely the right ends of the unmarked edges are yet to be decided: \((i_0,i_1)\) is fixed as \(i_0, i_1\) have already been chosen (\(i_1\) is marked); by induction, any subsequent edge has its left end fixed, and therefore only its right end must be selected. This yields marked edges are fully labeled: step \(2\) determines their positions, while step \(3\) appoints their right endpoints. The overall bound for this stage remains as in \cite{sinaisosh},
\begin{equation}\label{bdstep444}
    \prod_{k \geq 2}{(2k)^{kn_k}},
\end{equation}
(see, for instance, lemma \(1\) in \cite{oldpaper} for a complete proof).
\par
\underline{Step \(5.\)} Bound the expectation generated by \(\mathbf{i}.\)
\vspace{0.3cm}
\par
Regarding the recursion referred to earlier, this concerns the even cycles giving the first-order term in (\ref{genexp}): keeping the notation from \cite{oldpaper}, let \(\mathcal{C}(l)\) be the set of pairwise non-isomorphic even cycles of length \(2l,\) with \(n_1=l,\) and the first vertex unmarked (call two cycles \(\mathbf{i},\mathbf{j}\) of length \(q\) \textit{isomorphic} if \(i_s=i_t \Longleftrightarrow j_s=j_t\) for all \(0 \leq s,t \leq q\)). This collection of sets has two key properties:
\vspace{0.1cm}
\par
\(1. \hspace{0.1cm} |\mathcal{C}(l)|=C_l,\hspace{0.5cm}\) 
\par
\(2.\) a recursive description holds: \(\mathcal{C}(l+1)\) consists of three pairwise disjoint families, 
\par
\((i)\) \(\mathbf{i}=(v_0,v_1, \hspace{0.05cm} ... \hspace{0.05cm}, v_{2l-1},v_0) \in \mathcal{C}(l)\) with a loop at \(v_0: (v_0,u,v_0,v_1,\hspace{0.05cm} ... \hspace{0.05cm}, v_{2l-1},v_0)\) and \(u\) %\(\not \in \{v_i: 0 \leq i \leq 2l-1\}\) 
new (i.e., not among the vertices of \(\mathbf{i}\));
\par
\((ii)\) \(\mathbf{i}=(v_0,v_1, \hspace{0.05cm} ... \hspace{0.05cm}, v_{2l-1},v_0) \in \mathcal{C}(l)\) with a loop at \(v_1: (v_0,v_1,u,v_1,v_2,\hspace{0.05cm} ... \hspace{0.05cm}, v_{2l-1},v_0)\) and \(u\) new;
\par
\((iii)\) \((u_0,u_1,u_2,S_1,u_2,u_1,S_2,u_0)\) with \((u_2,S_1,u_2) \in \mathcal{C}(a), (u_0,u_1,S_2,u_0) \in \mathcal{C}(l-a),\) no vertex appearing in both, \(u_0,u_1,u_2\) pairwise distinct, and \(1 \leq a \leq l-1.\) 
\vspace{0.1cm}
\par
One straightforward consequence of \(1.\) is essential when justifying Theorem~\ref{th4}:
\begin{center}
    the elements of \(\mathcal{C}(l)\) are unions of cycles of the form \((i_0,\mathcal{L},i_0),\) 
\end{center}
\begin{center}
    where \(\mathcal{L} \in \cup_{q \leq l-1}{\mathcal{C}(q)},\) any two sharing no vertex but \(i_0\) and \(i_0 \not \in \mathcal{L}\)
\end{center} 
(these elements belong to \(\mathcal{C}(l)\) as long as their lengths add up to \(2l\) in light of the definition of  \(\mathcal{C}(l),\) and their number is
\begin{equation}\label{easycatalan}
    \sum_{1 \leq k \leq l}{\sum_{
l_1+...+l_k=l-k,l_i \geq 0}{C_{l_1}C_{l_2}...C_{l_k}}}=C_l:
\end{equation}
this follows from \(C_{l+1}=\sum_{0 \leq k \leq l}{C_kC_{l-k}},\) and induction on \(l\) after rewriting the left-hand side term as
\[\sum_{1 \leq k \leq l}{\sum_{l_1+...+l_k=l-k,l_i \geq 0}{C_{l_1}C_{l_2}...C_{l_k}}}=C_{l-1}+\sum_{0 \leq l_1 \leq l-2}{C_{l_1}\sum_{2 \leq k \leq l-l_1}{\sum_{l_2+...+l_k=l-k-l_1}{C_{l_2}...C_{l_k}}}},\]
the first term corresponding to \(k=1).\) This recursive description yields that for \(p,m \in \mathbb{N},\)
\begin{equation}\label{keyeq}
    \sum_{l_1+...+l_m=p, l_i \geq 0}{C_{l_1}C_{l_2}...C_{l_m}}=f_{m+p,m+1}=\binom{m+2p-1}{p}-\binom{m+2p-1}{p-1}=\frac{m}{m+2p}\binom{m+2p}{p},
\end{equation}
where for \(2 \leq k \leq l+1,\) \(f_{l,k}\) is the number of elements of \(\mathcal{C}(l)\) in which the first vertex appears exactly \(k-1\) times: note the sum is the number of elements of \(\mathcal{C}(p+m)\) with the first vertex appearing exactly \(m\) times, which by definition in \cite{oldpaper} is \(f_{m+p,m+1},\) computed in lemma \(4\) therein. The numbers \(f_{l,k}\) are essential for the variance of the diagonal entries as it will become apparent in the following subsection, while property \(2.\) is critical in the proof of Theorem~\ref{th1} (subsection~\ref{4sect.1}) since it allows counting the dominant cycles in (\ref{trcov}) by induction on \(p.\)
\par
In the rest of this section,
\begin{itemize}
    \item \ref{treigen1} presents the proof of (\ref{moments}) for \(l \leq 2;\)

    \item \ref{treigen2} introduces lemmas on which computing high moments (\(l>2\)) relies;
    
    \item \ref{treigen22} completes the justification of (\ref{moments}) by considering \(l>2;\)
    
    \item \ref{treigen3} argues (\ref{moments2}).
\end{itemize}

\subsection{Diagonal Entries}\label{treigen1}

The goal is (\ref{moments}) for \(l \leq 2:\) the parity of the power of \(A\) in this case is irrelevant. Suppose without loss of generality that \(i=1:\) \(l=1\) is clear by linearity of expectation, and for \(l=2,\)
\begin{equation}\label{startingpoint}
    \mathbb{E}[(e_1^TA^{p}e_1-\frac{1}{n}\mathbb{E}[tr(A^{p})])^2]=n^{-p}\sum_{(\mathbf{i},\mathbf{j}): \mathbf{i},\mathbf{j} \in \mathcal{S}(p,1)}{(\mathbb{E}[a_{\mathbf{i}} \cdot a_{\mathbf{j}}]-\mathbb{E}[a_{\mathbf{i}}] \cdot \mathbb{E}[a_{\mathbf{j}}])},
\end{equation}
where \(\mathcal{S}(q,1):=\{\mathbf{i}=(1,i_1,i_2, \hspace{0.05cm} ... \hspace{0.05cm}, i_{q-1},1), i_1, i_2, \hspace{0.05cm} ... \hspace{0.05cm},i_{q-1} \in \{1,2, \hspace{0.05cm}... \hspace{0.05cm},n \}\}.\) By independence, the contribution of a pair \((\mathbf{i},\mathbf{j})\) does not vanish if and only if \(\mathbf{i},\mathbf{j}\) share some edge, with their union being an even cycle of length \(2p:\) all the pairs considered next are assumed to satisfy this property.
\par
Similarly to Sinai and Soshnikov~\cite{sinaisosh}, the key is gluing \(\mathbf{i},\mathbf{j}\) into an even cycle \(\mathcal{P}\) of length \(2p-2.\)  This is done by using the first common undirected edge \(e\) and cutting two of its copies: specifically, let \(i_{t-1}i_{t}=j_{s-1}j_{s}\) with \(t,s\) minimal in this order, 
\[t=\min{\{1 \leq k \leq p, \exists 1 \leq q \leq p, i_{k-1}i_{k}=j_{q-1}j_{q}\}}, \hspace{0.4cm} s=\min{\{1 \leq q \leq p, j_{q-1}j_{q}=i_{t-1}i_{t}\}}.\] 
\(\mathcal{P}\) is obtained by fusing both cycles along this common edge, whose copies are then erased: it traverses \(\mathbf{i}\) up to \(i_{t-1}i_{t},\) uses it as a bridge to switch to \(\mathbf{j},\) traverse all of it, and get back to the rest of \(\mathbf{i}\) upon returning to \(j_{s-1}j_s=i_{t-1}i_{t}:\) specifically, if \((i_{t-1},i_{t})=(j_{s},j_{s-1}),\) then
\begin{equation}\label{paste1}
    \mathcal{P}:=(i_0, \hspace{0.05cm} ... \hspace{0.05cm}, i_{t-2},i_{t-1},j_{s+1}, \hspace{0.05cm} ... \hspace{0.05cm}, j_{p-1},j_0, \hspace{0.05cm} ... \hspace{0.05cm}, j_{s-1},i_{t+1}, \hspace{0.05cm} ... \hspace{0.05cm}, i_{p});
\end{equation}
else \((i_{t-1},i_{t})=(j_{s-1},j_{s}),\) and 
\begin{equation}\label{paste2}
     \mathcal{P}:=(i_0, \hspace{0.05cm} ... \hspace{0.05cm},i_{t-2},i_{t-1},j_{s-2}, \hspace{0.05cm} ... \hspace{0.05cm},j_0,j_{p-1}, \hspace{0.05cm} ... \hspace{0.05cm}, j_{s},i_{t+1}, \hspace{0.05cm} ... \hspace{0.05cm},i_{p}).
\end{equation}
\(\mathcal{P}\) is an even cycle of length \(2 \cdot p-2=2p-2,\) and \(e\) has endpoints at distance \(p-1\) in \(\mathcal{P}\) (the undirected edge between the \(t^{th}\) and \((t+p-1)^{th}\) vertices in \(\mathcal{P}\) is \(e\)).
%\(e=\rho_{t-1}\rho_{(t-1)+p-1}\)
%(\(i_{t-1}j_{s-1}\) in (\ref{paste1}), and \(i_{t-1}j_{s}\) in (\ref{paste2})). 
By convention, all graphs \(\mathcal{G}\) in what follows are directed, and \(e=uv \in \mathcal{G}\) is a shorthand for \((u,v) \in E(\mathcal{G}).\)
\par
Take first the leading pairs in (\ref{startingpoint}), namely the preimages of their merges: since the dominant configurations (i.e., yielding the largest contribution in expectation) among even cycles of length \(2l\) are the elements of \(\mathcal{C}(l),\) this remains true here as well via a similar rationale to the one for the trace (this is detailed at the end of the subsection: intuitively this should hold because the cycles resulting from these merges simply have the first vertex equal to \(1\)); in the current situation, there is an additional factor of \(n^{-1}\) in step \(3,\) accounting for the set vertex \(1,\) which is no longer to be chosen. Consider now the preimage of the gluing given by (\ref{paste1}), (\ref{paste2}) for \(\rho=(\rho_{0k})_{0 \leq k \leq p-1} \in \mathcal{C}(p-1)\) with \(\rho_{00}=1.\) Suppose \(\rho,\) under description above (\ref{easycatalan}), has \(k\) cycles, denoted by \(\mathcal{L}_1, \mathcal{L}_2, \hspace{0.05cm} ... \hspace{0.05cm}, \mathcal{L}_k,\) of lengths \(2l_1,2l_2, \hspace{0.05cm} ... \hspace{0.05cm},2l_k\) and with first vertices \(i_1,i_2, \hspace{0.05cm} ... \hspace{0.05cm},i_k,\) respectively; continue denoting by \(i_0\) the first vertex of \(\rho\) (\(i_0=1\)) for the sake of simplicity. The condition on the lengths is 
\begin{equation}\label{leqq22}
    \sum_{1 \leq j \leq k}{(l_j+1)}=p-1.
\end{equation}
For \((\mathbf{i},\mathbf{j})\) in the preimage of \(\rho,\) denote by \(r\) the position with \(e:=\rho_{0r} \rho_{0(r+p-1)}\) the first (in the sense aforesaid) shared edge between \(\mathbf{i}\) and \(\mathbf{j}\) %where \(0 \leq r \leq p-1\) 
with \(q_1,q_2 \in \{1,2,\hspace{0.05cm}...\hspace{0.05cm},k\}\) such that 
\[r \in [2\sum_{1 \leq j<q_1}{(l_j+1)},2\sum_{1 \leq j \leq q_1}{(l_j+1)}), \hspace{0.5cm} p-1+r \in [2\sum_{1 \leq j<q_2}{(l_j+1)},2\sum_{1 \leq j \leq q_2}{(l_j+1)}].\] 
Then \(\rho_{0(r+1)}\) is the first vertex in \(\mathcal{L}_{q_1},\) i.e., \(r=L=2\sum_{1 \leq j<q_1}{l_j}\) since otherwise the path \((\rho_{0L},\rho_{0(L+1)},\hspace{0.05cm}...\hspace{0.05cm},\rho_{0r})\) is not empty and thus must share an edge with \((\rho_{0r},\rho_{0(r+1)},\hspace{0.05cm}...\hspace{0.05cm},\rho_{0(L+2l_{q_1})}),\) contradicting the minimality of \(r\) (the path has first vertex \(i_0\) and is contained in \((i_0,\mathcal{L}_{q_1},i_0) \in \mathcal{C}(l_{q_1}+1),\) whereby it can only be a cycle if it is \((i_0,\mathcal{L}_{q_1},i_0),\) in which case \(r=L+2l_{q_1},\) contradicting the definition of \(q_1;\) hence the path is not a cycle and must share an edge with the remainder of the cycle, \((\rho_{0r},\rho_{0(r+1)},\hspace{0.05cm}...\hspace{0.05cm},\rho_{0(L+2l_{q_1})})\)). This entails that the pairs in the preimage of interest can be described by tuples \((\rho_1,t_1,\rho_2,t_2,\rho_3,o)\) with 
\[\rho_1 \in \mathcal{C}(a),\rho_2 \in \mathcal{C}(b),\rho_3 \in \mathcal{C}(p-1-a-b), \hspace{0.5cm} 0 \leq a,b \leq \frac{p-1}{2}, \hspace{0.5cm} o \in \{0,1\},\]
\[t_1\in \{j: 0 \leq j \leq 2a-1, \rho_{1t_1}=i_0\},\hspace{0.5cm} t_2 \in \{j: 0 \leq j \leq 2b-1, \rho_{2t_2}=i_0\},\] \(\rho_1:=(\rho_{10},\rho_{11},\hspace{0.05cm}...\hspace{0.05cm},\rho_{1(2a-1)},\rho_{10}),\rho_2:=(\rho_{20},\rho_{21},\hspace{0.05cm}...\hspace{0.05cm},\rho_{2(2b-1)},\rho_{20}),\rho_3:=(\rho_{30},\rho_{31},\hspace{0.05cm}...\hspace{0.05cm},\rho_{3(2p-2-a-b-1)},\rho_{20}),\) \(\rho_{10}=\rho_{20}=\rho_{30}=i_0,\) \(i_0 \not \in \{\rho_{3j}, 1 \leq j \leq 2p-a-b-3\},\) and \(\rho_{k_1j_1}=\rho_{k_2j2}\) for \(k_1 \ne k_2\) entailing \(\rho_{k_1j_1}=\rho_{k_2j2}=i_0.\)
The mapping underlying this takes \((\mathbf{\rho}_1,t_1,\mathbf{\rho}_2,t_2,\mathbf{\rho}_3,o)\) to \((\mathbf{i},\mathbf{j}),\) where \(\mathbf{i}\) is \(\rho_1\) with a cycle attached at position \(t_1,\) given by an edge at it connected to the last \(p-1-2a\) edges of \(\rho_3,\) while \(\mathbf{j}\) is \(\rho_2\) with a cycle attached at position \(t_2,\) given by an edge at it connected to the first \(p-1-2b\) edges of \(\rho_3,\) with \(o\) yielding the orientation of the latter cycle. Specifically,   
\[\mathbf{i}=(\rho_{10},\rho_{11},\hspace{0.05cm}...\hspace{0.05cm},\rho_{1(t_1-1)},\rho_{1t_1},\rho_{3(p-1-2b)},\rho_{3(p-2b)},\hspace{0.05cm}...\hspace{0.05cm},\rho_{3(2p-2-2a-2b-1)},\rho_{30},\rho_{1(t_1+1)},\hspace{0.05cm}...\hspace{0.05cm},\rho_{1(2a-1)},\rho_{10}),\]
\[\mathbf{j}=\begin{cases}
    (\rho_{20},\rho_{21},\hspace{0.05cm}...\hspace{0.05cm},\rho_{2(t_2-1)},\rho_{2t_2},\rho_{31},\rho_{32},\hspace{0.05cm}...\hspace{0.05cm},\rho_{3(p-1-2b)},\rho_{2t_2},\rho_{2(t_2+1)},\hspace{0.05cm}...\hspace{0.05cm},\rho_{2(2b-1)},\rho_{20}), \hspace{1.5cm} o=0,\\
    (\rho_{20},\rho_{21},\hspace{0.05cm}...\hspace{0.05cm},\rho_{2(t_2-1)},\rho_{2t_2},\rho_{3(p-1-2b)},\rho_{3(p-2-2b)},\hspace{0.05cm}...\hspace{0.05cm},\rho_{31},\rho_{2t_2},\rho_{2(t_2+1)},\hspace{0.05cm}...\hspace{0.05cm},\rho_{2(2b-1)},\rho_{20}), \hspace{0.4cm} o=1.
\end{cases}\]
This mapping is injective because \(o, t_1,t_2,a,b,\) hence \(\rho_1,\rho_2,\rho_3,\) are functions of \((\mathbf{i},\mathbf{j}):=((u_k)_{0 \leq k \leq p},(v_k)_{0 \leq k \leq p}),\)
\[t_1=\max\{2k: (u_0,u_1,\hspace{0.05cm}...\hspace{0.05cm},u_{2k-1},u_{2k}) \in \mathcal{C}(k)\},\]
\[p-1-2a=\max{\{p-1-2d: i_0 \not \in \{u_{y},t_1+1 \leq y \leq t_1+p-1-2d\}\}},\] 
\[t_2=\max\{2k: (v_0,v_1,\hspace{0.05cm}...\hspace{0.05cm},v_{2k-1},v_{2k}) \in \mathcal{C}(k)\},\]
\[p-1-2b=\max{\{p-1-2d: i_0 \not \in \{v_{y},t_2+1 \leq y \leq t_2+p-1-2d\}\}},\] 
\[o=\chi_{(u_{t_1},u_{t_1+1}) \in \{(v_t,v_{t+1}),0 \leq t \leq p-1\}},\]
as well as surjective: under the notation below (\ref{leqq22}), for \(L_2=2\sum_{1 \leq j \leq q_2}{(l_j+1)},\)
\[t_1=r,\hspace{0.2cm} \rho_1=(\rho_{00},\rho_{01},\hspace{0.05cm}...\hspace{0.05cm},\rho_{0r},\rho_{0L_2},\hspace{0.05cm}...\hspace{0.05cm},\rho_{0(2p-3)},\rho_{0(2p-2)}),\] 
\[t_2=\max\{2k: (v_0,v_1,\hspace{0.05cm}...\hspace{0.05cm},v_{2k-1},v_{2k}) \in \mathcal{C}(k)\}, \hspace{0.2cm} \rho_2=(v_0,v_1,\hspace{0.05cm}...\hspace{0.05cm},v_{t_2},v_{t_3},v_{t_3+1}\hspace{0.05cm}...\hspace{0.05cm},v_{p}),\]
\[\hspace{0.5cm}o=\chi_{(\rho_{0r},\rho_{0(r+p-1)}) \in \{(v_t,v_{t+1}),0 \leq t \leq p-1\}}, \hspace{0.5cm} \rho_3=\begin{cases}
    (v_{t_2},v_{t_2+1},\hspace{0.05cm}...\hspace{0.05cm},v_{t_3-1},v_{t_3},\rho_{0(r+p-1)},\rho_{0(r+p)},\hspace{0.05cm}...\hspace{0.05cm},\rho_{0L_2}), \hspace{0.1cm} o=0,\\
    (v_{t_3},v_{t_3-1},\hspace{0.05cm}...\hspace{0.05cm},v_{t_2+1},v_{t_2},\rho_{0(r+p-1)},\rho_{0(r+p)},\hspace{0.05cm}...\hspace{0.05cm},\rho_{0L_2}), \hspace{0.1cm} o=1,
\end{cases}\]
where 
\[\mathbf{j}=(v_t)_{0 \leq t \leq p}, \hspace{0.5cm} t_3=1+\max\{t: t \geq t_2, i_0 \not \in \{v_{y},t_2+1 \leq y \leq t\}\}\]
(\(\rho_3\) equals, up to orientation, to \((i_0,\mathcal{L}_{q_1},i_0),\) and \(\rho_1,\rho_2,\rho_3\) are disjoint unions of pairwise distinct cycles among \(((i_0,\mathcal{L}_{q},i_0))_{1 \leq q \leq k}:\) this entails \((\rho_1,t_1,\rho_2,t_2,o)\) belongs to the domain of the aforesaid mapping).
\par
Consider now the contributions of the preimages, \(\mathbb{E}[a_{\mathbf{i}} \cdot a_{\mathbf{j}}]-\mathbb{E}[a_{\mathbf{i}}] \cdot \mathbb{E}[a_{\mathbf{j}}],\) which depends on whether \(e \in \rho\) or \(e \not \in \rho.\)
%the first term is either \(\mathbb{E}[a^4_{11}\chi_{|a_{11}| \leq n^{\delta}}]\) or \(\mathbb{E}[a^2_{11}\chi_{|a_{11}| \leq n^{\delta}}],\) depending on whether \(e \in \rho\) or \(e \not \in \rho.\) 
Under the tuple notation above, the former occurs exactly when \(\rho_{1t_1}\rho_{3(p-1-2b)}\) appears in \(\rho_3,\) which is equivalent to the position \(p-1-2b\) being the second or second to last in \(\rho_3\) insomuch as \(\rho_3\) contains solely one copy of \(i_0=\rho_{1t_1}.\) This gives that the contribution of the leading pairs is
\[n^{p-1} \cdot (1+O(\frac{p^2}{n})) \cdot (1+O(pn^{-2\delta})) \cdot  [\mathcal{A}_1(p)+(\mathbb{E}[a^4_{11}]-2)\mathcal{A}_2(p)],\]
where \(\mathcal{A}_1(p)\) is the size of the preimage modulo vertex isomorphisms, and \(\mathcal{A}_2(p)\) is the number of elements in it with \(e \in \rho:\) the first two factors come from choosing the vertices of the cycles among \(\{1,2,\hspace{0.05cm}...\hspace{0.05cm},n\},\) one of them being \(1,\) while the last two are due to
\[\mathbb{E}[a_{\mathbf{i}} \cdot a_{\mathbf{j}}]-\mathbb{E}[a_{\mathbf{i}}] \cdot \mathbb{E}[a_{\mathbf{j}}]=\begin{cases}
    (\mathbb{E}[a^{2}_{11}\chi_{|a_{11}| \leq n^{\delta}}])^{p}, \hspace{6.8cm} e \not \in \rho,\\
    (\mathbb{E}[a^{2}_{11}\chi_{|a_{11}| \leq n^{\delta}}])^{p-2} \cdot [\mathbb{E}[a^{4}_{11}\chi_{|a_{11}| \leq n^{\delta}}]-(\mathbb{E}[a^{2}_{11}\chi_{|a_{11}| \leq n^{\delta}}])^2],  \hspace{0.6cm} e \in \rho,
\end{cases}\]
from \(e \not \in \rho\) entailing \(\mathbb{E}[a_{\mathbf{i}}] \cdot \mathbb{E}[a_{\mathbf{j}}]=0, \mathbb{E}[a_{\mathbf{i}} \cdot a_{\mathbf{j}}]=(\mathbb{E}[a^{2}_{11}\chi_{|a_{11}| \leq n^{\delta}}])^{p}\) (\(e\) has multiplicity \(1\) in both \(\mathbf{i},\mathbf{j}\)), and otherwise \(e \in \rho,\)  
\(\mathbb{E}[a_{\mathbf{i}}]=\mathbb{E}[a_{\mathbf{j}}]=(\mathbb{E}[a^{2}_{11}\chi_{|a_{11}| \leq n^{\delta}}])^{p/2},\mathbb{E}[a_{\mathbf{i}} \cdot a_{\mathbf{j}}]=(\mathbb{E}[a^{2}_{11}\chi_{|a_{11}| \leq n^{\delta}}])^{p-2} \cdot \mathbb{E}[a^{4}_{11}\chi_{|a_{11}| \leq n^{\delta}}],\) as well as 
\[\mathbb{E}[a^{2}_{11}\chi_{|a_{11}| \leq n^{\delta}}] \in [1-n^{-2\delta},1], \hspace{0.6cm} \frac{\mathbb{E}[a^{4}_{11}\chi_{|a_{11}| \leq n^{\delta}}]}{\mathbb{E}[a^{4}_{11}]}=1-\frac{\mathbb{E}[a^{4}_{11}\chi_{|a_{11}|>n^{\delta}}]}{\mathbb{E}[a^{4}_{11}]} \in [1-C(\epsilon_0)n^{-(4+\epsilon_0)\delta},1].\]
\par
The tuple description above yields
\begin{equation}\label{ref1}
    \mathcal{A}_1(p)=2\sum_{0 \leq a,b \leq \frac{p-1}{2}}{s(a)s(b)C_{p-2-a-b}}, \hspace{0.2cm} \mathcal{A}_2(p)=2\chi_{2|p} \cdot [2s(\frac{p}{2}-1)\sum_{0 \leq a \leq \frac{p}{2}-1}{s(a)C_{\frac{p}{2}-1-a}-s^2(\frac{p}{2}-1)}],
\end{equation}
where \(f_{m,t+1}\) is defined by (\ref{keyeq}) and 
\begin{equation}\label{formm}
    s(0)=1, \hspace{0.5cm} s(m)=\sum_{1 \leq t \leq m}{tf_{m,t+1}} \hspace{0.4cm}  m \geq 1
\end{equation}
since for fixed \(a,b,\) there are \(s(a),s(b)\) possibilities for \((\rho_1,t_1),\) and \((\rho_2,t_2),\) respectively, two values for \(o,\) and \(C_{p-2-a-b}\) possibilities for \(\rho_3\) (it has length \(2(p-1-a-b)\) and its first vertex appears solely once in it), and additionally \(p-1-2b \in \{1,2p-2-2a-2b-1\}\) in the configurations underlying \(\mathcal{A}_2(p),\) equivalent to \(a=\frac{p}{2}-1\) or \(b=\frac{p}{2}-1.\) Lemma~\ref{lems} in the Appendix encapsulates how these exact formulae for \(\mathcal{A}_1(p),\mathcal{A}_2(p)\) also render their asymptotic behavior as \(p \to \infty,\) claimed in the statement of the theorem.
\par
This concludes the analysis of \(l=2\) since the merges of length \(2p-2\) that do not belong to \(\mathcal{C}(p-1)\) are negligible via Lemma~\ref{bigcycles1}: this corresponds to the second term in the upper bound, and since in this situation \(\max_{\mathcal{L}'}{V(X,\mathcal{L}'_{0})} \leq \mathbb{E}[a^4_{11}]\) as there are only two paths, these cycles yield overall \(O(p^{9}n^{-2\delta_1}),\) after normalizing by \(\frac{n^{-(p-1)}}{c(p,\mathbb{E}[a^4_{11}])}.\)% and the upper bound is tight when the merges are in \(\mathcal{C}(p-1)\) due to 
%\[\mathbb{E}[a^{2}_{11}\chi_{|a_{11}| \leq n^{\delta}}] \in [1-n^{2\delta},1].\]
% besides it completes the justification of the claims made in Theorem~\ref{th4} on \(c(p,\mathbb{E}[a^4_{11}]).\) 

\subsection{Patched Paths}\label{treigen2}

Consider \(l>2\) in (\ref{moments}), and suppose again without loss of generality that \(i=1,\) similarly to (\ref{startingpoint}):
\begin{equation}\label{momi}
    \mathbb{E}[(e_1^TA^{p}e_1-\frac{1}{n}\mathbb{E}[tr(A^{p})])^l]=n^{-pl/2}\sum_{(\mathbf{i}_1,\mathbf{i}_2, \hspace{0.05cm} ... \hspace{0.05cm},\mathbf{i}_l)}{\mathbb{E}[(a_{\mathbf{i}_1}-\mathbb{E}[a_{\mathbf{i}_1}]) \cdot (a_{\mathbf{i}_2}-\mathbb{E}[a_{\mathbf{i}_2}]) \cdot ... \cdot (a_{\mathbf{i}_l}-\mathbb{E}[a_{\mathbf{i}_l}])]},
\end{equation}
where \(\mathbf{i}_1,\mathbf{i}_2, \hspace{0.05cm} ... \hspace{0.05cm},\mathbf{i}_l \in \mathcal{S}(p,1).\) For each summand, construct a simple undirected graph \(\mathcal{G}\) with vertices \((1,\mathbf{i}_1),(2,\mathbf{i}_2), \hspace{0.05cm} ... \hspace{0.05cm},(l,\mathbf{i}_l),\) and \((a,\mathbf{i}_a)(b,\mathbf{i}_b) \in E(\mathcal{G})\) if only if \(\mathbf{i}_a, \mathbf{i}_b\) share an edge: the nonzero contributions in (\ref{momi}) are from tuples for which all connected components of \(\mathcal{G}\) have size at least two (else the expectation vanishes by independence). The key observation in \cite{sinaisosh} is that solely graphs containing \(l/2\) components are first-order terms in (\ref{momi}): each must be of size \(2,\) and this generates \((l-1)!!,\) the number of partitions of \(\{1,2, \hspace{0.05cm} ... \hspace{0.05cm}, l\}\) in \(l/2\) unordered pairs. Given the additional factor of \(n^{1/2}\) in the normalization ((\ref{cltsinaisosh}) does not contain any additional power of \(n\)), the gluing from \cite{sinaisosh} is not tight enough: Lemma~\ref{stilldom} below treats a more general case that allows controlling the moments of both \(e_i^TA^pe_i\) and \(e_i^TA^pe_j.\) There are several reasons for considering a wider class of configurations than those underlying (\ref{momi}): the moments of \(e_i^TA^pe_j\) are also needed, induction is more amenable in this enlarged universe than it is in the original one, and this generalization can be adapted to handle mixed moments of off-diagonal entries, i.e., \(\mathbb{E}[e_{i_1}^TA^pe_{j_1}\cdot e_{i_2}^TA^pe_{j_2} \cdot ... \cdot e_{i_L}^TA^pe_{j_L}]\) with \(i_k \ne j_k\) for all \(1 \leq k \leq L,\) 
relevant in forthcoming sections (this is the content of Lemma~\ref{bigcycles}, the tool bounding many expectations to come). Another extension, Lemma~\ref{bigcycles1}, addresses the aforementioned situation when some entries are diagonal: this latter scenario differs from the former primarily because diagonal entries can be not centered (\(\mathbb{E}[e_i^TA^{2p+1}e_i]=0,\) whereas \(\mathbb{E}[e_i^TA^{2p}e_i]=\frac{1}{n}\mathbb{E}[tr(A^{2p})]\)). Before proceeding with these results, additional terminology and notation are necessary.
\par
For \(i_1,i_2, \hspace{0.05cm} ... \hspace{0.05cm},i_L, j_1,j_2, \hspace{0.05cm} ... \hspace{0.05cm}, j_L \in \{1,2, \hspace{0.05cm} ... \hspace{0.05cm},n\},\) call the tuple of edges \((i_kj_k)_{1 \leq k \leq L}\) \textit{even} if
\begin{equation}\label{evendef}
    |\{k:1 \leq k \leq L, i_k=v\}|+|\{k:1 \leq k \leq L, j_k=v\}| \in 2\mathbb{Z}
\end{equation}
for all \(1 \leq v \leq n\) (i.e., \(v\) appears an even number of times among \(i_1,i_2, \hspace{0.05cm} ... \hspace{0.05cm},i_L, j_1,j_2, \hspace{0.05cm} ... \hspace{0.05cm}, j_L\)). The contributors to the forthcoming expectations of interest are even tuples, and to understand which dominate, another combinatorial object is needed. 
\par
For \(l_1,l_2, \hspace{0.05cm} ... \hspace{0.05cm}, l_L \in \mathbb{N}\) and \((i_kj_k)_{1 \leq k \leq L}\) an even tuple, let
\begin{equation}\label{mdef}
    \mathcal{M}(((i_kj_k,l_k))_{1 \leq k \leq L})=\begin{cases}
        \sum_{\{\mathcal{L}_0,\mathcal{L}_1, \hspace{0.05cm} ... \hspace{0.05cm},\mathcal{L}_t\}}{\prod_{1 \leq r \leq t}{\mathcal{D}(\mathcal{Q}(\mathcal{L}_r))}}, \hspace{1.8cm} \max_{1 \leq r \leq L}{l_r}>1,\\
        1, \hspace{6.7cm} l_1=l_2=...=l_L=1,
    \end{cases} 
\end{equation}
where the summation is over sets \(\{\mathcal{L}_0, \mathcal{L}_1,\mathcal{L}_2, \hspace{0.05cm} ... \hspace{0.05cm},\mathcal{L}_t\}\) with the property that for each \(q,1 \leq q \leq t,\)
\par
\((i) \hspace{0.1cm} \mathcal{L}_q=(k_{q1},k_{q2}, \hspace{0.05cm}... \hspace{0.05cm},k_{qT_q}),k_{q1}=\min_{1 \leq r \leq T_q}{k_{qr}}, k_{qr} \in \{1,2, \hspace{0.05cm} ... \hspace{0.05cm},L\}\) for \(1 \leq r \leq T_q,\) 
\par
\((ii) \hspace{0.1cm} i_{k_{qr}}j_{k_{qr}}=v_{qr}v_{q(r+1)}\) (as undirected edges) for \(1 \leq r \leq T_q,\) and a set \(\{v_{q1},v_{q2}, \hspace{0.05cm} ... \hspace{0.05cm},v_{qT_q}\} \subset \{1,2, \hspace{0.05cm} ... \hspace{0.05cm},n\}\) with \(v_{q(T_q+1)}:=v_{q1},\)
\par
\((iii) \hspace{0.1cm} \mathcal{Q}(\mathcal{L}_q)=(\{(l_{k_{q1}}+...+l_{k_{q(r-1)}},v_{qr}), 1 \leq r \leq T_q\},\sum_{1 \leq r \leq T_q}{l_{k_{qr}}})\) (the sum is \(0\) when \(r=1), \newline\)
with \(\mathcal{D}(S,q)\) the number of Dyck paths of length \(2q\) to which \(\mathbf{i} \in \mathcal{C}(q)\) with \(i_s=v_s\) for \((s,v_s) \in S\) are mapped in step \(1\) below (\ref{tuplecond}) (\(S \subset  \{0,1, \hspace{0.05cm} ... \hspace{0.05cm},2q-1\} \times \{1,2, \hspace{0.05cm} .... \hspace{0.05cm},n\}\)), and
\par
\((iv) \hspace{0.1cm} \mathcal{L}_0=\{1 \leq k \leq L: i_k=j_k\} \cup (\cup_{1 \leq k \leq L}{\mathcal{S}(i_kj_k)}),\) where \(\mathcal{S}(uv)\) is the set of the smallest \(2 \cdot l(uv)\) elements of \(\mathcal{S}'(uv)=\{1 \leq k \leq L: i_kj_k=uv,l_k=1\},l(uv):=\lfloor \frac{|\mathcal{S}'(uv)|}{2} \rfloor,\)
\par
\((v)\) for each \(k \in \{1,2,\hspace{0.05cm} ... \hspace{0.05cm},L\}-\mathcal{L}_0,\) there exist unique \(s,r\) with \(1 \leq s \leq t, 1 \leq r \leq T_s, k=k_{sr}.\)
\par
More succinctly, in the non-degenerate scenario \(\min_{1 \leq k \leq L}{l_k}>1,\) \(\mathcal{M}\) counts the configurations of Dyck paths for a collection of cycles belonging to \(\cup_{q \geq 1}{\mathcal{C}(q)},\) each formed by patching paths with endpoints \(i_k,j_k\) and of length \(l_k\) for \(1 \leq k \leq L:\) the summands in it are the weights of the leading terms in 
\[n^{-(\sum_{1 \leq k \leq L}{l_k}-L')/2}\sum_{(\mathbf{i}_k)_{1 \leq k \leq L} \in \mathcal{P}}{\mathbb{E}[\prod_{1 \leq k \leq L}{a_{\mathbf{i}_k}}]},\]
where \(L'=|\{k:1 \leq k \leq L, u_k \ne v_k\}|,\) and \(\mathcal{P}\) consists of tuples of paths 
\[((u_k,i_{k1},\hspace{0.05cm} ... \hspace{0.05cm}, i_{k(l_k-1)},v_k))_{1 \leq k \leq L}:=(\mathbf{i}_k)_{1 \leq k \leq L}, \hspace{0.1cm} i_{kh} \in \{1,2, \hspace{0.05cm} ... \hspace{0.05cm},n\}\]
(this is the content of Lemma~\ref{bigcycles}; although it consists solely of an upper bound, it can be easily shown it is tight when \(\mathcal{M}(((i_kj_k,l_k))_{1 \leq k \leq L})>1:\) see end of subsection~\ref{s1}; lastly, note that due to symmetry, all such expectations vanish unless the corresponding tuple is even). These weights arise in step \(1\) of the counting procedure leading to the change of summation in (\ref{genexp}): what justifies the definition of \(\mathcal{M}\) is the description above identity (\ref{easycatalan}), entailing any loop in an element of \(\mathcal{C}(l)\) belongs to \(\cup_{1 \leq k \leq l}{\mathcal{C}(k)}\) (if \((i_0,i_1,\hspace{0.05cm} ... \hspace{0.05cm},i_{2l-1},i_0) \in \mathcal{C}(l)\) and \(i_a=i_b\) for \(a<b,\) then \((i_a,i_{a+1},\hspace{0.05cm} ... \hspace{0.05cm}, i_{b}) \in \mathcal{C}(\frac{b-a}{2}):\) this occurs because \(\mathcal{C}(l)\) is invariant under shifts in \(\mathbb{Z}/2l\mathbb{Z},\) 
\begin{equation}\label{shiftt}
    (i_0,i_1,\hspace{0.05cm} ... \hspace{0.05cm},i_{2l-1},i_0) \in \mathcal{C}(l) \Leftrightarrow (i_a,i_{a+1},\hspace{0.05cm} ... \hspace{0.05cm}, i_{2l-1},i_0, i_1, \hspace{0.05cm} ... \hspace{0.05cm},i_{a-1},i_a) \in \mathcal{C}(l)),
\end{equation}
implying, for instance,
\[\mathcal{D}(S,q)=\prod_{0 \leq i \leq s}{\chi_{2|q_{i+1}-q_i}} \cdot\prod_{0 \leq i \leq s}{C_{\frac{q_{i+1}-q_{i}}{2}}},\] 
when \(S=\{(q_i,v), 1 \leq i \leq s, 0 \leq q_1<q_2<...<q_s \leq 2q-1\},q_{s+1}:=2q+q_1\) (the relevant Dyck paths return at the origin after \(q_2-q_1,q_3-q_1, \hspace{0.05cm} ... \hspace{0.05cm}, q_s-q_1\) steps).
For example, there are \(\mathcal{M}((i_1i_1,2l_1))= C_{l_1}=|\mathcal{C}(l_1)|\) elements in \(\cup_{q \geq 1}{\mathcal{C}(q)}\) of length \(2l_1,\) and \(\mathcal{M}(((i_1i_1,2l_1),(i_1i_1,2l_2)))=C_{l_1}C_{l_2}\) elements in \(\cup_{q \geq 1}{\mathcal{C}(q)}\) of length \(2l_1+2l_2\) with \(i_0=i_{2l_1}\) (i.e., concatenations of one element of \(\mathcal{C}(l_1)\) and one of \(\mathcal{C}(l_2)\)). Certain degeneracy occurs when some of lengths are \(1\) as such elements of \(\mathcal{P}\) yielding nonzero expectations exist, whereas \(\mathcal{M}\) can vanish, e.g., \(i_r=j_r,l_r=1\) (no element of \(\cup_{q \geq 1}{\mathcal{C}(q)}\) has two consecutive vertices equal), and the second branch in (\ref{mdef}) is meant to account for these cases.
\par
Having introduced even tuples and the combinatorial function \(\mathcal{M},\) proceed with justifying the claim above on leading paths. The first stage for this is Lemma~\ref{stilldom}, which deals with the case \(|\{i_k,j_k, 1 \leq k \leq L\}| \leq 2.\)

\begin{lemma}\label{stilldom}
    Suppose \(B=(b_{ij})_{1 \leq i,j \leq n} \in \mathbb{R}^{n \times n}, B=B^T, (b_{ij})_{1 \leq i \leq j \leq n}\) i.i.d. with
    \[b_{11} \overset{d}{=} -b_{11}, \hspace{0.2cm}  \mathbb{E}[b_{11}^2]=1, \hspace{0.2cm} \mathbb{E}[b_{11}^{2q}] \leq L(n)n^{\delta(2q-4)}, \hspace{0.2cm} \mathbb{P}(|b_{11}| \leq n^{\delta})=1\]
    for \(2 \leq q \leq p, \delta>0,\) and \(L:\mathbb{N} \to [1,\infty), L(n)<n^{2\delta}.\) For \(u,v \in \{1,2, \hspace{0.05cm} ... \hspace{0.05cm},n\}, u \ne v,\) an even tuple \((u_kv_k)_{1 \leq k \leq L},u_k,v_k \in \{u,v\},\) \(p,L \in \mathbb{N},p=\sum_{1 \leq k \leq L}{l_i},(l_i)_{1 \leq k \leq L} \subset \mathbb{N},\) let \(\mathcal{P}\) consist of all tuples of paths 
    \(((u_k,i_{k1},\hspace{0.05cm} ... \hspace{0.05cm}, i_{k(l_k-1)},v_k))_{1 \leq k \leq L}:=(\mathbf{i}_k)_{1 \leq k \leq L}, i_{kh} \in \{1,2, \hspace{0.05cm} ... \hspace{0.05cm},n\}.\)  Then for \(p \leq \frac{n^{\delta_1}}{2\sqrt{C\cdot L(n)}},n \geq n(\delta),\)
    \begin{equation}\label{claimmmedd}
    \sum_{(\mathbf{i}_k)_{1 \leq k \leq L} \in \mathcal{P}}{\mathbb{E}[\prod_{1 \leq k \leq L}{b_{\mathbf{i}_k}}]} \leq n^{(p-L+l)/2} \sqrt{\mathbb{E}[b^{2l_0}_{11}]}\cdot [\mathcal{M}(((u_kv_k,l_k))_{1 \leq k \leq L})+C_{\lfloor p/2 \rfloor}L(n)(Cp^2)^{2L-l}\cdot L!\cdot n^{-2\delta_1}],
    \end{equation}
    where \(\delta_1=\frac{1}{4}-\delta, l=|\{k: 1 \leq k \leq L, u_k=v_k\}|, l_0=|\{k: 1 \leq k \leq L, l_k=1\}|,\) and \(C \geq 32\) satisfies (\ref{subas}).
\end{lemma}

\begin{proof}
    Assume \(2|p:\) otherwise the claim is immediate (all expectations vanish). \(L=1\) forces \(u_1=v_1, l=~1, \newline \mathcal{M}((u_1u_1,l_1))=C_{p/2},\) whereby the result follows from (\ref{subas}) (\(l_0=0,\) and one vertex is fixed, furnishing an additional factor of \(n^{-1}\) in step \(3\)). Suppose \(L \geq 2,\) and use induction on \(p \geq L\) to show (\ref{claimmmedd}) holds with the second term encompassing the contribution of all configurations but those %its first term corresponding to configurations 
    in which the paths with indices in
    \begin{equation}\label{sett}
        \{1,2, \hspace{0.05cm} ... \hspace{0.05cm},L\}-(\{k: 1 \leq k \leq L, i_k=j_k, l_k=1\} 
    \cup (\cup_{1 \leq k \leq L}{\mathcal{S}(i_kj_k)}))
    \end{equation}
    (see (\ref{mdef})) are assembled in a collection of edge-disjoint elements of \(\cup_{q \geq 1}{\mathcal{C}(q)},\) sharing vertices solely trivially (i.e., if \(w\) is a vertex belonging to two distinct cycles, then \(w \in \{u,v\},\) and no collection of loops at \(u'\) contains \(v'\) for \(\{u',v'\}=\{u,v\}\)). %when \(\mathcal{M}(((u_kv_k,l_k))_{1 \leq k \leq L})>1.\) \textcolor{red}{UNNEEDED?} %(i.e., arrangements outside this class have an overall contribution at most the second term). %in other words, the contribution of all the other configurations are encompassed by the second term in (\ref{claimmmedd}) in these situations.
    \par
    If \(p=L,\) then \(l_0=p,\)
    \[\sum_{(\mathbf{i}_k)_{1 \leq k \leq L} \in \mathcal{P}}{\mathbb{E}[\prod_{1 \leq k \leq L}{b_{\mathbf{i}_k}}]}=\mathbb{E}[\prod_{1 \leq k \leq p}{b_{u_kv_k}}],\]
    yielding together with Hölder's inequality (\ref{claimmmedd}): 
    \begin{equation}\label{holderr}
        \mathbb{E}[\prod_{1 \leq k \leq p}{b_{u_kv_k}}] \leq \sqrt{\mathbb{E}[\prod_{1 \leq k \leq p}{b^2_{u_kv_k}}]} \leq \sqrt{\prod_{e \in \{u_kv_k, 1 \leq k \leq p\}}{(\mathbb{E}[b^{2p}_{e}]})^{\frac{|\{k: 1 \leq k \leq p, u_kv_k=e\}|}{p}}}=\sqrt{\mathbb{E}[b^{2l_0}_{11}]};
    \end{equation}
    the second part is trivially satisfied since \(\mathcal{M}(((u_kv_k,l_k))_{1 \leq k \leq L})=\mathcal{M}(((u_kv_k,1))_{1 \leq k \leq L})=1,\) and the set in (\ref{sett}) is empty.
    \par
    Assume \(p \geq L+1.\) Glue as many paths of length \(1\) as possible in cycles (all loops are included, and at most one copy of \(uv\) is left out): denote these cycles by \(\mathcal{L}_1,\mathcal{L}_2, \hspace{0.05cm} ... \hspace{0.05cm},\mathcal{L}_{T_1},\) and take \(T_1\) minimal (i.e., no two distinct cycles share a vertex). Consider the remaining paths, and assume their indices are \(1,2, \hspace{0.05cm} ... \hspace{0.05cm},T_2;\) among them, \(2m\) have distinct endpoints since the tuple \((u_kv_k)_{1 \leq k \leq L}\) is even: without loss of generality, let \(u_kv_k=uv\) for \(1 \leq k \leq 2m,\) construct \(m\) cycles by merging the paths with indices \(2r-1,2r\) for \(1 \leq r \leq m,\) and note the rest of the paths are cycles. Hence the elements of \(\mathcal{P}\) are unions of \(T_0:=T_1+m+(T_2-2m)\) cycles, \((\mathcal{L}_i)_{1 \leq i \leq T_0};\) let
    \begin{equation}\label{d0def}
        \mathcal{D}_0=\prod_{1 \leq i \leq m}{\chi_{2|l_{2i-1}+l_{2i}}C_{\frac{l_{2i-1}+l_{2i}}{2}}} \cdot \prod_{2m <i \leq T_2}{\chi_{2|l_{i}}C_{\frac{l_i}{2}}}.
    \end{equation}
    \vspace{0.2cm}
    \par
    \textit{Case 1:} \(\mathcal{L}_{s_1}\) and \(\mathcal{L}_{s_2}\) share no edge for any \(s_1 \ne s_2, 1 \leq s_1,s_2 \leq T_0, s_1>T_1.\) %\((\mathcal{L}_i)_{1 \leq  i \leq T_0}\) are edge disjoint with \((\mathcal{L}_i)_{T_1< i \leq T_0}\).
    \par
    This entails the contributions split into a product with \(T_0-T_1+1\) factors, and an upper bound is
    \begin{equation}\label{err1}\tag{T1}
        n^{(p-L+l)/2} \sqrt{\mathbb{E}[b^{2l_0}_{11}]} \cdot \mathcal{D}_0 \cdot (1+ CL(n)p^2n^{-2\delta_1})^{L}
    \end{equation}
    from (\ref{subas}): if some indicator function in (\ref{d0def}) is \(0,\) then one factor in the expectations always vanishes as the cycles underlying it are odd, and consider next the remaining situations. The first \(T_1\) cycles yield at most \(\sqrt{\mathbb{E}[b^{2l_0}_{11}]}\) (arguing as for (\ref{holderr})), %by Hölder's inequality and \(l \to \mathbb{E}[b^{2l}_{11}]\) being nondecreasing (\(\mathbb{E}[b^{2}_{11}]=1\)), 
    while by (\ref{subas}), the remaining \(T_0-T_1\) give at most
    \[n^{\sum_{T_1<k \leq T_0}{(|\mathcal{L}_k|/2-d(\mathcal{L}_k))}} \cdot \mathcal{D}_0 \cdot (1+ CL(n)p^2n^{-2\delta_1})^{T_0-T_1} \leq n^{(p-L+l)/2} \cdot \mathcal{D}_0 \cdot (1+ CL(n)p^2n^{-2\delta_1})^{L}\]
    where \(d(\mathcal{L}_k)=1\) if \(\mathcal{L}_k\) has its two set vertices (the endpoints of the initial path(s) that consists of) distinct, otherwise \(d(\mathcal{L}_k)=0,\) because in step \(3\) there is another factor of \(n^{-1} (n^{-2})\) when \(\mathcal{L}_k\) has its two fixed vertices equal (distinct), and
    \[p-L+l=\sum_{1 \leq k \leq L}{l_k\chi_{u_k=v_k}}+\sum_{1 \leq k \leq L}{(l_k-1)\chi_{u_k \ne v_k}} \geq \sum_{T_1<k \leq T_0}{(|\mathcal{L}_k|-2d(\mathcal{L}_k))}\]
    (if \(\mathcal{L}_t\) contains \(u,v,\) then both paths forming it belong to the second sum: otherwise, its component belongs to the first). 
    \vspace{0.2cm}
    \par
    \textit{Case 2:} there are \(s_1 \ne s_2, 1 \leq s_1,s_2 \leq T_0, s_1>T_1\) such that \(\mathcal{L}_{s_1}\) and \(\mathcal{L}_{s_2}\) share an edge. %\((\mathcal{L}_i)_{1 \leq i \leq T_0}\) are not pairwise edge disjoint. 
    \par
    Take \(t_1<t_2\) with \(\mathcal{L}_{t_1},\mathcal{L}_{t_2}\) sharing an edge \(e=xy,\) and if possible, choose one of them to be a loop: note the two shared copies are not both among the paths of length \(1\) in light of the construction of \((\mathcal{L}_r)_{1 \leq r \leq T_1}\). %take \(s_1<s_2\) arbitrary with \(\mathcal{L}_{s_1},\mathcal{L}_{s_2}\) sharing an edge \(e:\) if both of its copies are paths of length \(1,\) then \(e=uv\) and  \(s_1 \leq T_0;\) either \(uv\) has odd multiplicity in \(\cup_{1 \leq k \leq T_0}{\mathcal{L}(k)},\) in which case all expectations vanish, or there is a path of length at least \(2\) containing \(uv\) (its multiplicity in \(\cup_{T_1< k \leq T_0}{\mathcal{L}(k)}\) is even and positive, hence at least \(2\)), yielding a pair with the desired properties. 
    The edge \(e\) belongs to two of the original paths, call them \(\rho_1,\rho_2:\) by deleting the first copy of \(e\) in both of them and gluing the resulting four paths into two cycles (one potentially empty), two edges are lost and the configuration is of the same type (i.e., a collection of paths with endpoints in \(\{u,v\}\)). %by gluing them at the \(x,y.\) 
    Consider what can occur with this merge by seeing \(p,L,l\) as functions of the configurations of interest:
    \par
    \((I)\) \(\rho_1,\rho_2\) are loops with the same endpoint, say \(u:\) \(L-l\) remains constant (\(l\) and \(L\) can only decrease simultaneously exactly when
    \(\rho_1\) or \(\rho_2\) has length \(1\));
    \par
    \((II)\) \(\rho_1,\rho_2\) are loops with different endpoints: \(L-l\) increases by at least \(1\) as \(l\) decreases by \(2,\) and \(L\) can only decrease by \(1\) (the merge is empty if and only if both \(\rho_1,\rho_2\) have length \(1,\) which cannot occur);% in this case);
    \par
    \((III)\) \(\rho_1\) is a loop, say with endpoint \(u,\) and \(\rho_2\) has endpoints \(u,v:\) \(L-l\) stays constant or decreases by \(1\) (\(l\) decreases solely when one of the merges is empty, in which case \(L\) also drops by \(1;\) \(L\) cannot decrease by \(2\) since \(\rho_1\) or \(\rho_2\) has length at least \(2\)); % so does \(L\)); 
    the case \(\rho_1\) containing \(u,v,\) \(\rho_2\) being a loop is analogous; 
    \par
    \((IV)\) \(\rho_1,\rho_2\) have endpoints \(u,v:\) the new paths can either both contain \(u,v,\) or be loops, one at \(u\) and one at \(v;\) in the former case, \(L-l\) can only decrease by \(1\) (\(l\) does not change, and \(L\) can solely drop by \(1\) because \(\rho_1\) or \(\rho_2\) has length at least \(2\)), while in the latter case, denoted by \((*),\) it suffices to assume both new paths are nonempty, case in which \(L-l\) decreases by \(2\) (one path being empty is equivalent to \(\rho_1\) or \(\rho_2\) having length \(1,\) and in this situation, both paths can be taken to have distinct endpoints; otherwise \(l\) increases by \(2\) and \(L\) stays constant, whereby \(L-l\) drops by \(2\)).%; otherwise say \(\rho_1=(u,v):\) then either \(L\) stays constant, \(l\) increases by \(2,\) or \(L\) decreases by \(1,\) \(l\) increases by \(1:\) the latter happens when the copy of \(uv\) in \(\rho_2\) is incident with one its endpoints), a case denoted by \((*).\)
    \par
    The new paths can be assembled back to their original format in at most \(4p^2=4(\sum_{1 \leq s \leq L}{l_s})^2\) ways insomuch as for fixed \(t_1,t_2\) with \(\rho_1,\rho_2\) having lengths \(l_{s_1},l_{s_2},\) respectively, the preimage of each merge is most \(4l_{s_1}l_{s_2}\) (in \((I),\) choose two vertices \(x,y\) at distance \(l_{s_1}-1\) in the merge, a total of at most \(l_{s_1}+l_{s_2}-1 \leq l_{s_1}l_{s_2},\) in \((II),\) on each path with endpoints \(u,v,\) choose a vertex such that two of the segments have lengths adding up to \(l_{s_1},\) yielding anew at most \(l_{s_1}+l_{s_2}-1,\) and in \((III),(IV),\) proceed as in \((II)\) and select as well orientation for the paths, giving  \(4(l_{s_1}+l_{s_2}-1) \leq 4l_{s_1}l_{s_2}\)). This together with the induction hypothesis provides a bound for \((I)-(IV),\) apart from \((*),\)
    \begin{equation}\label{err2}\tag{T2}
        4p^2n^{-1/2} \cdot n^{2\delta} \cdot n^{(p-L+l)/2} \cdot C_{p/2} \cdot \sqrt{\mathbb{E}[b^{2l_0}_{11}]}\cdot [(1+(L-l-1)!! \chi_{L>l})+L(n)(Cp^2)^{2L-l}\cdot L!\cdot n^{-2\delta_1}]:
    \end{equation}
    in all situations, at least a factor of \(n^{-1/2}\) is gained (\(p\) decreases by \(2,\) and \(l-L\) increases by at most \(1\)), the number of paths of length \(1\) can be left as \(l_0\) (in \((I),(II),\) another factor of \(n^{-1/2}\) can be added and note \[n^{-1/2}\sqrt{\mathbb{E}[b^{2l_0+2}_{11}]} \leq n^{-1/2+\delta}\sqrt{\mathbb{E}[b^{2l_0}_{11}]} \leq \sqrt{\mathbb{E}[b^{2l_0}_{11}]};\]
    in \((III),\) the aforesaid saving can also be employed because \(L\) decreasing by \(1\) leads to a new path with endpoints \(u,v\) and one empty, leaving \(l_0\) unchanged; finally, in \((IV),\) \(l_0\) can increase solely when \((*)\) holds), the two erased copies of \(xy\) can be accounted for by a factor of \(n^{2\delta},\) the uniform bound
    \[\mathcal{M}(((u_kv_k,l_k))_{1 \leq k \leq L}) \leq (1+(L-l-1)!! \chi_{L>l}) \cdot C_{p/2}\]
    can be employed (solely \(\mathcal{M}(((u_kv_k,l_k))_{1 \leq k \leq L})>1\) must be justified; when
    \[L-l=|\{k: 1 \leq k \leq L, u_k \ne v_k, l_k=1\}|,\] 
    the claim ensues because 
    \[\mathcal{M}(((u_kv_k,l_k))_{1 \leq k \leq L}) \leq \prod_{1 \leq k \leq L: l_k>1}{\chi_{2|l_k}C_{l_k/2}} \leq C_{p/2};\]
    otherwise the summands \(\mathcal{L}\) contributing to \(\mathcal{M}\) have the tuples
    \((\mathcal{L}_r)_{1 \leq r \leq t}\) of length \(2\) because this is the only way of merging paths with endpoints in \(\{u,v\}\) into elements of \(\cup_{q \geq 1}{\mathcal{C}(q)},\) given that any loop in the former belongs to \(\cup_{q \geq 1}{\mathcal{C}(q)}\) too; there are thus at most \((L-l-1)!!\) possibilities for \(\{\mathcal{L}_r,1 \leq r \leq t\},\) and each summand in (\ref{mdef}) is at most
    \begin{equation}\label{prodcat}
        \max_{(q,(p_j)_{1 \leq j \leq q}): \sum_{1 \leq j \leq q}{p_j}=p}{\prod_{1 \leq j \leq t}{C_{p_j/2}}} \leq C_{p/2}
    \end{equation}
    employing \(\mathcal{C}(a_1) \times \mathcal{C}(a_2) \times ... \times \mathcal{C}(a_k) \subset \mathcal{C}(\sum_{1 \leq r \leq k}{a_r})\)), and lastly, when \(2L-l\) increases, then it does so by at most \(2\) in \((II),\) in which case there is an additional factor of \(n^{-1}\) because \(p-L+l\) decreases by at least \(3\) from \(l-L=-(2L-l)+L\)
    (\(l\) can decrease in \((III)\) as well: however, it drops by \(1\) only when the copies of \(xy\) are incident with the endpoints of \(\rho_1,\rho_2,\) reducing \(L\) by \(1\) as well), and \((Cp^2)^2 \le C^2n^{4\delta_1} \leq n\) for \(n \geq n(\delta).\)
    \par
    Finally, take \((*):\) the induction hypothesis for the new configuration \(\rho\) yields the bound
    \begin{equation}\label{err3}\tag{T3}
        4p^2 \cdot n^{(p-L+l)/2} \cdot \sqrt{\mathbb{E}[b^{2l_0}_{11}]} \cdot  C_{p/2}L(n)\cdot(Cp^2)^{2L-l-2}\cdot L! \cdot n^{-2\delta_1}
    \end{equation}
    when \(\rho\) does not fall within the exceptional set described in the second part of the induction hypothesis as \(2L-l\) decreases by at least \(2\) (\(l\) increases by \(2,\) and \(L\) does not change). Suppose \(\rho\) is among those maximal configurations. %the two newly formed loops are vertex-disjoint by the definition of \(\mathcal{M}.\) 
    If there are no summands in \(\mathcal{M}\) for \(\rho,\) then either 
    \par
    \((c1)\) all the loops in \(\rho\) have length \(1,\) or 
    \par
    \((c2)\) there is also no summand in \(\mathcal{M}\) for the original arrangement (\(\cup_{q \geq 1}{\mathcal{C}(q)}\) is invariant under removing loops). 
    \par
    Both
    parts of the conclusion hold under \((c1):\) the original configuration has all paths of length one, apart one of length \(2,\) \((u,x,v)\) with \(x \in \{u,v\},\) or two of length \(2,\) \((u,x,v)\) and \((u,y,v).\) In the former case, 
    \begin{equation}\label{c1p1}\tag{\textit{c1.1}}
        \sum_{(\mathbf{i}_k)_{1 \leq k \leq L} \in \mathcal{P}}{\mathbb{E}[\prod_{1 \leq k \leq L}{b_{\mathbf{i}_k}}]} \leq 2\sqrt{\mathbb{E}[b_{11}^{2l_0+2}]} \leq 2n^{\delta}\sqrt{\mathbb{E}[b_{11}^{2l_0}]} \leq n^{1/2}\sqrt{\mathbb{E}[b_{11}^{2l_0}]}
    \end{equation}
    via a rationale similar to (\ref{holderr}) after choosing \(x \in \{u,v\},\) and \(p-L+l= 1,\) while in the latter case,
    \begin{equation}\label{c1p2}\tag{\textit{c1.2}}
        \sum_{(\mathbf{i}_k)_{1 \leq k \leq L} \in \mathcal{P}}{\mathbb{E}[\prod_{1 \leq k \leq L}{b_{\mathbf{i}_k}}]} \leq \sqrt{\mathbb{E}[b_{11}^{2l_0}]} \cdot (n+4n^{4\delta})=n\sqrt{\mathbb{E}[b_{11}^{2l_0}]} \cdot (1+4n^{-4\delta_1})
    \end{equation}
    since \(x=y\) or \(\{x,y\}=\{u,v\}\) (the two paths of length \(2\) must share an edge), \(n\) accounts for selecting \(x=y,\) and notice %\not \in \{u,v\}
     \(p-L+l=2,\) \(\mathcal{M}(((u_k,v_k,l_k))_{1 \leq k \leq L})=1\) (for the paths to form elements of \(\cup_{q \geq 1}{\mathcal{C}(q)},\) \(x,y \not \in \{u,v\},\) and thus \(x=y:\) this yields exactly one configuration as the original tuple is even). Take now \((c2):\) this can occur solely when there is \(r, 1 \leq r \leq L,\) \(u_r=v_r, 2|l_r+1, l_r>1,\) or \(|\{k: 1 \leq k \leq L, u_kv_k=uv,2|l_k+1\}|\) is odd; in both situations, all expectations vanish (\(\rho_1,\rho_2\) are not loops, and thus all paths with identical endpoints of length at least \(2\) are edge-disjoint from the rest in light of the definition of \((t_1,t_2)\)).
    %there is nothing to prove \(\mathcal{M}(((u_k,v_k,l_k))_{1 \leq k \leq L})=1.\) 
    \par
    Assume now \(\mathcal{M}>1\) for the paths comprising \(\rho:\) by adding back the two copies of \(xy,\) the configuration remains maximal (\((x,y,\mathcal{L}',y,x,\mathcal{L}'',x) \in \cup_{q \geq 1}{\mathcal{C}(q)}\) because \(\mathcal{L}',\mathcal{L}'' \in \cup_{q \geq 1}{\mathcal{C}(q)}\) are vertex-disjoint, and \(x \ne y\) from the second part of the induction hypothesis on the clipped configuration), entailing these terms are accounted by the first component in (\ref{claimmmedd}).
    \par
    Putting together (\ref{err1})-(\ref{err3}) completes the induction step: the final bound becomes
    \[n^{(p-L+l)/2} \cdot \sqrt{\mathbb{E}[b^{2l_0}_{11}]} \cdot [\mathcal{M}(((u_kv_k,l_k))_{1 \leq k \leq L})+C_{p/2}\cdot [(1+CL(n)p^2n^{-2\delta_1})^{L}-1]+\]
    \[4p^2n^{-2\delta_1}C_{p/2} \cdot L(n)(Cp^2)^{2L-l}\cdot L!\cdot n^{-2\delta_1}+4p^2n^{-2\delta_1}C_{p/2}(1+(L-l-1)!!\chi_{L>l})+4p^2n^{-2\delta_1}C_{p/2}(Cp^2)^{2L-l-2} \cdot L!];\]
    \(p=2\) is clear as \(p=L\) or the sum is \(\sum_{x}{\mathbb{E}[b^2_{wx}]} \leq n, p-L+l=2-1+1=2,\mathcal{M}((w,w,2))=1\) all the configurations with \(x \ne w\) are maximal, and \(x=w\) yields \(1 \leq n \cdot n^{-2\delta_1},\) and for \(p \geq 4,\) 
    \[(1+CL(n)p^2n^{-2\delta_1})^{L}-1 \leq L\cdot CL(n)p^2n^{-2\delta_1}(1+CL(n)p^2n^{-2\delta_1})^{L-1} \leq L\cdot CL(n)p^2n^{-2\delta_1}\exp(CL(n)p^2n^{-2\delta_1}(L-1))\]
    and \(L\cdot CL(n)p^2\exp(\frac{L-1}{4}) \leq L!CL(n)p^2 \leq L(n)(Cp^2)^{L}L!/4, \hspace{0.1cm} 4p^2n^{-2\delta_1} \leq \frac{1}{4},\)
    \[4p^2(1+(L-l-1)!!\chi_{L>l}) \leq 8p^2L! \leq (Cp^2)^{L}L!/4, \hspace{0.4cm} 4p^2(Cp^2)^{2L-l-2} \cdot L! \leq (Cp^2)^{2L-l}L!/4\]
    (the final step uses there is no overlap among maximal configurations in \textit{Case 1} and \textit{Case 2}, respectively).
\end{proof}

The generalization of Lemma~\ref{stilldom} employed below reads as follows.  

\begin{lemma}\label{bigcycles}
    (\ref{claimmmedd}) continues to hold in the setup of Lemma~\ref{stilldom}  after dropping the condition 
    \begin{center}
        \(u_k,v_k \in \{u,v\}\) for all \(1 \leq k \leq L,\)
    \end{center}
    letting \(C \geq 64,\) and replacing \(L!\) with \((2L)!.\)
\end{lemma}

\begin{proof}
The rationale is almost identical to the case \(\{u_k,v_k: 1 \leq k \leq L\} \subset \{u,v\},\) and solely the differences are elaborated on. Use anew induction on \(p \geq L:\) the exceptional configurations continue satisfying the properties stated in Lemma~\ref{stilldom} as well as each cycle obtained from assembling paths with indices in (\ref{sett}) contains at most two of them. %are a those in which the paths with indices in (\ref{sett}) are assembled in a collection of edge-disjoint elements of \(\cup_{q \geq 1}{\mathcal{C}(q)},\) sharing vertices solely trivially
\par
Glue the paths of length \(1\) in stages: first, all loops with a shared vertex into one cycle, second, as many pairs of undirected edges with distinct endpoints as possible, and third, as many cycles consisting of pairwise distinct undirected edges as possible; denote the outcomes by \((\mathcal{L}_r)_{1 \leq r \leq T_1}.\) The paired endpoints of the paths of length at least \(2\) (i.e., compress \((u,u_1,u_2,\hspace{0.05cm}...\hspace{0.05cm},u_k,v)\) to \(uv\)) together with the leftovers of length \(1\) form an even tuple, and induction on their number \(r\) yields they can be assembled into a collection of cycles (denote these by \((u_{k0}v_{k0})_{1 \leq k \leq r};\) if \(r=1,\) then \(u_{10}=v_{10};\) suppose \(r \geq 2;\) if \(u_{10}=v_{10},\) use the induction hypothesis on \((u_{k0}v_{k0})_{2 \leq k \leq r};\) else there is \(2 \leq k \leq r\) such that \(u_{10} \in \{u_{k0}\} \cup \{v_{k0}\};\) suppose without loss of generality that \(k=2,\) and let \((u_{00}w_1,u_{00}w_2)=(u_{01}v_{01},u_{02}v_{02});\) the induction hypothesis can be applied to \(w_1w_2,(u_{k0}v_{k0})_{3 \leq k \leq r},\) whereby the result ensues by replacing \(w_1w_2\) in its corresponding cycle with \(w_1u_{00},u_{00}w_2\)). Choose a gluing with as many components as possible and denote them by \((\mathcal{L}_r)_{T_1< r \leq T_0}:\) note all have pairwise distinct vertices (otherwise a cycle could be split into two, increasing the number of cycles by \(1\)).%, and denote them by \((\mathcal{L}_r)_{T_1< r \leq T_0}.\)
%\textcolor{brown}{Change \textit{Case 1} to no pair of distinct cycles sharing an edge containing a path of length at least \(2\) (\textit{Case 2} is its complement).} 
\vspace{0.2cm}
\par
\textit{Case 1.2}: \(\mathcal{L}_{s_1}\) and \(\mathcal{L}_{s_2}\) share no edge for any \(s_1 \ne s_2, 1 \leq s_1,s_2 \leq T_0, s_1>T_1.\)
%there are no \(r_1 \ne r_2\) with \(1 \leq r_1, r_2 \leq T_0,r_1>T_1,\) and \(\mathcal{L}_{r_1},\mathcal{L}_{r_2}\) sharing an edge.
\par
An analogous rationale to the one for \textit{Case 1} goes through: for the exponent of \(n,\) the relevant function is \(d(\mathcal{L}_k)=q_k-1,\) where \(q_k\) is the number of paths forming \(\mathcal{L}_k\) (\(q_k\) vertices are fixed in step \(3\) because these cycles have pairwise distinct vertices), and the desired inequality is justified in the same vein. The additional claim on the maximal configurations (stated at the beginning of the proof) follows from \((q-1)-q/2=q/2-1>0\) for \(q>2,\) entailing a factor of \(n^{-1/2}\) can be added to the bound when at least one of the cycles consists of at least three paths. Inequality (\ref{mmboundfoo}) below and \(-2\delta_1 \geq -1/2\) yield the configurations that do not fall within the exceptional set generate at most \(1/4\) of the second term in the bound.
\vspace{0.2cm}
\par
\textit{Case 2.2}: there are \(s_1 \ne s_2\) with \(1 \leq s_1, s_2 \leq T_0,s_1>T_1,\) and \(\mathcal{L}_{s_1},\mathcal{L}_{s_2}\) sharing an edge.
\par
A pair \(t_1,t_2\) as described in \textit{Case 2} continues to exist. 
%an identical argument yields the desired claim unless \((x,y)\) is a path of length \(1\) with \(x \ne y,\) \(xy\) having multiplicity \(1\) in \(\cup_{T_1<r \leq T_0}{\mathcal{L}(r)}\) and thus odd in \(\cup_{r \leq T_1}{\mathcal{L}(r)};\) since the endpoints of the paths underlying \(\cup_{r \leq T_1}{\mathcal{L}(r)}\) form an even tuple, there must be an edge \((y,z)\) of odd multiplicity in \(\cup_{k \leq T_1}{\mathcal{L}(k)};\) this entails a path among those of \(\cup_{T_1<k \leq T_0}{\mathcal{L}(k)}\) contains a copy of it, giving either the needed pair or another path \((y,z)\) of length \(1;\) iterating this procedure finitely many times produces either \(t_1,t_2\) as wanted, or a sequence \(x_1,x_2, \hspace{0.05cm} ... \hspace{0.05cm},x_m\) with \(x_1x_2,x_2x_3,\hspace{0.05cm}...\hspace{0.05cm},x_mx_1\) paths of length \(1,\) each among the components of \(\cup_{T_1<k \leq T_0}{\mathcal{L}(k)},\) contradicting the maximality of \(T_1.\) 
Proceeding with the argument in \textit{Case 2}, the other possibilities (i.e., the endpoints of \(\rho_1,\rho_2\) forming a set of size at least \(3\)) apart from \((I)-(IV)\) can only increase \(l\) by \(1,\) and thus all situations apart from \((*)\) continue to yield the \(n^{-1/2}\) factor in (\ref{err2}). It can be seen that
\begin{equation}\label{mmboundfoo}
    \mathcal{M}(((u_kv_k,l_k))_{1 \leq k \leq L}) \leq (2L-2l)! \cdot C_{p/2}:
\end{equation}
%(1+(2L-2l)!!\chi_{L>l})C_p/2
in virtue of (\ref{prodcat}), showing there are at most \((2L-2l)!\) configurations underlying \(\mathcal{M}\) suffices. These consist of merges of the \(L-l\) paths with distinct endpoints, 
and the aforesaid claim is a consequence of the following injective mapping that takes them into permutations of \(\{k_1,k_2,\hspace{0.05cm}...\hspace{0.05cm},k_{L-l},k_1+L,k_2+L,\hspace{0.05cm}...\hspace{0.05cm},k_{L-l}+L\}\) for \(\{k_1,k_2,\hspace{0.05cm}...\hspace{0.05cm},k_{L-l}\}=\{k: 1 \leq k \leq L, u_k \ne v_k\}.\) Let %being injective.
\[(\mathcal{L}_1,\mathcal{L}_2,\hspace{0.05cm}...\hspace{0.05cm},\mathcal{L}_t)\to (k_{11}+s_{11}L,k_{12}+s_{12}L,\hspace{0.05cm}...\hspace{0.05cm},k_{1m_1}+s_{1L}L,\hspace{0.05cm}...\hspace{0.05cm},k_{t1}+s_{t1}L,k_{t2}+s_{t2}L,\hspace{0.05cm}...\hspace{0.05cm},k_{tm_t}+s_{tm_t}L),\]
%k_{22}+s_{22}L,\hspace{0.05cm}...\hspace{0.05cm},k_{2m_2}+s_{2m_2}L,
where for \(1 \leq j \leq t,\) 
\(\mathcal{L}_j\) is formed by pasting the paths with positions \((k_{jy})_{1 \leq y \leq m_j},\) the indices listed in the order in which they are glued, i.e., \((u^*_{k_{jy}},v^*_{k_{jy}})=(w_{jy},w_{j(y+1)}),u^*_{k_{jy}}v^*_{k_{jy}}=u_{k_{jy}}v_{k_{jy}},\) and 
\[s_{jy}=\begin{cases}
    1, \hspace{1cm} (u^*_{k_{jy}},v^*_{k_{jy}})=(u_{k_{jy}},v_{k_{jy}}),\\
    0, \hspace{1cm} (u^*_{k_{jy}},v^*_{k_{jy}})=(v_{k_{jy}},u_{k_{jy}}),
\end{cases}\]
for all \(1 \leq y \leq m_j,w_{j(m_j+1)}:=w_{j1},\)
\(k_{j1}=\min_{1 \leq y \leq m_j}{k_{jy}},\) and \(k_{j2} \leq k_{jm_j}.\) Constructing the inverse of this function proves it is injective: fix \((s_{q})_{1 \leq q \leq L-l}\) in the image. The tuple \((s_q\chi_{s_q \leq L}+(s_q-L)\chi_{s_q>L})_{1 \leq q \leq L-l}\) is a permutation \((g(q))_{1 \leq q \leq L-l}\) of \(k_1,k_2,\hspace{0.05cm}...\hspace{0.05cm},k_{L-l}\) because it
consists of pairwise distinct elements, and it naturally yields orientations for the edges in the original even tuple:
\[(u_{0q},v_{0q})=\begin{cases}
    (u_{g(q)},v_{g(q)}), \hspace{0.8cm} s_q >L,\\
    (v_{g(q)},u_{g(q)}), \hspace{0.8cm} s_q \leq L;
\end{cases}\]
what is left defining \((m_j)_{1 \leq j \leq t}.\)
Note \(m_1\) is minimal such that \(((u_{0q},v_{0q}))_{q \leq m_1}\) form a cycle (in this order); put aside these \(m_1\) edges, repeat this procedure to obtain all the lengths \(m_1,m_2,\hspace{0.05cm}...\hspace{0.05cm},m_t,\) and thus recover \((\mathcal{L}_1,\mathcal{L}_2,\hspace{0.05cm}...\hspace{0.05cm},\mathcal{L}_t)\) (the minimality employed to determine the lengths ensues from each cycle having pairwise distinct vertices: see \((ii)\) below (\ref{mdef})).
%the resulting cycles have all the endpoints of multiplicity \(2\) (i.e., after fusing them at the shared endpoints, the remaining vertices are pairwise distinct), and each contains at least \(2\) paths, implying their number is at most \(\prod_{1 \leq r \leq L'}{(m_r-1)!!},\) where
%\[L'=|\{u_kv_k: 1 \leq k \leq L,u_k \ne v_k\}|, \hspace{0.3cm} \{u_kv_k: 1 \leq k \leq L,u_k \ne v_k\}=\{u_{01}v_{01},u_{02}v_{02}, \hspace{0.05cm} ... \hspace{0.05cm},u_{0L'}v_{0L'}\},\]
%\[m_r=|\{k:1 \leq k \leq L, \exists u,v,w, \hspace{0.1cm} u_kv_k=uv, u_{0r}v_{0r}=uw\}|, \hspace{0.3cm} 1 \leq r \leq L',\] 
%because the mapping taking each path to its right neighbor in the cycle it belongs to (take in each cycle the path of smallest index and let its right neighbor have minimal index as well: this defines an orientation) is injective, and 
%\[\prod_{1 \leq r \leq L'}{(m_r-1)!!} \leq (\sum_{1 \leq r \leq L'}{m_r}-1)!! \leq (2L-2l-1)!!\] employing \((a-1)!!(b-1)!! \leq (a+b-1)!!\) and \(a \to (a-1)!!\) non-decreasing. 
\par
%Lastly, consider the situation with the second part of the induction hypothesis satisfied (that is, \((*)\)): \((c1)\) the expectations vanish unless \(x=y\) or \(x,y \in \cup_{1 \leq k \leq L}{(\{u_k\} \cap \{v_k\})},\) yielding 
Lastly, consider the case with the second part of the induction hypothesis satisfied (that is, the maximal configurations in \((*)\)): \((c1)\) (\ref{c1p1}) and (\ref{c1p2}) hold once \(2,4\) are replaced by \(2L, 4L^2,\) respectively, and \(2L \leq 2p \leq n^{\delta_1}\) can be employed; \((c2)\) is identical.
%the expectations vanish unless \(x=y\) or \(x,y \in \{u_k,v_k, 1 \leq k \leq L\},\) yielding 
%\[\sum_{(\mathbf{i}_k)_{1 \leq k \leq L} \in \mathcal{P}}{\mathbb{E}[\prod_{1 \leq k \leq L}{b_{\mathbf{i}_k}}]} \leq \sqrt{\mathbb{E}[b_{11}^{2l_0}]} \cdot (n+(2L)^2n^{4\delta})=n\sqrt{\mathbb{E}[b_{11}^{2l_0}]} \cdot (1+4L^2n^{-4\delta_1}),\]
The final computations are very similar, and \(4L^2 \leq L^{2L}\) (\(p \geq L,2L-l \geq L\)) for \(L \geq 2\) can be used to conclude %The remainder of the argument continues valid, concluding 
the desired claim (due to the increase of the absolute constant \(C,\) each component can be made at most \(1/8\) of the second term in the bound, which together with the extra \(1/4\) in \textit{Case 1.2} conclude the induction hypothesis).
\end{proof}

Before stating and justifying the last extension of Lemma~\ref{stilldom}, consider another version of \(\mathcal{M}.\) For \(t_0 \in \mathbb{Z}_{\geq 0},\) \(((i_kj_k,l_k))_{1 \leq k \leq L},\) and \(X\) a real-valued random variable with finite moments, let
\begin{equation}\label{mdef2}
    \mathcal{M}_X(t_0,((i_kj_k,l_k))_{1 \leq k \leq L})=\begin{cases}
        %0, \hspace{6.7cm} t_0 \not \in \mathbb{Z},\\
        1, \hspace{8.4cm}  l_1=l_2=...=l_L=1,\\
        \sum_{\mathcal{L}}{V(X,\mathcal{L}_{0})\prod_{k \in \mathcal{L}_{01}}{s(\frac
        {l_k-1}{2})}\cdot \prod_{t \in \mathcal{L}_{02}}{\mathcal{D}_{2t(3)}(l_{t(1)},l_{t(2)})} \cdot \prod_{y \geq 1}{\mathcal{D}(\mathcal{Q}(\mathcal{L}_y))}}, \hspace{0.2cm} else,
    \end{cases} 
\end{equation}
for \(s(\cdot)\) given by (\ref{formm}), padded by zeros (\(s(t)=0\) for \(t \not \in \mathbb{Z}_{\geq 0}\)), where the summation is over \(\newline \mathcal{L}=\{\mathcal{L}_{00},\mathcal{L}_{01},\mathcal{L}_{02}, \mathcal{L}_1,\mathcal{L}_2, \hspace{0.05cm} ... \hspace{0.05cm},\mathcal{L}_t\},\)
satisfying \((i)-(iii), (v)\) (beneath (\ref{mdef})) with \(\mathcal{L}_0=\mathcal{L}_{00} \cup \mathcal{L}_{01} \cup \{a_1,b_1,a_2,b_2,\hspace{0.05cm}...\hspace{0.05cm},a_{t_1},b_{t_1}\},\)
\par
\((vi) \hspace{0.1cm} \mathcal{L}_{00}=\{k:1 \leq k \leq L, i_k=j_k,l_k=1\} \cup (\cup_{1 \leq k \leq L}{\mathcal{S}(i_kj_k)}),\) where \(\mathcal{S}(uv)\) is the set of the smallest \(2 \cdot l(uv)\) elements of \(\mathcal{S}'(uv)=\{k: 1 \leq k \leq L, i_kj_k=uv,l_k=1\},l(uv):=\lfloor \frac{|\mathcal{S}'(uv)|}{2} \rfloor,\)
\par
\((vii) \hspace{0.1cm} \mathcal{L}_{01}=\{k: 1 \leq k \leq L, i_k=j_k,l_k>1\}-\{s_1,s_2: \exists j \in \{0,1\},(s_1,s_2,j) \in \mathcal{L}_{02}\},|\mathcal{L}_{01}|+2|\mathcal{L}_{02}|=t_0,\) \(\mathcal{L}_{02}=\{(a_1,b_1,j_1),(a_2,b_2,j_2),\hspace{0.05cm}...\hspace{0.05cm},(a_{t_1},b_{t_1},j_{t_1})\},\) where for \(1 \leq r,s \leq t_1,\) 
\[a_r<b_r, \hspace{0.2cm} i_{a_r}=j_{a_r}=i_{b_r}=j_{b_r},  \hspace{0.2cm} a_r,b_r \not \in \mathcal{L}_{00}, \hspace{0.2cm} 
 j_r \in \{0,1\}, \hspace{0.2cm} |\{a_s,b_s\} \cap \{a_r,b_r\}|>0 \Leftrightarrow r=s,\]
and \(\mathcal{D}_{2j}(a,b)\) is number of Dyck paths of length \(a+b-2\) corresponding to elements of \(\cup_{q \geq 1}{\mathcal{C}(q)},\) each being the merge of two non-edge disjoint cycles \(\mathbf{i}=(v_0,v_1,\hspace{0.05cm}...\hspace{0.05cm},v_{a-1},v_0),\mathbf{j}=(v_a,v_{a+1},\hspace{0.05cm}...\hspace{0.05cm},v_{a+b-1},v_a)\) with \(v_a=v_0,\) and the shared edge \(e\) employed in the merge given by (\ref{paste1}) and (\ref{paste2}) %in subsection~\ref{treigen1} 
satisfying 
\begin{equation}\label{jlabels}
    \begin{cases}
    e \ne v_0v, \hspace{5.2cm} j=0,\\
    \exists  0 \leq k \leq a, e=v_kv_{k+1}=v_0v, v \ne v_0, \hspace{0.7cm} j=1:%\\
    %e=v_0v_0,\hspace{4.9cm}j=2:
\end{cases}
\end{equation}
%map described in subsection~\ref{treigen1}: 
a rationale analogous to the one presented in Lemma~\ref{lems} yields that for \(C'>0\) universal,
\begin{equation}\label{newm}
    \mathcal{D}_2(a,b):=\sum_{0 \leq j \leq 1}{\mathcal{D}_{2j}(a,b)} \leq C' \cdot C_{(a+b-2)/2},
\end{equation}
\par
\((viii) \hspace{0.1cm} V(X,\mathcal{L}_{0})=\prod_{v \in \{i_k,j_k, 1 \leq k \leq L\}}{\mathbb{E}[X^{|V_{2}(v)|+|V_{11}(v)|}] \cdot (\mathbb{E}[X^4])^{|V_{1}(v)|}},\)
where 
\[V_1(v)=\{k: 1 \leq k \leq L, \exists (k,k',1) \in \mathcal{L}_{02}, i_{k}=v\}, \hspace{0.5cm} V_2(v)=\{k: 1 \leq k \leq L, s \in \mathcal{L}_{01}, i_k=v,2|l_k+1\},\]
\begin{equation}\label{vs}\tag{VX}
    V_{11}(v)=\{k: 1 \leq k \leq L, i_k=j_k=v,l_k=1\},
\end{equation}
%\[S(a,b)=\begin{cases}
%    \{0,1,2\}, \hspace{0.4cm} 2|a+1,\\
%    \{0\}, \hspace{0.7cm} 2|a,
%\end{cases}\]
and by convention, \(V(X,\mathcal{L}_{0})=1\) when \(\mathcal{L}_{0}=\emptyset.\) By a slight abuse of notation, let
\begin{equation}\label{maxx}\tag{maxV}
  \max_{\mathcal{L}'}{V(X,\mathcal{L}'_0)}:=(\mathbb{E}[X^4])^{t_0/2}\max_{(k(v))_{v \in S},\sum_{v \in S}{k(v)} \leq t_0,k(v) \leq o(v)}{\prod_{v \in S}{\mathbb{E}[X^{k(v)+|V_{11}(v)|}] \cdot (\mathbb{E}[X^4])^{-k(v)/2}}},  
\end{equation}
where \(o(v)=|\{k: 1 \leq k \leq L, i_k=j_k=v,2|l_k+1,l_k>1\}|.\) This quantity is at least \(V(X,\mathcal{L}_0)\) for any \(\mathcal{L}_0\) as constructed above, and it is more amenable to induction than the maximum over all such sets.
%\textcolor{red}{add solos: cycles that are elements of cn with a 11 attached}
\par
\(\mathcal{M}_X\) differs from \(\mathcal{M}\) in one regard: it forces some loops to be non-edge disjoint from the rest (this is encompassed by \((vii)\)). These configurations are vital for the variances of the diagonal entries \(e_i^TA^pe_i,\) which depend primarily on pairs of cycles that share an edge, and this phenomenon naturally propagates to the higher moments of these random variables. Nevertheless, the core idea behind \(\mathcal{M}_X\) remains identical to that underlying \(\mathcal{M}:\) the paths are pasted in elements of \(\cup_{q \geq 1}{\mathcal{C}(q)},\) with no nontrivial vertex intersections (i.e., if two distinct cycles share a vertex, then each contains a path with it as one of its endpoints). Since a pair of edges is deleted in each merge in \((vii)\), there can be edge multiplicities larger than \(2:\) %exclusively when the lengths of the cycles are odd: 
namely, solely loops at the endpoints can have multiplicity at least \(6,\) and additionally edges incident with the first vertex in the merge could have multiplicity \(4,\) both of these situations being tracked by \((viii).\) It must be noted that the latter case is possible solely when the lengths of the cycles are odd (both must be elements of \(\cup_{l \in \mathbb{N}}{\mathcal{C}(l)}\) with an additional loop \(v_0v_0\) attached, where \(v_0\) is their first vertex, which is shared), a consequence of the property stated above (\ref{shiftt}).
%analysis for the variance in subsection~\ref{treigen1} (see discussion above (\ref{ref1}): both cycles must be elements of \(\cup_{l \in \mathbb{N}}{\mathcal{C}(l)}\) with an additional loop \(v_0v_0\) attached). %while the latter yields each belongs to \(\cup_{l \in \mathbb{N}}{\mathcal{C}(l)}\) with the first shared edge incident with \(v_0\)). 
For instance,
when \(L=6,u_k=v_k=u,l_k=2p_k+1, 1 \leq k \leq 6,\) the maximal configurations include unions of six paths belonging to \(\cup_{q \geq 1}{\mathcal{C}(q)},\) each with one loop at some appearance of \(u\) and no two sharing other vertex, leading to \(uu\) having multiplicity \(6.\)
%for loops at the endpoints, which is what \((viii')\) tracks: e.g., for \(L=4,u_k=v_k=u,l_k=2p_k+1, 1 \leq k \leq 4,\) the maximal configurations include unions of four paths belonging to \(\cup_{q \geq 1}{\mathcal{C}(q)}\) with one loop at \(u,\) leading to \(uu\) having multiplicity \(4.\)

\begin{lemma}\label{bigcycles1}
    Under the assumptions in Lemma~\ref{bigcycles}, require the elements of \(\mathcal{P}\) to have at least \(t_0\) paths with \(u_k=v_k,l_k>1\) and sharing some edge with other paths (i.e., each is not edge-disjoint with the rest). Then 
    \[\sum_{(\mathbf{i}_k)_{1 \leq k \leq L} \in \mathcal{P}}{\mathbb{E}[\prod_{1 \leq k \leq L}{b_{\mathbf{i}_k}}]} \leq n^{(p-L+l)/2} \cdot n^{-t_0/2} \cdot \sqrt{\mathbb{E}[b^{2l_{0}}_{11}]} \cdot [\mathcal{M}_{b_{11}}(t_0,((u_kv_k,l_k))_{1 \leq k \leq L})+\]
    \begin{equation}\label{claimmmedd2}
    +C_{\lfloor p/2 \rfloor }L(n) \cdot (CC'p^2)^{10(2L-l)}\cdot (2L)!\cdot n^{-2\delta_1} \cdot \max_{\mathcal{L'}}{V(b_{11},\mathcal{L}'_{0})}],
    \end{equation}
    where \(C'>0\) satisfies (\ref{newm}), \(C \geq 192\sqrt{\mathbb{E}[b^4_{11}]},\) %\(l_{0j}=\{k: 1 \leq k \leq L, l_k=1,\chi_{u_k=v_k}=j\}, j \in \{0,1\},\) 
    and the maximum is given by (\ref{maxx}).
    %the maximum is taken over sets \(\mathcal{L'}\) underlying (\ref{mdef2}).% for a sub-tuple of \(((u_kv_k,l_k))_{1 \leq k \leq L}.\)
\end{lemma}

\begin{proof}
Proceed in the same vein as in Lemma~ \ref{bigcycles}, and for simplicity, let \(X=b_{11}:\) the second part of the induction hypothesis remains the same for the cycles with indices in \(\{1,2, \hspace{0.05cm} ... \hspace{0.05cm},L\}-\mathcal{L}_{0},\) the cycles with positions in \(\mathcal{L}_{01}\) with the loops of length \(1\) left out, and the merges of the pairs in \(\mathcal{L}_{02}.\) %(the common edges left out). 
When \(t_0=0,\) Lemma~\ref{bigcycles} yields the conclusion since 
\[\mathcal{M}_{b_{11}}(0,((u_kv_k,l_k))_{1 \leq k \leq L}) \geq \mathcal{M}(((u_kv_k,l_k))_{1 \leq k \leq L},\]
using \(V(b_{11},\mathcal{L}_{02}) \in \{0\} \cup [1,\infty)\) for all \(\mathcal{L}_{02}.\)
%In virtue of (\ref{newm}), 
Reasoning as for (\ref{mmboundfoo}) and employing (\ref{newm}) render
\begin{equation}\label{mmbound2}
    \mathcal{M}_X(t_0,((u_kv_k,l_k))_{1 \leq k \leq L}) \leq (1+(2L-2l)!! \chi_{L>l}) \cdot (C')^{L/2}C_{p/2} \cdot \max_{\mathcal{L}'}{V(X,\mathcal{L}'_{02})}.
\end{equation}
The case \(t_0=L=l=2\) can be justified via Lemma~\ref{stilldom}: the pairs whose merge described by (\ref{paste1}) and (\ref{paste2}), a cycle of length \(l_1+l_2-2=p-2,\) belongs to \(\mathcal{C}(p/2-1)\) yield a contribution of at most \(\mathcal{M}_{b_{11}}(2,(u_k,v_k,l_k)_{1 \leq k \leq 2})\) (\(\mathcal{L}_{00}=\emptyset\) because \(t_0=L,\) \(\mathcal{L}_{01}=\{1,2\}\) if the shared edge is \(u_1u_1,\) and else \(\mathcal{L}_{01}=\emptyset,\mathcal{L}_{01}=\{(1,2,j)\}\) with \(j\) as defined in (\ref{jlabels})).
%\[\mathcal{M}_{b_{11}}(1,(u_k,v_k,l_k)_{1 \leq k \leq 2})=\mathcal{D}_{20}(l_1,l_2)+(\mathbb{E}[b^4_{11}]-1)\mathcal{D}_{21}(l_1,l_2),\]
%i.e., the pairs with \(t_0=2\) are analogous to their counterparts in (\ref{startingpoint}), 
%yielding the leading term in %the second moment 
%analyzed in subsection~\ref{treigen1}, with the cross term zeroed out (i.e., \(\mathbb{E}[a_{\mathbf{i}} \cdot a_{\mathbf{j}}]\) instead of \(\mathbb{E}[a_{\mathbf{i}}] \cdot \mathbb{E}[a_{\mathbf{j}}]\)). 
Consider now pairs \((\mathbf{i},\mathbf{j})\) for which the merge %described by (\ref{paste1}) and (\ref{paste2}), a cycle of length \(l_1+l_2-2=p-2,\) 
is even but does not belong to \(\mathcal{C}(p/2-1).\) Let \(xy\) be the first shared edge between \(\mathbf{i}\) and \(\mathbf{j}:\) use Lemma~\ref{stilldom} on the four paths created by cutting the copy of \(xy\) in each of the cycles that was employed for the merge (i.e., \(i_{t-1}i_{t},j_{s-1}j_{s}\)), two of them being \((x,y),\) and two having endpoints \((x,y).\) Suppose first \(x \ne y:\) the configurations of interest do not contribute to the first term in (\ref{claimmmedd}) (else the merge would belong to \(\mathcal{C}(p/2-1)\)), whereby the contribution is at most
\[n^{\chi_{x\ne u_1}+\chi_{y\ne u_1}-\chi_{x \ne u_1}\chi_{y \ne u_1}} \cdot (2p)^2 \cdot n^{p/2-2} \sqrt{\mathbb{E}[b^4_{11}]} \cdot C_{p/2} L(n)(Cp^2)^{8} \cdot 4! \cdot n^{-2\delta_1} \leq \]
\[\leq \frac{1}{2}n^{p/2-1} \cdot C_{p/2} L(n)(Cp^2)^{9} \cdot n^{-2\delta_1} \leq \frac{1}{2}n^{p/2-1} \cdot C_{p/2} L(n)(Cp^2)^{9} \cdot n^{-2\delta_1} \cdot \max_{\mathcal{L}'}{V(b_{11},\mathcal{L}'_{0})}\]
because in this new context, \(x,y\) are fixed, the vertex \(u_1,\) which appears in the paths, yields a discount factor of \(n^{-\chi_{x \ne u_1}\chi_{y \ne u_1}}\) due to being already selected, \(l=0,l_0=2,L=4,\max_{\mathcal{L}'}{V(b_{11},\mathcal{L}'_{0})}\geq 1,\) and the preimages of any cycle of length \(p-2\) has size at most \(p \cdot 2p\) under the described merge together with the split into four paths (at most \(p\) possibilities for choosing the position of the first vertex, and \(2\) subsequent possible orientations). When \(x=y,\) use induction on the cycles with the two copies of \(xx\) erased: the contribution is at most
\[(2p)^2 \cdot n^{(p-2)/2}C_{p/2}L(n)(Cp^2)^{4}4!n^{-2\delta_1}+(2p)^4 \cdot n^{4\delta} \cdot n^{(p-4)/2}[C_{p/2-2}+C_{p/2-1}L(n)(Cp^2)^{4}4!n^{-2\delta_1}] \leq\]
\[\leq n^{p/2-1}C_{p/2}L(n)(Cp^2)^{5}4!n^{-2\delta_1}+(2p)^4 \cdot n^{-4\delta_1} \cdot n^{p/2-1} \cdot 2C_{p/2-1}L(n)(Cp^2)^{4}4! \leq \]
\[\leq \frac{1}{4}n^{p/2-1} C_{p/2} L(n)(Cp^2)^{9} \cdot n^{-2\delta_1}+2 \cdot \frac{1}{8}n^{p/2-1} C_{p/2} L(n)(Cp^2)^{9} \cdot n^{-2\delta_1}.\]
To see this, notice that \(xx\) either does not appear in the merge, or it appears at least twice: the first term corresponds to the former (\((2p)^2\) accounts for choosing the first vertices and the orientations for the preimages, the bound rendered anew by Lemma~\ref{stilldom}, now employed on the two paths above that are not the copies of \(xx:\) the first term in it can be left out due to the merge not belonging to \(\mathcal{C}(p/2-1)\)), and the second to the latter by applying Lemma~\ref{stilldom} to two paths with first and last vertices \(u_1,\) and two additional copies of \(xx\) left out (the contribution of the edges to the expectation is accounted by \(n^{4\delta},\) and \((2p)^4\) accounts for gluing back the four copies of \(xx\) that were left out: for such configurations, the function \(\mathcal{M}\) is always at most \(\sum_{0 \leq t \leq 3}{C_{(l_1-1-t)/2}C_{(l_2-3+t)/2}} \leq C_{(l_1+l_2-4)/2+1}=C_{p/2-1}\)). This concludes the proof when \(t_0=L=l=2.\)
\par
Suppose now \(t_0>0.\) The base case for induction does not change since \(t_0=0\) when \(L=p\) (\(\mathcal{L}_{01}=\mathcal{L}_{02}=\emptyset \) as \(\mathcal{L}_{01} \cup \{a_s,b_s, 1 \leq s \leq t_1\} \subset \{k: 1 \leq k \leq L, l_k>1\}\)). \textit{Case 1.2} is impossible (\(t_0>0\)), and in \textit{Case 2.2}, select \(\rho_1,\rho_2\) such that \(\rho_1, \rho_2\) have endpoints \((u, u), (v, w),\) respectively, and denote by \(xy\) their (first) common edge; furthermore,
\begin{itemize}
    \item if possible, take \(vw \ne uu\) or \(xy \ne uu;\)
    \item if there are paths containing two
    copies of \(xy,\) choose \(\rho_1\) or \(\rho_2\) to be one of them; 
    \item when both conditions above fail, take \(\rho_1,\rho_2\) with the number of shared edges (counting multiplicities) maximal among such pairs (i.e., there are no \(\rho'_1,\rho'_2\) with
\(\rho'_1\) having equal endpoints, and \(\rho'_1,\rho'_2\) sharing more edges than \(\rho_1,\rho_2\) do), and the sum of their lengths maximal subject to the aforesaid constraint. 
\end{itemize}
%if there are paths containing two copies of \(xy\) among the original ones, choose \(\rho_1\) or \(\rho_2\) to be one of them, and else 
%assume also the number of shared edges between \(\rho_1\) and \(\rho_2\) (counting multiplicities) is maximal among such pairs (there are no \(\rho'_1,\rho'_2\) with
%\(\rho'_1\) having equal endpoints, and \(\rho'_1,\rho'_2\) sharing more edges than \(\rho_1,\rho_2\) do) as well as the sum of their lengths is maximal subject to the previous conditions. 
%\(\rho_1\) among the paths giving \(t_0,\) 
%Denote by \(xy\) the (first) common edge with \(\rho_2\) in \(\rho_1,\) and their endpoints \((u,u),(v,w),\) respectively. furthermore, if there are loops containing two copies of \(xy\) among the paths, choose \(\rho_1\) one of them. 
%denote the endpoints of \(\rho_1,\rho_2\) by \((u,u),(v,w),\) respectively. 
By deleting two copies of \(xy\) and gluing the four segments (with endpoints \((u,x),(u,y),(v,x),(w,y)\)) into two paths (with endpoints \((u,v), (u,w)\)), \(t_0\) decreases by at most \(2\) because all the loops underlying \(t_0,\) apart from \(\rho_1\) and \(\rho_2,\) continue to share some edge with another loop (the aforesaid deletion could reduce \(t_0\) by \(3\) when \(xy\) appears in solely one more cycle \(\rho_3\) and exactly once in \(\rho_1,\rho_2;\) this case is not possible because the multiplicity of \(xy\) is even, entailing \(\rho_3\) has at least two copies of \(xy,\) whereby so does \(\rho_1\) or \(\rho_2,\) given their definition, absurd). 
%either \(xy\) has multiplicity \(2,\) or   \(\rho_1,\rho_2\) are the only paths containing \(xy,\)  are yet two paths containing \(xy\)). 
%either \(\rho_1\) contains another copy of it or there are two other paths containing \(xy\) (its multiplicity is even: otherwise no contribution ensues). 
Call this new configuration of paths \(\rho',\) and see anew \(p,L,l,t_0,l_0\) as functions of the underlying configurations. 
\par
\((I') \hspace{0.2cm} v \ne w:\) %the remaining \(t_0-1\) cycles remain not edge disjoint with the rest. Hence for \(\rho',\) 
\(t_0\) can decrease by \(1,\) and \(l-L\) does not increase (\(l\) cannot increase, and \(L\) decreases by \(1\) when \((v,w)=(u,u'),\) and \(xy=uu''\) is a terminal edge, first or last, in \(\rho_1,\) and the first edge in \(\rho_2,\) in which case the nonempty path continues to be a path, leaving \(l-L\) unchanged), whereby \(p-L+l-t_0\) decreases by at least \(1,\) and the induction hypothesis yields the claim in the same fashion as before when \(l_0\) remains constant; suppose next \(l_0\) increases: if \(L\) does not change, use the induction hypothesis and that \(p-L+l-t_0\) drops by \(2;\) else \(\rho_1=(u,u',u),\rho_2 \in \{(v,w),(w,v)\},\) entailing all the other paths containing \(uu'=vw\) must have length \(1,\) from which the induction hypothesis on all paths apart from \(\rho_1\) yields the result (the decrease in \(p-L+l-t_0\) by \(1\) can be again used to account for the two additional copies of \(uu');\)
\par
\((II') \hspace{0.2cm} v=w, v \ne u:\) \(l\) decreases by \(2,\) \(t_0\) drops by at most \(2,\) \(L\) can only decrease by \(1,\) whereby \(p-L+l-t_0\) decreases by \(1,\) the induction hypothesis with the previous rationale being effective when \(l_0\) does not increase; the case in which \(l_0\) increases can be treated as in \((I');\)
\par
\((III') \hspace{0.2cm} v=w=u:\) denote by \(\rho_3\) the merge of \(\rho_1,\rho_2\) (using either (\ref{paste1}) or (\ref{paste2})), % in subsection~\ref{treigen1}),
and suppose \(\rho_1,\rho_2\) are the first two paths for the sake of notational simplicity. The case in which \(\rho_3\) has length \(1\) can be treated as in \((I'),\) and assume next its length is at least \(2.\)
\par
Suppose \(\rho_3\) shares no edge with the rest of the paths: when \(2|l_1+l_2+1,\) all expectations vanish by symmetry, and assume next \(2|l_1+l_2.\) If \(xy\) does not appear in the remaining paths, then the induction hypothesis and the case \(t=L=l=2\) analyzed above
%and the variance analysis in subsection~\ref{treigen1} 
can be used together with independence to obtain an upper bound of %the conclusion:
%\[n^{(l_1+l_2)/2-1} \mathcal{D}(l_1,l_2)(1+CL(n)n^{-2\delta_1}) \cdot n^{(p-L+l)/2-t_0/2-(l_1+l_2)/2+1} \cdot\]
%\[+C_{\lfloor p/2 \rfloor }L(n) \cdot (CC'p^2)^{5(2L-l)}\cdot (2L)!\cdot n^{-2\delta_1} \cdot \max_{\mathcal{L'}}{V(X,\mathcal{L}'_{02})}]\]
\[n^{(l_1+l_2)/2-1} [\mathcal{M}_{b_{11}}(2,((u_kv_k,l_k))_{1 \leq k \leq 2})+C_{(l_1+l_2)/2}(Cp^2)^{9}L(n)n^{-2\delta_1}] \cdot n^{(p-L+l-t_0)/2-(l_1+l_2)/2+1} \cdot\]
\[\cdot \sqrt{\mathbb{E}[b^{2l_{0}}_{11}]} \cdot [\mathcal{M}_{b_{11}}(t_0-2,((u_kv_k,l_k))_{3 \leq k \leq L})+C_{(p-l_1-l_2)/2}L(n)(CC'p^2)^{10(2L-l-2)}(2L-4)!n^{-2\delta_1} \cdot \max_{\mathcal{L'}}{V(X,\mathcal{L}'_{0})}],\]
this being due to \(p-L+l-t_0\) for \((u_kv_k,l_k)_{3 \leq k \leq L}\) decreasing by \(-l_1-l_2+2;\) the first-order terms remain included in \(\mathcal{M}_{b_{11}}(t_0,((u_kv_k,l_k))_{1 \leq k \leq L})\) (vertex-disjoint configurations underlying \(\mathcal{M}_{b_{11}}\) are closed under unions), while the second-order terms are at most, using (\ref{newm}) and (\ref{mmbound2}),
\[C'C_{(l_1+l_2)/2}(1+(Cp^2)^{9}L(n)n^{-2\delta_1}) \cdot C_{(p-l_1-l_2)/2}L(n)(CC'p^2)^{10(2L-l-2)}(2L-4)!n^{-2\delta_1}\cdot \max_{\mathcal{L'}}{V(X,\mathcal{L}'_{0})}+\]
\begin{equation}\label{boundxy}
    +C_{(l_1+l_2)/2} \cdot (Cp^2)^{9}L(n)n^{-2\delta_1} \cdot (1+(2L-2l)!! \chi_{L>l}) \cdot (C')^{L/2-1}C_{(p-l_1-l_2)/2} \cdot \max_{\mathcal{L}'}{V(X,\mathcal{L}'_{0})}.
\end{equation}
Else \(xy\) appears in some other path: when \(x \not \in \{u_k,v_k,1 \leq k \leq L\}\) or \(y \not \in \{u_k,v_k,1 \leq k \leq L\},\) a bound is the sum in (\ref{boundxy}) times \(n^{-1+2\delta}\) (an edge of multiplicity \(2k\) contributes \(n^{2\delta\max{(2k-4,0)}}\) to the individual bounds, and the choice of \(x\) or \(y\) is double counted). Take next \(x,y \in \{u_k,v_k,1 \leq k \leq L\}, x \ne y.\) Apply the induction hypothesis on the two paths with endpoints \(x,y\) formed by cutting out \(xy\) out of \(\rho_1,\rho_2,\) and \(((u_kv_k,l_k))_{3 \leq k \leq L}:\) since \(xy\) appears in some other path but not in \(\rho_3\) (a consequence of the definition of \(\rho_1,\rho_2\)), the two new paths have length larger than \(1,\) and thus for this new configuration \(L,l_0\) do not change, while \(l,p,t_0\) decrease by \(2,\) entailing \((p-L+l-t_0)/2\) decreases by \(1,\) which can counteract the left out copies of \(xy\) (generating once again an extra factor of \(n^{-1+2\delta}\) in (\ref{boundxy})).  Finally, suppose \(x,y \in \{u_k,v_k,1 \leq k \leq L\}, x=y.\) If \(x \ne u,\) then a factor of \(n^{-1+2\delta} \cdot 2L\) can be added to the sum in (\ref{boundxy}) due to the vertex \(x\) having solely \(2L\) possibilities in \(\rho_3\) instead of \(n.\) If \(x=u,\) then in light of the definition of \(\rho_1,\rho_2,\) any path among the original ones has at most one copy of \(uu=xy,\) with any two that do sharing no other edge apart from it and any such path having both endpoints \(u\) (else choosing a shared edge \(x'y' \ne uu\) or \(vw \ne uu\) would be possible, respectively). Denote by \(q \geq 2\) the number of paths containing \(uu:\) the induction hypothesis can be used on the remainder of the paths which have \(t_0-q\) loops of size at least \(2\) sharing an edge with the rest. For fixed positions of the \(t_1\) paths, denoted by \((n_j)_{1 \leq j \leq q},\) the exponent of \(n\) becomes 
\[(\sum_{1 \leq j \leq q}{l_{n_j}}-q)/2+
(p-\sum_{1 \leq j \leq q}{l_{n_j}}-(L-q)+(l-q)-(t_0-q))/2=(p-L+l-t_0)/2,\] 
the pairs formed with configurations underlying \(\mathcal{M}_{b_{11}}\) in each continue in this category when put together, while similarly to above the second-order terms yield via (\ref{mmbound2}) at most
\[\Pi_s
[(1+Cp^2L(n)n^{-2\delta_1})^q-1] \cdot C'^{(L-q)/2}C_{p_s/2}L(n)(CC'p^2)^{10(2L-l-q)}(2L-2q)!\cdot \max_{\mathcal{L'}}{V(X,\mathcal{L}'_{0})}+\]
\[+\Pi_s
(1+Cp^2L(n)n^{-2\delta_1})^q \cdot C_{p_s/2}L(n)(CC'p^2)^{10(2L-l-q)}(2L-2q)!n^{-2\delta_1}\cdot \max_{\mathcal{L'}}{V(X,\mathcal{L}'_{0})}.\]
for \(p_s=p-\sum_{1 \leq j \leq q}{l_{n_j}},\Pi_s=\prod_{1 \leq j \leq q}{(l_{n_j} \cdot C_{(l_{n_j}-1)/2} \cdot \chi_{2|l_{n_j}+1})}.\) 
Since
\[\prod_{1 \leq j \leq q}{l_{n_{j}}} \leq p^{q}, 
\hspace{0.5cm} q(1+Cp^2L(n)n^{-2\delta_1})^{q-1} \leq 4^q,\]
\[\sum_{q \geq 2, \{n_1,n_2,\hspace{0.05cm}...\hspace{0.05cm},n_q\} \subset \{1,2,\hspace{0.05cm}...\hspace{0.05cm},L\}}{4^qp^{-5q}C_{(l_{n_1}-1)/2}C_{(l_{n_2}-1)/2}...C_{(l_{n_q}-1)/2}C_{(p-\sum_{1 \leq j \leq q}{l_{n_j}})/2}} \leq\]
\[\leq \sum_{q \geq 2}{4^qp^{-4q}C_{(p+q)/2}} \leq C_{p/2}\sum_{q \geq 2}{p^{-4q}8^q} \leq C_{p/2} \cdot \frac{1}{2^2} \cdot 2=\frac{C_p/2}{2},\]
summing over \(\{n_1,n_2,\hspace{0.05cm}...\hspace{0.05cm},n_q\} \subset \{1,2,\hspace{0.05cm}...\hspace{0.05cm},L\}\) gives an upper bound of (amounting the second term in the last factor of (\ref{claimmmedd2}))
\[C_{p/2} \cdot Cp^2L(n)n^{-2\delta_1} \cdot C'^{L/2-1}L(n)(CC'p^2)^{10(2L-l)-4}(2L-2)!\cdot \max_{\mathcal{L'}}{V(X,\mathcal{L}'_{0})}+\]
\[+C_{p/2} \cdot L(n)(CC'p^2)^{10(2L-l)-4}(2L-2)!n^{-2\delta_1}\cdot \max_{\mathcal{L'}}{V(X,\mathcal{L}'_{0})},\]
which in turn completes the analysis in this case.
\par
Lastly, suppose \(\rho_3\) shares an edge with at least one of the remaining paths. Use the induction hypothesis on \(((u_kv_k,l_k))_{3 \leq k \leq L}\) and \(\rho_3\) (its length is \(l_1+l_2-2>1\)): \(t_0\) decreases by \(1,\) whereby so do \(p-L+l-t_0\) and \(2L-l\) (\(l,L\) decrease by \(1\)), and again a bound can be obtained from the induction hypothesis (the decrease in \(p-L+l-t_0\) can be used to account for the two deleted copies of \(xy,\) which contribute at most \(n^{2\delta}\)),
\[(2p)^2 \cdot n^{2\delta} \cdot n^{(p-L+l-t_0)/2-1/2} \cdot \sqrt{\mathbb{E}[b^{2l_0}_{11}]} \cdot [\mathcal{M}_{b_{11}}(t_0-1,((u_kv_k,l_k))_{3 \leq k \leq L+1})+\]
\[+C_{(p-2)/2}L(n)(CC'p^2)^{10(2L-l-1)}(2L-6)!n^{-2\delta_1} \cdot \max_{\mathcal{L'}}{V(X,\mathcal{L}'_{0})}]\]
as it will be shown next that the quantity (\ref{maxx}) for the new configurations is at most its counterpart for the original tuple: this together with all the previous bounds (including (\ref{mmbound2})) conclude the proof of the induction step since \(2L-l-2 \geq L-2 \geq 1\) in the aforesaid situation (\(L=2\) yields \(l=t_0=2,\) a situation already discussed). It remains to justify the claim on \(\max_{\mathcal{L'}}{V(X,\mathcal{L}'_{0})}:\) recall its definition (\ref{maxx}) along with that of \(V(\cdot,\cdot)\) above (\ref{vs}). 
Fix the tuple \(((u_kv_k,l_k))_{1 \leq k \leq L}\) and \(t_0 \geq 0:\) \(|V_{11}(v)|,o(v)\) do not change for \(v \ne \nu,\) and they can only decrease when \(v=u\) (the length of \(\rho_3\) shares parity with the sum of the lengths of \(\rho_1\) and \(\rho_2,\) and \(\rho_3\) has length larger than \(1\) with at most one of \(\rho_1,\rho_2\) has length \(1\)), entailing the desired inequality since \(t_0\) drops by \(1,\) \(\mathbb{E}[X^4] \geq (\mathbb{E}[X^2])^2=1,\) and the maximum over the tuples in (\ref{maxx}) is nonnegative. 
\end{proof}

%Lemmas (\ref{})-(\ref{}) are used in the coming subsection for the analysis of higher moments of \(e_i^TA^pe_j-\mathbb{E}[e_i^TA^pe_j],\) as well as mixed expectations,

\subsection{Higher Moments}\label{treigen22}

Return to (\ref{momi}) for \(l>2:\)
\[n^{l/2}\mathbb{E}[(e_1^TA^{p}e_1-\frac{1}{n}\mathbb{E}[tr(A^{p})])^l]=n^{-(p-1)l/2}\sum_{(\mathbf{i}_1,\mathbf{i}_2, \hspace{0.05cm} ... \hspace{0.05cm},\mathbf{i}_l)}{\mathbb{E}[(a_{\mathbf{i}_1}-\mathbb{E}[a_{\mathbf{i}_1}]) \cdot (a_{\mathbf{i}_2}-\mathbb{E}[a_{\mathbf{i}_2}]) \cdot ... \cdot (a_{\mathbf{i}_l}-\mathbb{E}[a_{\mathbf{i}_l}])]}.\]
The additional normalization for each factor is \(\sqrt{c(p,\mathbb{E}[a^4_{11}])} \geq C^{1/2}_{p-1},\) and \(\frac{1}{n}|\mathbb{E}[tr(A^p)]| \leq 2\chi_{2|p} \cdot C_{p/2},\) whereby 
\[\frac{\frac{1}{n}|\mathbb{E}[tr(A^p)]|}{\sqrt{c(p,\mathbb{E}[a^4_{11}])}} \leq \frac{2\chi_{2|p} \cdot C_{p/2}}{\sqrt{C_{p-1}}}=O(1)\]
using \(\lim_{s \to \infty}{\frac{s^{3/2}C_s}{4^s}}=\sqrt{\pi},\) a consequence of Stirling's formula; in light of the case \(l=2,\) this entails that to conclude (\ref{moments}) for \(l>2\) it suffices to show that connected components of \(\mathcal{G}\) with \(L>2\) vertices are negligible, i.e., the sum of their corresponding expectations is \(o(n^{(p-1)L/2}C^{-L/2}_{p-1})\) (see paragraph below (\ref{momi}) for the definition of the graph \(\mathcal{G}\)). %to conclude (\ref{moments}) for \(l>2.\)
%momi
\par
Take first \(2|p,\) and consider a connected component with \(L>1\) cycles, \(\mathbf{i}_1,\mathbf{i}_2, \hspace{0.05cm} ... \hspace{0.05cm},\mathbf{i}_L.\) Apply Lemma~\ref{bigcycles1} to 
\begin{equation}\label{bbbdef}
    B=(\mathbb{E}[a^2_{11}\chi_{|a_{11}| \leq n^{\delta}}])^{-1/2}n^{1/2}A_s=(\mathbb{E}[a^2_{11}\chi_{|a_{11}| \leq n^{\delta}}])^{-1/2}(a_{ij}\chi_{|a_{ij}|\leq n^{\delta}})_{1 \leq i,j \leq n},
\end{equation}
(\(L(n)=2C(\epsilon_0)\) suffices for \(n \geq n(\epsilon_0)\) since \(\mathbb{E}[a^2_{11}\chi_{|a_{11}| \leq n^{\delta}}] \in [1-n^{-2\delta},1]\)), and the even tuple formed by the endpoints of \(\mathbf{i}_1,\mathbf{i}_2, \hspace{0.05cm} ... \hspace{0.05cm},\mathbf{i}_L\) (i.e., \(u_k=v_k=1, 1 \leq k \leq L\)); in this situation, \(t_0=l=L,l_0=0,\) and the conclusion follows because the first order term in (\ref{claimmmedd2}) corresponds to configurations with pairs of cycles sharing no vertex apart from \(1,\) entailing connected components with \(L>2\) are negligible (the second term is negligible since \(\frac{C_{pL/2}}{C^{L/2}_{p-1}}=O(p^{3(L-1)/2}),\) and \(V(X,\mathcal{L}_{0}) \leq (\mathbb{E}[X^4])^{L/2} \leq (C(\epsilon_0))^{L/4},\) due to the observations above (\ref{claimmmedd2})).  
%when there are no paths of odd length from \((viii)\) are pairwise distinct among the pairs \((\{a_r,b_r\})_{1 \leq r \leq L/2}\)). This concludes \(2|p.\)\textcolor{blue}{\textbf{HERE}}
%entailing the common edges from \((viii)\) are pairwise distinct among the pairs \((\{a_r,b_r\})_{1 \leq r \leq L/2}\)). 
This concludes \(2|p.\)
\par
Some comments are in order for \(2|p+1:\) the order of magnitude of the moments in (\ref{momi}) can be still computed in the given range of \(p.\) Symmetry yields 
\[\mathbb{E}[(e_1^TA^pe_1)^l]=0\]
for \(2|l+1\) (the sum of the lengths of the paths \(\mathbf{i}_1,\mathbf{i}_2, \hspace{0.05cm} ... \hspace{0.05cm},\mathbf{i}_l\) is odd, \(pl\)). Take now the case \(2|l,\) and employ Lemma~\ref{bigcycles1} on the cycles underlying a cluster of size \(L.\) The leading configurations (i.e., the tuples behind \(\mathcal{M}_{b_{11}}\)) contribute a quantity that, after normalization, depends on all even moments of \(a_{11}.\) up to the \(l^{th}\) one. Namely, a cluster with \(L\) cycles for \(L>2,2|L\) has leading configurations given by \(L\) cycles, each an element of \(\mathcal{C}(\frac{p-1}{2})\) with a loop attached at some appearance of \(1\) in it; these give, after normalization,
\begin{equation}\label{2mcluster}
    \mathbb{E}[a^{L}_{11}\chi_{|a_{11}| \leq n^{\delta}}] \cdot \frac{(\sum_{1 \leq t \leq \frac{p-1}{2}}{tf_{\frac{p-1}{2},t+1})})^{L}}{(c(p,\mathbb{E}[a^{4}_{11}]))^{L/2}}=\mathbb{E}[a^{L}_{11}\chi_{|a_{11}| \leq n^{\delta}}] \cdot \frac{(s(\frac{p-1}{2}))^{2m}}{(c(p,\mathbb{E}[a^{4}_{11}]))^{L/2}},
\end{equation}
where the function \(s\) is defined by (\ref{formm}), and Lemma~\ref{lems} entails
this is 
\[O(\mathbb{E}[a^{L}_{11}\chi_{|a_{11}| \leq n^{\delta}}] \cdot (\frac{\tilde{C}C_{(p-1)/2}}{\sqrt{C_{p-1}}})^{L})=O(\mathbb{E}[a^{L}_{11}\chi_{|a_{11}| \leq n^{\delta}}] \cdot (\overline{C}p^{-3/4})^{L}).\]
%which in turn justifies the condition (\ref{pcondie}).
%\[\frac{\sum_{2 \leq t \leq \frac{p+1}{2}}{t(f_{\frac{p-1}{2}},t}-1)}{c(p,\mathbb{E}[a^{4}_{11}])}=O(\frac{C_{(p-1)/2}}{\sqrt{C_{p-1}}})=O(p^{-3/4}).\]
This shows all moments of \(a_{11}\) contribute to those of \(e_1^TA^{p}e_1\) when \(2|p+1:\) as \(p \to \infty,\) the higher moments can turn negligible if \(p^{3l/4}\) grows faster than \(\mathbb{E}[a^{2l}_{11}].\) %In particular, condition (\ref{pcondie}) suffices in this regard: one implication of this and \(p=n^{o(1)}\) is
%\begin{equation}\label{lcond}
%    \mathbb{E}[a_{11}^{2l}\chi_{|a_{11}| \leq n^{\delta}}]=n^{o_l(1)}.
%\end{equation}
More specifically, Lemma~\ref{bigcycles1} entails the leading term is given by graphs whose clusters have all clusters of size \(2,\) apart from at most one, which has the structure aforementioned. Thus, when \(l\) is fixed and \(2|l,\) the first-order term of the \(l^{th}\) moment is
%\[C^{-l/2}_{p-1}\sum_{k \leq l/2}{\binom{l}{2k} (2k-1)!!C^k_{p-1}C^{l-2k}_{(p-1)/2}\mathbb{E}[a^{l-2k}_{11}\chi_{|a_{11}| \leq n^{\delta}}]},\]
\[(1+o(1))\sum_{k \leq l/2}{\binom{l}{2k} (2k-1)!! \cdot \mathbb{E}[a^{l-2k}_{11}\chi_{|a_{11}| \leq n^{\delta}}]\cdot (\overline{C}p^{-3/4})^{l-2k}},\]
%When \(\lim_{n \to \infty}{(\log{p}-\frac{4\log{\mathbb{E}[a^{2l}_{11}\chi_{|a_{11}| \leq n^{\delta}}]}}{3l})}=\infty,\) 
with the term corresponding to \(k=l/2\) dominating when condition (\ref{pcondie}) holds:
\[\sum_{k \leq l/2}{\binom{l}{2k} (2k-1)!! \cdot \mathbb{E}[a^{l-2k}_{11}\chi_{|a_{11}| \leq n^{\delta}}]\cdot (\overline{C}p^{-3/4})^{l-2k}} \leq (l-1)!!+\sum_{k<l/2}{l^{2k} M^{-(l-2k)}} \leq (l-1)!!+l^{l} \cdot M^{-1}\]
where 
\[\log{M}=-\log{(\overline{C}p^{3/4})}-\max_{1 \leq t \leq l/2}{\frac{\log{\mathbb{E}[a^{2t}_{11}\chi_{|a_{11}| \leq n^{\delta}}]}}{2t}}=\log{\frac{2}{\overline{C}}}+\frac{3}{4}(\log{p}-\max_{1 \leq t \leq l/2}{\frac{2\log{\mathbb{E}[a^{2t}_{11}\chi_{|a_{11}| \leq n^{\delta}}]}}{3t}}) \to \infty.\] %for any fixed \(l.\) 
This provides the convergence in part \((a)\) of Theorem~\ref{th4} when \(2|p+1\)  (the second term in the bound remains negligible due to the moment assumption (\ref{pcondie}) and \(p=n^{o(1)}\)).

\subsection{Off-Diagonal Entries}\label{treigen3}

The objective is justifying (\ref{moments2}) for \(i \ne j:\) suppose without loss of generality that \(i=1,j=2.\) Consider first \(l=1,\) and note that any path \(\rho\) of length \(p\) between \(1\) and \(2\) has corresponding expectation \(0,\) i.e.,
\[\mathbb{E}[a_{1i_1}a_{i_1i_2}...a_{i_{p-1}2}]=0,\]
whereby
\[\mathbb{E}[e_1^TA^pe_2]=0\]
(there is an odd number of edges in \(\rho\) incident with \(i=1\) that are not \(ii:\) the copies of the latter can be excluded, the result remaining a path between \(1\) and \(2,\) one with no copy of \(ii,\) and whose number of edges incident with \(1\) then is \(2k-1\) for \(k\) the number of times \(i\) appears among its vertices). 
\par
The case \(l=2\) follows from merging each pair of paths \(\mathbf{i},\mathbf{j}\) into a cycle of length \(2p\) with first vertex \(1,\) and \((p+1)^{st}\) vertex \(2.\) There are \(C_p-\chi_{2|p}C^2_{p/2}\) first-order terms (i.e., elements of \(\mathcal{C}(p)\)), yielding the contribution 
\[n^{p-1}(C_p-\chi_{2|p}C^2_{p/2})(1+O(\frac{p^2}{n}))\]
(two vertices are already chosen, whereby the power of \(n\) is \(n^{p+1-2}\)), and for the second order, the trace estimate (\ref{subas}) gives their contribution is
\[O(C_pn^{p-1}L(n)p^2n^{-2\delta_1}).\]
This completes the result for \(l \leq 2,\) and take next \(l>2.\) When \(2|l+1,\) the result is immediate because
\[\mathbb{E}[(e^T_1A^pe_2)^{l}]=\sum_{(\mathbf{i}_1,\mathbf{i}_2, \hspace{0.05cm} ... \hspace{0.05cm},\mathbf{i}_l)}{\mathbb{E}[a_{\mathbf{i}_1}a_{\mathbf{i}_2}...a_{\mathbf{i}_l}]}=0:\]
the same rationale as for \(l=1\) yields there exists an edge incident with \(1\) in the union of \(\mathbf{i}_1,\mathbf{i}_2, \hspace{0.05cm} ... \hspace{0.05cm},\mathbf{i}_l\) of odd multiplicity (\(\mathbf{i}_1,\mathbf{i}_2, \hspace{0.05cm} ... \hspace{0.05cm},\mathbf{i}_l\) can be pasted into one path from \(1\) to \(2).\)
\par
%because in any path, there exists an edge in path of odd multiplicity, incident with either \(1\)) or \(2\))
The remaining case is \(2|l.\) Apply Lemma~\ref{stilldom} to \(B\) given by (\ref{bbbdef}), \(u_k=1,v_k=2,l_k=p,1 \leq k \leq l,\) whereby
\[\mathbb{E}[(e^T_1A^pe_2)^{l}] \leq n^{pl/2-l/2}\cdot  [\mathcal{M}(((u_k,v_k,l_k))_{1 \leq k \leq l})+C_{pl/2}\cdot(Cpl^2)^{2l}\cdot l!\cdot n^{-2\delta_1}]\]
and \(\mathcal{M}(((u_k,v_k,l_k))_{1 \leq k \leq l})=(l-1)!! \cdot C_p^{l/2}\) since the cycles underlying the contributors in \(\mathcal{M}\) consist of \(l/2\) loops, each formed with two paths, and there are \((l-1)!!\) possibilities for forming these pairs. The second term is of smaller order than the first because
\[C_p^{-l/2}C_{pl/2}\cdot(Cpl^2)^{2l}\cdot l!\cdot n^{-2\delta_1} \leq (Cp)^{3l/2}(Cpl^2)^{2l}\cdot l!\cdot n^{-2\delta_1} \leq (Clp)^{4l}n^{-2\delta_1}\]
employing \(\lim_{p \to \infty}{\frac{p^{3/2}C_p}{4^p}}=\sqrt{\pi}.\) Conversely, this bound is tight since the arrangements yielding \(\mathcal{M}\) contribute at least
\[n^{pl/2-l/2}\cdot  (l-1)!! \cdot (C_p-\chi_{2|p}C^2_{p/2})^{l/2}(\mathbb{E}[a^2_{11}\chi_{|a_{11}| \leq n^{\delta}}])^{pl/2}\prod_{1 \leq t \leq (p-1)l/2}{(1-\frac{t}{n})}\geq \]
\[\geq n^{pl/2-l/2}\cdot  (l-1)!! \cdot (C_p-\chi_{2|p}C^2_{p/2})^{l/2}(1-n^{-2\delta})^{pl/2} \exp(-\frac{p^2l^2}{4n}).\]
%using \((\mathbb{E}[a^2_{11}\chi_{|a_{11}| \leq n^{\delta}}])^{pl/2} \geq (1-n^{-2\delta})^{pl/2} \geq 1-n^{-2\delta}pl/2.\) 
This completes the proof of (\ref{moments2}), and thus of part \((b)\) in Theorem~\ref{th4}.

\section{Weak Convergence}\label{weakconvv}

This section contains the proof of Theorem~\ref{th55}, from which Corollary~\ref{cor1} and Theorem~\ref{th7} are inferred. The key towards (\ref{weakconv}) is Lemma~\ref{ll2}, quantifying \(\mu_{n,p} \approx \nu_{n,p}\) at the level of characteristic functions (subsection~\ref{s1}). Theorem~\ref{th5} builds on this and gives the mechanism behind its claim (subsection~\ref{s2}). Lastly, Corollary~\ref{cor1} and Theorem~\ref{th7} are justified (subsection~\ref{s3}), and delocalization with high probability under the Haar (probability) measure on the orthogonal group \(O(n)\) is discussed in more depth (subsection~\ref{subsechaar2}).
\par
In what follows, the subscripts of the two measures of interest are dropped for ease of notation: 
\[\mu=\mu_{n,p}, \hspace{0.5cm} \nu=\nu_{n,p},\] 
given by (\ref{munp}) and (\ref{nunp}), respectively, and \(\mathbb{E}_{\mu-\nu}[\cdot]=\mathbb{E}_{\mu}[\cdot]-\mathbb{E}_{\nu}[\cdot]\) denotes the difference between the expectations with respect to \(\mu\) and \(\nu,\) respectively.

\subsection{Characteristic Functions}\label{s1}

This subsection presents Lemma~\ref{ll2}, the backbone of Theorem~\ref{th55}. 

\begin{lemma}\label{ll2}
Let \(\theta \in \mathbb{R}^{m}.\) Under the assumptions of Theorem~\ref{th55}, there exist \(C(\epsilon_0), c_4, c_5(\epsilon_0), c_6(\epsilon_0), c_7>0\) such that if \(p \in \mathbb{N}, 2|p, ||\theta|| \leq c_4\sqrt{p}, p^{45} \leq c_5(\epsilon_0)\log{n},\) then
%for \(\mu=\mu_{n,p},\nu=\nu_{n,p}\) given by (\ref{munp}) and (\ref{nunp}), %\textcolor{blue}{check indices constants} \textcolor{brown}{get theta out of G(p,theta)?}
\begin{equation}\label{ineq1}
    |\mathbb{E}_{\mu}[e^{i\theta \cdot X}]-\mathbb{E}_{\nu}[e^{i\theta \cdot X}]| \leq (1+\sigma_0(\epsilon_0)\sum_{1 \leq i \leq n}{|\theta_{ii}|})^{C(\epsilon_0)p}\exp(c_6(\epsilon_0)p^{45}-2\delta_2(\epsilon_0)\log{n})+ \exp(-c_7p)%:=G(p,\theta),
\end{equation}
for \(\delta(\epsilon_0)\) given by (\ref{deltadef}), \(\delta_1(\epsilon_0)=\frac{1}{4}-\delta_1(\epsilon_0), \delta_2(\epsilon_0)=\min{(\delta(\epsilon_0),\delta_1(\epsilon_0))},\sigma_0(\epsilon_0)=\sup_{p \in \mathbb{N}}{\frac{c(p,C(\epsilon_0))}{C_p}}<\infty.\)
\end{lemma}

\textit{Remark:} \(\sigma_0(\epsilon_0)<\infty\) by virtue of \(\lim_{p \to \infty}{\frac{c(p,K)}{C_p}}=\frac{\sigma(\sigma+6)}{8}\) for all \(K \in \mathbb{R}.\)

\begin{proof} 
For \(T \in \mathbb{N}, x \in \mathbb{R}, |x| \leq \frac{T}{2e},\)
\[|e^{ix}-\sum_{m \leq T}{\frac{(ix)^{m}}{m!}}| \leq \sum_{m>T}{\frac{|x|^{m}}{m!}} \leq \sum_{m>T}{\frac{|x|^{m}}{(m/e)^m}} \leq 2 \cdot (\frac{e|x|}{T})^{T+1},\]
whereby
\[|\mathbb{E}_{\mu}[e^{i\theta \cdot X}]-\mathbb{E}_{\nu}[e^{i\theta \cdot X}]| \leq |\mathbb{E}_{\mu}[\sum_{j \leq T}{\frac{(i\theta \cdot X)^{j}}{j!}}]-\mathbb{E}_{\nu}[\sum_{j \leq T}{\frac{(i\theta \cdot X)^{j}}{j!}}]|+\]
\[+2\mathbb{E}_{\mu}[(\frac{e|\theta \cdot X|}{T})^{T+1}]+2\mathbb{E}_{\nu}[(\frac{e|\theta \cdot X|}{T})^{T+1}]+\mathbb{E}_{\mu}[|e^{i\theta \cdot X}-\sum_{j \leq T}{\frac{(i\theta \cdot X)^{j}}{j!}}| \cdot \chi_{|\theta \cdot X|>\frac{T}{2e}}]+\mathbb{E}_{\nu}[|e^{i\theta \cdot X}-\sum_{j \leq T}{\frac{(i\theta \cdot X)^{j}}{j!}}|\cdot  \chi_{|\theta \cdot X|>\frac{T}{2e}}]:=\]
\begin{equation}\label{errors}
    :=\mathcal{B}_1+\mathcal{B}_2+\mathcal{B}_3+\mathcal{B}_4+\mathcal{B}_5.
\end{equation}
Begin with the expectations with respect to \(\nu:\) Cauchy-Schwarz inequality, rotational invariance for normal distributions yielding 
\[\theta \cdot X\overset{d}{=} N(0,h(\theta)), \hspace{0.5cm} h(\theta)=\sum_{1 \leq i,j \leq n}{d_{ij} \theta^2_{ij}} \in [||\theta||^2, \sigma^2_0 ||\theta||^2]\] 
where \((d_{ij})_{1 \leq i,j \leq n}\) are defined by (\ref{dij}), and \(\mathbb{E}[Z^{2t}]=(2t-1)!!\) for \(t \in \mathbb{N},Z\overset{d}{=}N(0,1)\) entail 
%Cauchy-Schwarz inequality gives
\begin{equation}\label{b3bound}
    \mathcal{B}_3 \leq 2 \cdot (\frac{e \sigma_0 \cdot ||\theta||}{T})^{T+1} \cdot \sqrt{(2T+1)!!} \leq 2 \cdot (\frac{e\sigma_0 \cdot ||\theta||}{T})^{T+1} \cdot (\sqrt{2T+2})^{T+1} \leq 2 \cdot (\frac{2e\sigma_0 \cdot ||\theta||}{\sqrt{T}})^{T+1};
\end{equation}
for \(M \in \mathbb{N},\) Markov inequality and \(|e^{i\theta \cdot X}| \leq 1\) give
\[\frac{\mathcal{B}_5}{2} \leq \mathbb{E}_{\nu}[\sum_{j \leq T}{\frac{(|\theta \cdot X|)^{j}}{j!}}\cdot  \chi_{|\theta \cdot X|>\frac{T}{2e}}]  \leq (\frac{2e}{T})^M \cdot \mathbb{E}_{\nu}[\sum_{j \leq T}{\frac{(|\theta \cdot X|)^{j+M}}{j!}}] \leq (\frac{2e \sigma_0 ||\theta||}{T})^M\sum_{j \leq T}{\frac{(\sigma_0||\theta||)^{j}\sqrt{(2j+2M+1)!!}}{j!}} \leq\]
\[\leq (\frac{4e\sigma_0||\theta||}{T})^M\sum_{j \leq T}{\frac{ (\sigma_0||\theta||)^{j} \cdot 2^j\sqrt{(j+M)!}}{j!}} \leq (\frac{4e\sigma_0 ||\theta||}{T})^M\sum_{j \leq T}{\frac{ (\sigma_0||\theta||)^{j} \cdot 2^j(j+M)^{(j+M)/2}}{j!}} \leq\]
\begin{equation}\label{b5eq}
    \leq (\frac{4e\sigma_0||\theta|| \sqrt{T+M}}{T})^M \sum_{j \leq T}{\frac{(\sigma_0||\theta||)^{j} \cdot 2^j(j+M)^{j/2}}{j!}} \leq (\frac{4e\sigma_0||\theta|| \sqrt{T+M}}{T})^M \cdot \exp(2\sigma_0||\theta|| \cdot \sqrt{T+M}),
\end{equation}
employing \((2t+1)!!=3 \cdot 5 \cdot ... \cdot (2t+1) \leq  4 \cdot 6 \cdot ... \cdot (2t+2)= 2^t(t+1)!.\)
\par
Consider next the mixed term, \(\mathcal{B}_1:\) 
%\[\mathcal{B}_1=\mathbb{E}_{\mu-\nu}[\sum_{j \leq T}{\frac{(i\theta \cdot X)^{j}}{j!}}]=\sum_{\sum_{l \leq m}{r_l}=r,r \leq T}{\frac{i^r}{r!} \cdot \binom{r}{r_1,r_2,\hspace{0.05cm}...\hspace{0.05cm},r_m} \cdot \theta^{r_1}_{l_1}\theta^{r_2}_{l_2}...\theta^{r_m}_{l_m}} \cdot \mathbb{E}_{\mu-\nu}[X^{r_1}_{l_1}X^{r_2}_{l_2}...X^{r_m}_{l_m}].\]
\[\mathcal{B}_1=\mathbb{E}_{\mu-\nu}[\sum_{j \leq T}{\frac{(i\theta \cdot X)^{j}}{j!}}]=\sum_{\sum_{1 \leq u \leq v \leq n}{r_{uv}}=r \leq T}{\frac{i^r}{r!} \cdot \binom{r}{(r_{uv})_{1 \leq u \leq v \leq n}} \cdot \prod_{1 \leq u \leq v \leq n}{\theta^{r_{uv}}_{uv}} \cdot \mathbb{E}_{\mu-\nu}[\prod_{1 \leq u \leq v \leq n}{X^{r_{uv}}_{uv}}]}.\]
%Let %\(g(l_k)=(u_k,v_k),\) where 
%\(g:\{1,2, \hspace{0.05cm} ... \hspace{0.05cm},m\} \to \{(i,j): 1 \leq i \leq j \leq n\},g(l)=(u(l),v(l))\) be the bijection underlying %the notation
%\[\mathbb{R}^m=\{(a_{ij})_{1 \leq i \leq j \leq n}^T, a_{ij} \in \mathbb{R}\};\] 
A consequence of Lemma~\ref{bigcycles1} is that 
\begin{equation}\label{munu}
    |\mathbb{E}_{\mu-\nu}[\prod_{1 \leq u \leq v \leq n}{X^{r_{uv}}_{uv}}]| \leq \sigma_0^{\sum_{1 \leq u \leq n}{r_{uu}}} \cdot C^{-r/2}_{p-1}C_{pr/2}[C''prn^{-2\delta}+2^r(CC'C(\epsilon_0)p^2)^{20r}(2r)!n^{-2\delta_1}].
\end{equation}
To see this, fix a tuple \((r_{uv})_{1 \leq u \leq v \leq n}\) with \(r:=\sum_{1 \leq u \leq v \leq n}{r_{uv}},\) and apply Lemma~\ref{bigcycles1} to the \(r\) paths with endpoints
\((uv)_{(u,v): r_{uv}>0},\) \(l_k=p\) for \(1 \leq k \leq r,\) \(t_0=\sum_{1 \leq u \leq n}{r_{uu}}\) (when \(t_0<\sum_{1 \leq u \leq n}{r_{uu}},\) the expectation vanishes by independence). The induction hypothesis in the proof of the lemma  entails the first term in the bound comes from configurations in which each path either belongs to \(\cup_{q \geq 1}{\mathcal{C}(q)},\) or is paired with another path that shares its endpoints, their merge belonging to \(\cup_{q \geq 1}{\mathcal{C}(q)}.\) This entails that for such paths \(\mathbf{p}_1,\mathbf{p}_2,\hspace{0.05cm}...\hspace{0.05cm},\mathbf{p}_r,\)
\[\mathbb{E}[(a_{\mathbf{p}_1}-\mathbb{E}[a_{\mathbf{p}_1}]) \cdot (a_{\mathbf{p}_2}-\mathbb{E}[a_{\mathbf{p}_2}]) \cdot ... \cdot (a_{\mathbf{p}_r}-\mathbb{E}[a_{\mathbf{p}_r}])]=0\]
when there exist \(1 \leq u_0 \leq v_0 \leq n\) with \(2|r_{u_0v_0}+1\) (there is a path \(\mathbf{p}'\) with endpoints \(u_0,v_0\) that is not paired and thus, \(a_{\mathbf{p}'}\) is independent of \((a_{\mathbf{p}})_{p\ne p'},\) yielding a vanishing expectation). Otherwise, the contributions are
\[(1+O(\frac{p^2r^2}{n}))(\mathbb{E}[a^2_{11}\chi_{|a_{11}| \leq n^{\delta}}])^{pr/2}\prod_{1 \leq u < v \leq n}{d_{uv}^{r_{uv}/2}(r_{uv}-1)!!},\]
entailing together with \(\mathbb{E}[a^2_{11}\chi_{|a_{11}| \leq n^{\delta}}] \in [1-n^{-2\delta},1]\) that these maximal configurations contribute to \(|\mathbb{E}_{\mu-\nu}|\)
\[O(\frac{p^2r^2}{n}+prn^{-2\delta})\prod_{1 \leq u < v \leq n}{d_{uv}^{r_{uv}/2}(r_{uv}-1)!!}=O(prn^{-2\delta})\prod_{1 \leq u < v \leq n}{d_{uv}^{r_{uv}/2}(r_{uv}-1)!!}.\]
For the remaining tuples of paths, Lemma~\ref{bigcycles1} gives
\[\sum_{(\mathbf{p}_1,\mathbf{p}_2,\hspace{0.05cm}...\hspace{0.05cm},\mathbf{p}_r)}{\mathbb{E}[a_{\mathbf{p}_1} \cdot a_{\mathbf{p}_2} \cdot ... \cdot a_{\mathbf{p}_r}]} \leq n^{(pr-r)/2}C_{pr/2}L(n)(CC'p^2)^{20r}(2r)!n^{-2\delta_1} \cdot (C(\epsilon_0))^r\]
%since \(l_0=0,L=r,l=\sum_{1 \leq u<v \leq n}{r_{uv}},t_0=\sum_{1 \leq u \leq n}{r_{uu}},\) and \(|V_{11}(v)|=|V_{1}(v)|=0\) for all \(v\) (\(|V_{1}(v)|>0\) generates paths of odd length, but all have length \(p\)) renders \(\max_{\mathcal{L}}{V(X,\mathcal{L}_0)} \leq (\mathbb{E}[a^4_{11}])^{r} \leq (C(\epsilon_0))^r.\) Lastly,
since \(l_0=0,L=r,l=\sum_{1 \leq u<v \leq n}{r_{uv}},t_0=\sum_{1 \leq u \leq n}{r_{uu}},\) and \(|V_{11}(v)|=o(v)=0,\) and \(\max_{\mathcal{L}'}{V(X,\mathcal{L}'_0)} \leq (\mathbb{E}[a^4_{11}])^{r} \leq (C(\epsilon_0))^r.\) Lastly,
\[|\mathbb{E}[(a_{\mathbf{p}_1}-\mathbb{E}[a_{\mathbf{p}_1}]) \cdot (a_{\mathbf{p}_2}-\mathbb{E}[a_{\mathbf{p}_2}]) \cdot ... \cdot (a_{\mathbf{p}_r}-\mathbb{E}[a_{\mathbf{p}_r}])]| \leq 2^{r-1} \cdot \mathbb{E}[a_{\mathbf{p}_1} \cdot a_{\mathbf{p}_2} \cdot ... \cdot a_{\mathbf{p}_r}]\]
can be used to conclude (\ref{munu}): this inequality holds because
\[|\mathbb{E}[(a_{\mathbf{p}_1}-\mathbb{E}[a_{\mathbf{p}_1}]) \cdot (a_{\mathbf{p}_2}-\mathbb{E}[a_{\mathbf{p}_2}]) \cdot ... \cdot (a_{\mathbf{p}_r}-\mathbb{E}[a_{\mathbf{p}_r}])]|=|\sum_{S \subset \{1,2,\hspace{0.05cm}...\hspace{0.05cm},L\}}{(-1)^{L-|S|}\prod_{s \in S}{\mathbb{E}[a_{\mathbf{p}_s}]} \cdot \mathbb{E}[\prod_{s \not \in S}{a_{\mathbf{p}_s}}]}| \leq\]
\[\leq 2^{r-1}\max_{S \subset \{1,2,\hspace{0.05cm}...\hspace{0.05cm},L\}}{\prod_{s \in S}{\mathbb{E}[a_{\mathbf{p}_s}]} \cdot \mathbb{E}[\prod_{s \not \in S}{a_{\mathbf{p}_s}}]} \leq 2^{r-1} \cdot  \mathbb{E}[a_{\mathbf{p}_1} \cdot a_{\mathbf{p}_2} \cdot ... \cdot a_{\mathbf{p}_r}]\]
due to all expectations being nonnegative, independence, and for \(Y \geq 0, s_0=\sum_{s \in S_1}{s}+\sum_{s \in S_2}{s},\) 
\[(\prod_{s \in S_1}{\mathbb{E}[Y^{s}]}) \cdot \mathbb{E}[\prod_{s \in S_2}{Y^s}] \leq \prod_{s \in S_1}{(\mathbb{E}[Y^{s_0}])^{s/s_0}} \cdot (\mathbb{E}[Y^{s_0}])^{(\sum_{s \in S_2}{s})/s_0}=\mathbb{E}[Y^{s_0}],\]
via Hölder's inequality (\(S_1,S_2 \subset \{1,2,\hspace{0.05cm}...\hspace{0.05cm},L\}\)). Lastly, (\ref{munu}), Lemma~\ref{lemhelp}, and \(C^{-r/2}_{p-1}C_{pr/2} \leq p^{3r/4}\) yield %for \(\delta_2=\min{(\delta,\delta_1)},\)
\begin{equation}\label{btt}
    \mathcal{B}_1 \leq (\sum_{1 \leq i \leq n}{|\theta_{ii}|})^T \cdot n^{-2\delta_2} \exp(C'(\epsilon_0) \cdot ||\theta||^2p^{42}T^2):=\mathcal{B}(T).
\end{equation}
\par
Consider the two remaining errors in (\ref{errors}): \(\mathcal{B}_2/2\) can be absorbed by the sum of the upper bounds for \(\mathcal{B}_3\) and \(\mathcal{B}(T+1),\) upon multiplying them by \(\tilde{C}\sqrt{T}\) (choose \(2|T+1\)) since
\[e^{T+1} \cdot (\frac{|\theta \cdot X|}{T})^{T+1}= \frac{(\theta \cdot X)^{T+1}}{(T+1)!} \cdot \frac{e^{T+1} \cdot (T+1)!}{T^{T+1}} \leq \tilde{C}\sqrt{T} \cdot \frac{(\theta \cdot X)^{T+1}}{(T+1)!},\]
and similarly a bound for \(\mathcal{B}_4\) can be obtained via Markov's inequality, (\ref{btt}), and (\ref{b5eq}): %use a similar rationale as for \(\mathcal{B}_5.\) Use Markov's inequality for the first component: 
when \(k \in \mathbb{N},\)
\[\frac{\mathcal{B}_4}{2} \leq \mathbb{E}_{\mu}[\sum_{j \leq T}{\frac{(|\theta \cdot X|)^{j}}{j!}}\cdot  \chi_{|\theta \cdot X|>\frac{T}{2e}}]  \leq (\frac{2e}{T})^{2k} \cdot \mathbb{E}_{\mu}[\sum_{j \leq T}{\frac{(|\theta \cdot X|)^{j+2k}}{j!}}] \leq (\frac{2e}{T})^{2k}\cdot (\mathbb{E}_{\nu}[\sum_{j \leq T}{\frac{(|\theta \cdot X|)^{j+2k}}{j!}}]+\mathcal{B}(T+2k)).\]
%\[\mathbb{P}_{\mu}(|\theta \cdot X|>\frac{T}{2e}) \leq \frac{\mathbb{E}_{\mu}[(\theta \cdot X)^{2k}]}{(T/2e)^{2k}} \leq \frac{\mathbb{E}_{\nu}[(\theta \cdot X)^{2k}]+n^{-2\delta_2}(\exp(c||\theta||^2p^4k^2)-1)}{(T/2e)^{2k}}=\]
%\[\mathbb{P}_{\mu}(|\theta \cdot X|>\frac{T}{2e}) \leq \frac{\mathbb{E}_{\mu}[(\theta \cdot X)^{2k}]}{(T/2e)^{2k}} \leq \frac{\mathbb{E}_{\nu}[(\theta \cdot X)^{2k}]+(2k)! \cdot \mathcal{B}(2k)}{(T/2e)^{2k}} \leq \frac{(\sigma_0||\theta||)^{2k} \cdot (2k-1)!!+(2k)! \cdot \mathcal{B}(2k)}{(T/2e)^{2k}}.\]
%\[=\frac{||\theta||^{2k} \cdot (2k-1)!!+n^{-2\delta_2}(\exp(c||\theta||^2p^4k^2)-1)}{(T/2e)^{2k}}.\]
Gathering all the inequalities gives for \(T=2T_0-1,T_0,k \in \mathbb{N},M>0,\)
%\[|\mathbb{E}_{\mu}[e^{i\theta \cdot X}]-\mathbb{E}_{\nu}[e^{i\theta \cdot X}]| \leq 2n^{-2\delta_2}(\sum_{1 \leq i \leq n}{|\theta_{ii}|})^T(\exp(c||\theta||^2p^{40}T^2)-1)+2 \cdot (\frac{2||\theta||}{\sqrt{T}})^{T+1}+\]
%\[+\frac{2||\theta||^{2k} \cdot (2k-1)!!+n^{-2\delta_2}(\exp(c||\theta||^2p^{40}k^2)-1)}{(T/2e)^{2k}}+(\frac{4e||\theta|| \sqrt{T+M}}{T})^M \cdot \exp(2||\theta|| \cdot \sqrt{T+M}).\]
\[|\mathbb{E}_{\mu}[e^{i\theta \cdot X}]-\mathbb{E}_{\nu}[e^{i\theta \cdot X}]| \leq (1+2\tilde{C}T)n^{-2\delta_2}(1+\sigma_0\sum_{1 \leq i \leq n}{|\theta_{ii}|})^{T+1}\exp(C'(\epsilon_0)||\theta||^2p^{42}(T+1)^2)+4 \cdot (\frac{2e\sigma_0||\theta||}{\sqrt{T}})^{T+1}+\]
%\[|\mathbb{E}_{\mu}[e^{i\theta \cdot X}]-\mathbb{E}_{\nu}[e^{i\theta \cdot X}]| \leq (1+2\tilde{C}T) \cdot \mathcal{B}(T+1)+4 \cdot (\frac{2e\sigma_0||\theta||}{\sqrt{T}})^{T+1}+\]
%\[+\frac{2||\theta||^{2k} \cdot (2k-1)!!+(2k)! \cdot \mathcal{B}(k)}{(T/2e)^{2k}}+(\frac{4e||\theta|| \sqrt{T+M}}{T})^M \cdot \exp(2||\theta|| \cdot \sqrt{T+M}).\]
%\[+\frac{(\sigma_0||\theta||)^{2k} \cdot (2k-1)!!+(2k)! \cdot \mathcal{B}(2k)}{(T/2e)^{2k}}+ 4(\frac{4e\sigma_0||\theta|| \sqrt{T+M}}{T})^M \cdot \exp(2\sigma_0||\theta|| \cdot \sqrt{T+M})\]
\[+(\frac{2e}{T})^{2k}\cdot (\mathbb{E}_{\nu}[\sum_{j \leq T}{\frac{(|\theta \cdot X|)^{j+2k}}{j!}}]+\mathcal{B}(T+2k))+ 2(\frac{4e\sigma_0||\theta|| \sqrt{T+M}}{T})^M \cdot \exp(2\sigma_0||\theta|| \cdot \sqrt{T+M}).\]
%Since \(||\theta|| \leq c_4\sqrt{p},p^{43} \leq c_5\log{n},\) let \(T=\lfloor C'(||\theta||^2+p) \rfloor,\) \(M=\frac{c'T^2}{||\theta||^2+1}, k= T\) for \(c',C'>0:\) this makes the terms above \(\exp(C''p^{43}+C''||\theta||^6p^2-\log{n}), \exp(-c''T), \exp(-c''T),\exp(-c'' p)\) and renders the conclusion of the lemma. \textcolor{red}{FINALIZE COMPUTATIONS}
Since \(||\theta|| \leq c_4\sqrt{p},p^{45} \leq c_5\log{n},\) let \(T=\lfloor C''(\epsilon_0)(||\theta||^2+p) \rfloor,\) \(M=\frac{c'(\epsilon_0)T^2}{||\theta||^2+1}, k=\lfloor \frac{M}{2}\rfloor:\) this makes the terms above 
\(\exp(C'''(\epsilon_0)p^{45}-2\delta_2\log{n}), 4\exp(-c''T), 4\exp(-c'' p)+\exp(C'''(\epsilon_0)p^{45}-2\delta_2\log{n}),4\exp(-c'' p)\) and renders the conclusion of the lemma. 
\end{proof}

\begin{lemma}\label{lemhelp}
    Let \(n,k \in \mathbb{N}, m=\frac{n(n+1)}{2}, \theta \in \mathbb{R}^m=\{(x_{ij})_{1 \leq i \leq j \leq n}\},\) \(\mathcal{G}\) a directed multigraph on \(\{1,2,\hspace{0.05cm}...\hspace{0.05cm},k\}\) with edges \(((u_q,v_q))_{1 \leq q \leq r}\) belonging to \(\{(i,j):1 \leq i \leq j \leq k\},\) and \((u_qv_q)_{1 \leq q \leq r}\) an even tuple. Then
    \[\sum_{\rho \in S(\mathcal{G},n)}{\theta_{\rho(u_1)\rho(v_1)}\theta_{\rho(u_2)\rho(v_2)}...\theta_{\rho(u_r)\rho(v_r)}} \leq (\sum_{1 \leq i \leq n}{|\theta_{ii}|})^{k_0} \cdot (\sqrt{2}||\theta||)^{r-k_0},\]
    where \(S(\mathcal{G},n)=\{\rho| \rho:\{1,2,\hspace{0.05cm}...\hspace{0.05cm},k\} \to \{1,2,\hspace{0.05cm}...\hspace{0.05cm},n\}, \rho(u_j) \leq \rho(v_j), 1 \leq j \leq k\},\) and 
    \[k_0=|\{j: 1\leq j \leq r, u_j=v_j, u_j \not \in \{u_q,v_q, 1 \leq q \leq r, q \ne j\}\}|.\] 
\end{lemma}

\begin{proof}
    Suppose without loss of generality that \(\min_{1 \leq i \leq j \leq n}{\theta_{ij}} \geq 0, k_0=0\) %Since 
    %\[\sum_{1 \leq i \leq n}{\theta_{ii}} \leq \sqrt{n\sum_{1 \leq i \leq n}{\theta^2_{ii}} } \leq n^{1/2}||\theta||,\]
    (the edges underlying \(k_0\) form an even tuple). Let \(M=n^2,\tau \in \mathbb{R}^{M}=\{(x_{ij})_{1 \leq i,j \leq n}\}\) be given by \(\tau_{ij}=\theta_{i_0j_0}\) for 
    \[i_0=\min{(i,j)}, \hspace{0.2cm} j_0=\max{(i,j)}, \hspace{0.2cm}  1 \leq i,j \leq n.\] 
    The undirected edges of \(\mathcal{G}\) can be partitioned into cycles, a property that was proved by induction on \(r\) in Lemma~\ref{bigcycles} (see the first paragraph in which the gluing of the paths is described).  %\textcolor{blue}{The base case \(r=1\) follows from \(u_1=v_1,\) and take next \(r \geq 2;\) if \(u_1=v_1,\) then use the induction hypothesis on \(u_2v_2, \hspace{0.05cm}...\hspace{0.05cm},u_rv_r;\) else, there is \(q\) with \(2 \leq q \leq r\) and \(u_qv_q\) incident with \(u_1v_1;\) suppose without loss of generality that \(k=2,\) and let \((uw_1,uw_2)=(u_1v_1,u_2v_2);\) the induction hypothesis can be applied to \(w_1w_2,u_3v_3,u_4v_4, \hspace{0.05cm}...\hspace{0.05cm},u_rv_r,\) whereby the result ensues by replacing \(w_1w_2\) in its corresponding cycle with \(w_1uw_2.\)} 
    Fix such a decomposition with cycles \(\mathcal{L}_1=(w_{1q})_{1 \leq q \leq m_1},\mathcal{L}_2=(w_{2q})_{1 \leq q \leq m_2},\hspace{0.05cm}...\hspace{0.05cm},\mathcal{L}_s=(w_{sq})_{1 \leq q \leq m_s},\) %arbitrarily directed (i.e., label in each of them the vertices such that the edges are a set of the type \((w_1,w_2),(w_2,w_3), \hspace{0.05cm}...\hspace{0.05cm},(w_l,w_1)\)), 
    and \(s\) minimal: %and set arbitrary orientations for each of them (label their vertices: %, and additionally assign labels to the vertices of each cycle (
    it suffices to show that for \(y \geq 2,\)
    \begin{equation}\label{taures}
        \sum_{(i_1,i_2,\hspace{0.05cm}...\hspace{0.05cm},i_y) \in \{1,2,\hspace{0.05cm}...\hspace{0.05cm},n\}^y}{\tau_{i_1i_2}\tau_{i_2i_3}...\tau_{i_{y}i_1}} \leq ||\tau||^y
    \end{equation}
    because each cycle has size at least \(2\) (\((\mathcal{L}_q)_{1 \leq q \leq s}\) are vertex disjoint by the minimality of \(s,\) and \(k_0=0\)), \(||\tau|| \leq \sqrt{2}||\theta||,\)
    \[\sum_{\rho \in S(\mathcal{G},n)}{\theta_{\rho(u_1)\rho(v_1)}\theta_{\rho(u_2)\rho(v_2)}...\theta_{\rho(u_r)\rho(v_r)}}=\sum_{\rho \in S(\mathcal{G},n)}{\tau^{q(\mathcal{G},\rho(1))}_{1}...\tau^{q(\mathcal{G},\rho(M))}_M},\]
    where for each \((u,v), u \leq v\) underlying the positions on the left-hand side, the corresponding exponents on the right-hand side for \((u,v),(v,u)\) are given by their multiplicities in the multiset \(((w_{tq},w_{t(q+1)}))_{1 \leq t \leq s,1 \leq q \leq m_t},\) respectively, where \(w_{t(m_t+1)}:=w_{t1}\) for \(1 \leq t \leq s,\) and
    \[\sum_{\rho \in S(\mathcal{G},n)}{\tau^{q(\mathcal{G},\rho(1))}_{1}...\tau^{q(\mathcal{G},\rho(M))}_M}=\prod_{1 \leq q \leq s}{\sum_{\rho \in S(\mathcal{L}_q,n)}{\tau^{q(\mathcal{L}_q,\rho(1))}_{1}...\tau^{q(\mathcal{L}_q,\rho(M))}_M}}\]
    due to \((\mathcal{L}_q)_{1 \leq q \leq s}\) being vertex-disjoint.% (by the minimality of \(s\)).
    \par
    Consider now (\ref{taures}): the case \(y=2\) is clear, and assume next \(y \geq 3.\) Cauchy-Schwarz inequality yields
    \begin{equation}\label{cs}
      (\sum_{(i_1,i_2,\hspace{0.05cm} ...\hspace{0.05cm},i_y)}{\tau_{i_1i_2}\tau_{i_2i_3}...\tau_{i_{y}i_1}})^2 \leq (\sum_{(i_1,i_2)}{\tau^2_{i_1i_2}})(\sum_{(i_1,i_2)}{(\sum_{(i_3,\hspace{0.05cm} ...\hspace{0.05cm},i_y)}{\tau_{i_2i_3}\tau_{i_3i_4}...\tau_{i_{y}i_1}})^2}).  
    \end{equation}
    Justifying for \(l \geq 1,1 \leq a,b \leq n,\)
    \begin{equation}\label{taures2}
        (\sum_{(i_1,\hspace{0.05cm} ...\hspace{0.05cm},i_l)}{\tau_{ai_1}\tau_{i_1i_2}...\tau_{i_{l}b}})^2 \leq ||\tau||^{2l-2}(\sum_{i}{\tau^2_{ai}})(\sum_{i}{\tau^2_{ib}})
    \end{equation}
    is enough insomuch as together with (\ref{cs}) it will give
    \[\sum_{(i_1,i_2,\hspace{0.05cm} ...\hspace{0.05cm},i_y)}{\tau_{i_1i_2}\tau_{i_2i_3}...\tau_{i_{y}i_1}} \leq \sqrt{||\tau||^2 \cdot ||\tau||^{2y-6} \sum_{(i_1,i_2)}{(\sum_{i}{\tau^2_{i_2i}})(\sum_{i}{\tau^2_{ii_1}})}}=||\tau||^{y}.\]
    The desired inequality (\ref{taures2}) 
    follows easily by induction on \(l:\) the base case \(l=1\) is a direct consequence of Cauchy-Schwarz inequality, and the induction step for \(l \geq 2\) ensues from the aforesaid inequality and the induction hypothesis:
    \[(\sum_{(i_1,\hspace{0.05cm} ...\hspace{0.05cm},i_l)}{\tau_{ai_1}\tau_{i_1i_2}...\tau_{i_{l}b}})^2 \leq (\sum_{i_1}{\tau^2_{ai_1}})(\sum_{i_1}(\sum_{(i_2,\hspace{0.05cm} ...\hspace{0.05cm},i_l)}{\tau_{i_1i_2}...\tau_{i_{l}b}})^2) \leq \]
    \[\leq (\sum_{i_1}{\tau^2_{ai_1}})\sum_{i_1}{||\tau||^{2l-4}(\sum_{j}{\tau^2_{i_1j}})(\sum_{j}{\tau^2_{jb}})}=||\tau||^{2l-2}(\sum_{i}{\tau^2_{ai}})(\sum_{i}{\tau^2_{ib}}).\]
\end{proof}

\subsection{Pointwise Errors}\label{s2}

Return to the pseudo-distance \(d_{n}:\) this subsection contains the proof of Theorem~\ref{th5}, which entails Theorem~\ref{th55}.

\begin{theorem}\label{th5}
    Suppose \(f:\mathbb{R}^m \to \mathbb{R}_{\geq 0}\) is continuous, \(supp(f) \subset B_m(R):=\{x \in \mathbb{R}^m, ||x|| \leq R\}.\)  
    Then there exist \(c_8(\epsilon_0),c_9(\epsilon_0)>0\) such that for \(p \in \mathbb{N}, p \leq c_9(\epsilon_0)(\log{n})^{1/45},\) %\textcolor{brown}{\textbf{HERE}}
    \begin{equation}\label{conclusion}
        |\int{fd\mu}-\int{fd\nu}| \leq (\frac{R^2}{c_8(\epsilon_0)mp})^{m/2}\cdot n^{p/c_8(\epsilon_0)} \cdot ||f||_{\infty}.
    \end{equation}
    %where \(\tilde{R}=R+\sqrt{m}.\)
\end{theorem}

\begin{proof}
Let
\[G:\mathbb{R} \to [0,\infty), \hspace{0.5cm} G(x)=ce^{-\frac{1}{x(1-x)}}\chi_{x \in (0,1)}, \hspace{0.5cm} c=(\int_{0}^{1}{e^{-\frac{1}{x(1-x)}}dx})^{-1} \in (1,\infty),\]
\[\tau:\mathbb{R}^m \to [0,\infty), \hspace{0.5cm} \tau(x)=\prod_{1 \leq i \leq m}{G(x_i)}:\]
\(G\) is smooth since by induction on \(k \geq 0,\) \(G^{(k)}(x)=ce^{-\frac{1}{x(1-x)}}P_k(x,\frac{1}{x},\frac{1}{1-x})\chi_{x \in (0,1)}\) for some \(P_k \in \mathbb{R}[X,Y,Z]\) (using \(\frac{1}{x(1-x)}=\frac{1}{x}+\frac{1}{1-x}),\) and thus \(\tau\) is smooth with \(\int_{\mathbb{R}^m}{\tau(x)dx}=1.\)
\par
Take an arbitrary \(\epsilon \in (0,1]:\) it is shown next that for some universal \(c_8>0,\)
\begin{equation}\label{epsilonconv}
    |\int{fd\mu}-\int{fd\nu}| \leq (\frac{R^2}{c_8mp})^{m/2}\cdot 3n^{p/c_8}||f||_{\infty}+\epsilon \cdot ||f||_{\infty}+2\sup_{||x-z|| \leq \epsilon\sqrt{m}}{|f(x)-f(z)|},
\end{equation}
which entails (\ref{conclusion}) since by letting \(\epsilon \to 0,\) the last term also tends to \(0:\) else, there exists \(\epsilon_0>0\) such that for all \(k \in \mathbb{N},\) there are \(x_k,z_k \in \mathbb{R}^m\) with \(||x_k||,||z_k||\leq 2R,\) \(|f(x_k)-f(z_k)| \geq \epsilon_0, ||x_k-z_k|| \leq k^{-1}\sqrt{m}\) because for \(\epsilon\leq \frac{R}{\sqrt{m}},\) 
\[\sup_{||x-z|| \leq \epsilon\sqrt{m}}{|f(x)-f(z)|}=\sup_{||x-z|| \leq \epsilon\sqrt{m},||x||,||z|| \leq 2R}{|f(x)-f(z)|};\]
two applications of Bolzano-Weierstrass theorem yield a subsequence \((k_j)_{j \in \mathbb{N}}\) with 
\[\lim_{j \to \infty}{x_{k_j}}=x, \hspace{0.5cm} \lim_{j \to \infty}{z_{k_j}}=z,\] 
and thus \(x=z\) from \(||x-z||=\lim_{j \to \infty}{||x_{k_j}-z_{k_j}||}=0;\) however, this and \(|f(x_k)-f(z_k)| \geq \epsilon_0\) contradict the continuity of \(f\) at \(x.\) 
\par
To justify (\ref{epsilonconv}), let \(\phi:\mathbb{R}^m \to [0,\infty)\) be given by
\[\phi(x)=\prod_{1 \leq i \leq m}{g(x_i)}, \hspace{0.5cm} g=g_1*g_2*...*g_l*G, \hspace{0.5cm} g_i(x)=\epsilon_1 G(x\epsilon_1),\]
for \(l=8, \epsilon_1=\epsilon_1(m,R,\epsilon)>0\) chosen subsequently, and \(h_{\epsilon}=f*\tau_{\epsilon},f_\epsilon=h_{\epsilon}*\phi_\epsilon:\) i.e., 
\begin{equation}\label{defeps}
    h_{\epsilon}(x)=\int_{\mathbb{R}^m}{f(x-y)\epsilon^{-m}\tau(\epsilon^{-1}y)dy}, \hspace{0.5cm} f_{\epsilon}(x)=\int_{\mathbb{R}^m}{h_{\epsilon}(x-y)\epsilon^{-m}\phi(\epsilon^{-1}y)dy}, \hspace{0.4cm} x \in \mathbb{R}^m.
\end{equation}
Return to the difference of interest: the triangle inequality yields
\[|\int{fd\mu}-\int{fd\nu}| \leq |\int{h_\epsilon d\mu}-\int{h_\epsilon d\nu}|+2\sup_{x \in \mathbb{R}^m}{|f(x)-h_\epsilon(x)|} \leq\]
\begin{equation}\label{c1}
    \leq |\int{f_\epsilon d\mu}-\int{f_\epsilon d\nu}|+2\sup_{x \in \mathbb{R}^m}{|h_{\epsilon}(x)-f_\epsilon(x)|}+2\sup_{x \in \mathbb{R}^m}{|f(x)-h_\epsilon(x)|}.
\end{equation}
It is shown next that the three terms in (\ref{c1}) are upper bounded by their counterparts in (\ref{epsilonconv}).
\par
Begin with the last term in (\ref{c1}): since \(\tau \geq 0, \int_{\mathbb{R}^m}{\tau(y)dy}=1, supp(\tau)=[0,1]^m,\)
\[\sup_{x \in \mathbb{R}^m}{|f(x)-h_\epsilon(x)|} \leq \sup_{x \in \mathbb{R}^m}{\int_{[0,\epsilon]^m}{|f(x)-f(x-y)|\epsilon^{-m}\tau(\epsilon^{-1}y)dy}} \leq \sup_{||x-z|| \leq \epsilon\sqrt{m}}{|f(x)-f(z)|}.\]
\par
Continue with the second term in (\ref{c1}): \(h_\epsilon\) is Lipschitz because for \(x,z \in \mathbb{R}^m,\)
\[|h_{\epsilon}(x)-h_{\epsilon}(z)| \leq \int_{\mathbb{R}^m}{f(y)\epsilon^{-m}|\tau(\epsilon^{-1}(x-y))-\tau(\epsilon^{-1}(z-y))|dy} \leq\]
\[\leq ||f||_{\infty} \cdot (\int_{y-x \in [-\epsilon,0]^m}{\epsilon^{-m}|\tau(\epsilon^{-1}(x-y))-\tau(\epsilon^{-1}(z-y))|dy}+\int_{y-z \in [-\epsilon,0]^m}{\epsilon^{-m}|\tau(\epsilon^{-1}(x-y))-\tau(\epsilon^{-1}(z-y))|dy}) \leq\]
\[\leq 2||f||_{\infty} \cdot \epsilon^{-1}c(\tau) \cdot ||x-z||, \hspace{0.5cm} c(\tau)=\sup_{x \ne y}{\frac{|\tau(x)-\tau(y)|}{||x-y||}} \in (0,\infty).\]
This property and \(\phi \geq 0,\int_{\mathbb{R}^m}{\phi(y)dy}=1,\) a consequence of \(g_1,g_2, \hspace{0.05cm}... \hspace{0.05cm},g_l \geq 0,\) from which
\[||g||_{L_1}=||g_1||_{L_1} \cdot ||g_2||_{L_1} \cdot ... \cdot ||g_l||_{L_1}  \cdot||G||_{L_1} =||G||^{l+1}_{L_1}=1, \hspace{0.5cm} ||\phi||_{L_1}=||g||_{L_1}^m=1,\]
entail
\[|h_{\epsilon}(x)-f_\epsilon(x)| \leq \int_{\mathbb{R}^m}{|h_{\epsilon}(x)-h_{\epsilon}(x-y)| \epsilon^{-m}\phi(\epsilon^{-1}y)dy} \leq\]
\[\leq 2||f||_{\infty}\epsilon^{-1}c(\tau)\int_{\mathbb{R}^m}{||y|| \epsilon^{-m}\phi(\epsilon^{-1}y)dy}=2||f||_{\infty}c(\tau)\int_{\mathbb{R}^m}{||y||\phi(y)dy}.\]
Cauchy-Schwarz inequality and \(||\phi||_{L_1}=1\) 
give
\[\int_{\mathbb{R}^m}{||y||\phi(y)dy} \leq (\int_{\mathbb{R}^m}{||y||^2\phi(y)dy})^{1/2}(\int_{\mathbb{R}^m}{\phi(y)dy})^{1/2}=(\int_{\mathbb{R}^m}{||y||^2\phi(y)dy})^{1/2}=(m\int_{\mathbb{R}}{y^2g(y)dy})^{1/2},\]
and
\[\int_{\mathbb{R}}{y^2g(y)dy}=\int_{0}^{1+\frac{l}{\epsilon_1}}{y^2g(y)dy} \leq (\int_{0}^{1+\frac{l}{\epsilon_1}}{y^4dy})^{1/2} \cdot (\int_{0}^{1+\frac{l}{\epsilon_1}}{g^2(y)dy})^{1/2} \leq (1+\frac{l}{\epsilon_1})^{5/2} \cdot \epsilon_1^{l/2} \cdot ||G||_{L_2}^{(l+1)/2},\]
employing \(||h_1*h_2||_{L_2} \leq ||h_1||_{L_2} \cdot ||h_2||_{L_2},\) which implies \(||g||_{L_2} \leq ||G||_{L_2} \prod_{1 \leq i \leq l}{||g_i||_{L_2}}=\epsilon_1^{l} \cdot ||G||_{L_2}^{l+1}.\) Lastly, \(l=8, \epsilon_1 \leq c(G)\epsilon, \epsilon \leq 1\) give
\[2\sup_{x \in \mathbb{R}^m}{|h_{\epsilon}(x)-f_\epsilon(x)|} \leq \epsilon ||f||_{\infty}.\]
\par
Consider now the first term in (\ref{c1}): because \(f_\epsilon\) is smooth and compactly supported, the Fourier inversion formula holds for it and all its partial derivatives (the Fourier transform here is \(\hat{u}(\xi)=\int_{\mathbb{R}^m}{u(x)e^{- 2\pi i x\cdot \xi}dx}\)). Fubini's theorem and (\ref{ineq1}) yield %in light of (\ref{ineq1}),
\[|\int{f_\epsilon d\mu}-\int{f_\epsilon d\nu}|=|\int_{\mathbb{R}^m}{\hat{f}_\epsilon(\theta)\mathbb{E}_{\mu-\nu}[e^{2\pi i\theta \cdot X}]d\theta}| \leq \int_{||\theta|| \leq 2\pi}{|\hat{f}_{\epsilon}(\theta)| \cdot \overline{G}(p,\frac{\theta}{2\pi})d\theta}+2\int_{||\theta||>2\pi}{|\hat{f_{\epsilon}}(\theta)|d\theta},\]
where
\begin{equation}\label{defg}
    \overline{G}(p,\theta):=(1+\sigma_0(\epsilon_0)\sum_{1 \leq i \leq n}{|\theta_{ii}|})^{C(\epsilon_0)p}\exp(c_6(\epsilon_0)p^{45}-2\delta_2(\epsilon_0)\log{n})+ \exp(-c_7p).
\end{equation}
The claimed bound ensues from this last inequality, and
\begin{equation}\label{p1}
    \int_{||\theta||>2\pi}{|\hat{f_{\epsilon}}(\theta)|d\theta} \leq  (\frac{2R^2}{mp})^{m/2} \cdot ||f||_{\infty},
\end{equation}
\begin{equation}\label{p2}
    \int_{||\theta|| \leq 2\pi}{|\hat{f}_{\epsilon}(\theta)| \cdot \overline{G}(p,\frac{\theta}{2\pi})d\theta} \leq (\frac{R^2}{c_{12}(\epsilon_0) \cdot mp})^{m/2} \cdot n^{p/c_{12}(\epsilon_0)} \cdot ||f||_{\infty},
\end{equation}
for \(c_{12}(\epsilon_0)>0.\) %universal.\textcolor{red}{CHECK LAST BOUND}
\vspace{0.3cm}
\par
\textit{Proof of (\ref{p1}):} For \(k \in \mathbb{N},\) Plancherel formula yields
\begin{equation}\label{plancherel}
    \int_{\mathbb{R}^m}{|\hat{f_\epsilon} (\xi)|^2 \cdot ||\xi||^{2k}d\xi}=(2\pi)^{-k}\sum_{|\kappa|=k}{\binom{k}{\kappa_1,\kappa_2, \hspace{0.05cm} ... \hspace{0.05cm},\kappa_m}\int_{\mathbb{R}^m}{(f_\epsilon^{(\kappa)}(x))^2dx}},
\end{equation}
where for \(\kappa=(\kappa_1,\kappa_2, \hspace{0.05cm}...\hspace{0.05cm},\kappa_m) \in \mathbb{Z}_{\geq 0}^{m},|\kappa|:=\kappa_1+\kappa_2+...+\kappa_m, u^{(\kappa)}=\frac{\partial^{k}u}{\partial{x^{\kappa_1}_1}...\partial{x^{\kappa_m}_m}}.\)
Note that
\[f_\epsilon^{(\kappa)}(x)=\int_{\mathbb{R}^m}{h_{\epsilon}(y)\epsilon^{-m}\phi^{(\kappa)}(\epsilon^{-1}(x-y))dy}=\int_{\mathbb{R}^m}{h_{\epsilon}(y)\epsilon^{-m-|\kappa|}\prod_{i \leq m}{g^{(\kappa_i)}(\epsilon^{-1}(x_i-y_i))}dy}=\]
\[=\int_{\mathbb{R}^m}h_{\epsilon}(x-z\epsilon)\epsilon^{-|\kappa|} \prod_{i \leq m}{g^{(\kappa_i)}(z_i)dz},\]
whereby
\[(f_\epsilon^{(\kappa)}(x))^2 \leq \epsilon^{-2|\kappa|} \cdot (\int_{\mathbb{R}^m}{h_{\epsilon}^2(x-z\epsilon)dz}) \cdot \int_{||x-z\epsilon|| \leq R_\epsilon}{(\phi^{(\kappa)}(z))^2dz},\]
for \(R_\epsilon=R+\epsilon\sqrt{m}\) as (\ref{defeps}) gives \(supp(h_{\epsilon}) \subset B_m(R_\epsilon).\) If \(||x|| \leq 2R_\epsilon,\) then
\[(f_\epsilon^{(\kappa)}(x))^2 \leq \epsilon^{-2|\kappa|} \cdot (\int_{\mathbb{R}^m}{h_{\epsilon}^2(x-z\epsilon)dz}) \cdot \int_{\mathbb{R}^m}{(\phi^{(\kappa)}(z))^2dz},\]
and else
\[(f_\epsilon^{(\kappa)}(x))^2 \leq \epsilon^{-2|\kappa|} \cdot (\int_{\mathbb{R}^m}{h_{\epsilon}^2(x-z\epsilon)dz}) \cdot \int_{||z|| \geq \frac{||x||}{2\epsilon}}{(\phi^{(\kappa)}(z))^2dz}.\]
\par
The last two inequalities imply
\[\int_{\mathbb{R}^m}{(f_\epsilon^{(\kappa)}(x))^2dx} \leq \epsilon^{-2|\kappa|} \cdot (\int_{||x|| \leq 2R_\epsilon}{h_{\epsilon}^2(x-z\epsilon)dzdx}) \cdot \int_{\mathbb{R}^m}{(\phi^{(\kappa)}(z))^2dz}+\epsilon^{-2|\kappa|} \cdot (\int_{\mathbb{R}^m}{h_{\epsilon}^2(z\epsilon)dz}) \cdot \int_{||z|| \geq \frac{||x||}{2\epsilon}}{(\phi^{(\kappa)}(z))^2dzdx} \leq\]
\[\leq \epsilon^{-2|\kappa|} \cdot V(B_m(2R_\epsilon)) \cdot ||f||^2_{\infty} \cdot \epsilon^{-m} \cdot V(B_m(R_\epsilon)) \int_{\mathbb{R}^m}{(\phi^{(\kappa)}(z))^2dz}+\epsilon^{-2|\kappa|} \cdot ||f||^2_{\infty} \cdot \epsilon^{-m} \cdot V(B_m(R_\epsilon))  \int_{||z|| \geq \frac{||x||}{2\epsilon}}{(\phi^{(\kappa)}(z))^2dzdx}\]
using \(||h_{\epsilon}||_{\infty} \leq ||f||_{\infty}, supp(h_{\epsilon}) \subset B_m(R_\epsilon),\) %Cauchy-Schwarz inequality entails 
and from \(||h_1*h_2||_{L^2} \leq ||h_1||_{L^2} \cdot ||h_2||_{L^2},\) %and thus 
\[||g^{(k)}||_{L^2} \leq \prod_{1 \leq j \leq l}{||g_j||_{L^2}} \cdot ||G^{(k)}||_{L^2}=(\epsilon_1 ||G||_{L^2})^l \cdot ||G^{(k)}||_{L^2}.\]
Hence
\[||\phi^{(\kappa)}||_{L^2} \leq (\epsilon_1 ||G||_{L^2})^{lk} \cdot \prod_{1 \leq j \leq m}{||G^{(k_j)}||_{L^2}},\]
and as \(G\) is compactly supported, there exists \(c(k)>0\) such that
\[(G^{(j)}(x))^2 \leq \frac{c(k)}{(1+x^2)^j}, \hspace{0.5cm} 0 \leq j \leq k, x \in \mathbb{R},\]
whereby
\[\int_{\mathbb{R}^m}{(f_\epsilon^{(\kappa)}(x))^2dx} \leq \epsilon^{-2|\kappa|-m} \cdot V(B_m(2R_\epsilon)) \cdot ||f||^2_{\infty} \cdot V(B_m(R_\epsilon)) \cdot (\epsilon_1 ||G||_{L^2})^{2lk} \cdot \prod_{1 \leq j \leq m}{||G^{(k_j)}||^2_{L^2}}+\]
\[+\epsilon^{-2|\kappa|-m} \cdot ||f||^2_{\infty} \cdot V(B_m(R_\epsilon)) \cdot  \int_{||z|| \geq \frac{||x||}{2\epsilon}}{\prod_{1 \leq i \leq m}{\frac{c(k)\epsilon_1}{(1+z_i^2\epsilon_1^2)^{k_j}}}dzdx}.\]
The final bound for (\ref{plancherel}) becomes 
\[(2\pi)^{-k}\sum_{|\kappa|=k}{\binom{k}{\kappa_1,\kappa_2, \hspace{0.05cm} ... \hspace{0.05cm},\kappa_m}\int_{\mathbb{R}^m}{(f_\epsilon^{(\kappa)}(x))^2dx}} \leq \]
\[\leq m^k \epsilon^{-2|\kappa|-m} \cdot V(B_m(2R_\epsilon)) \cdot ||f||^2_{\infty} \cdot V(B_m(R_\epsilon)) \cdot (\epsilon_1 ||G||_{L^2})^{2lk} \cdot \max_{\sum_{1 \leq j \leq m}{k_j}=k}{\prod_{1 \leq j \leq m}{||G^{(k_j)}||^2_{L^2}}}+\]
\[+\epsilon^{-2|\kappa|-m} \cdot ||f||^2_{\infty} \cdot V(B_m(R_\epsilon)) \cdot  \int_{||z|| \geq \frac{||x||}{2\epsilon}}{
(\sum_{1 \leq i \leq m}{\frac{c(k)\epsilon_1}{1+z_i^2\epsilon_1^2})^kdzdx}}.\]
Consider the integral:
\[\int_{||z|| \geq \frac{||x||}{2\epsilon}}{
(\sum_{1 \leq i \leq m}{\frac{c(k)\epsilon_1}{1+z_i^2\epsilon_1^2})^kdzdx}}=\int_{||z|| \geq \frac{||x||}{2\epsilon}}{
(\sum_{1 \leq i \leq m}{\frac{c(k)\epsilon_1}{1+z_i^2\epsilon_1^2})^k \cdot V(B_m(2\epsilon ||z||))dz}}=\]
\[=(2\epsilon)^{m}V(B_m(1))\int_{[0,1]^m}{
(\sum_{1 \leq i \leq m}{\frac{c(k)\epsilon_1}{1+z_i^2\epsilon_1^2})^k \cdot ||z||^mdz} \leq (2\epsilon)^{m}V(B_m(1))\int_{[0,1]^m}{
(mc(k)\epsilon_1)^k \cdot ||z||^mdz}}.\]
For \(k=m,\) the two bounds amount to
\[||f||^2_{\infty} \cdot  V(B_m(R_\epsilon)) \cdot V(B(1)) \cdot [\epsilon^{-3m}(2R_\epsilon m)^m(\epsilon_1||G||_{L^2})^{2lm}C(m)+\epsilon^{-2m}(2mc(m)\epsilon_1)^m],\]
and so long as \(\epsilon_1 \leq c(m,R,\epsilon),\) it can be concluded that
\[\int_{\mathbb{R}^m}{|\hat{f_\epsilon} (\xi)|^2 \cdot ||\xi||^{2k}d\xi} \leq ||f||^2_{\infty} \cdot (\frac{R^2}{mp})^{2m},\]
from which (\ref{p1}) ensues due to Cauchy-Schwarz inequality,
\[\int_{||\theta||>2\pi}{|\hat{f_{\epsilon}}(\theta)|d\theta} \leq (\int_{\mathbb{R}^m}{|\hat{f_{\epsilon}}(\theta)|^2 \cdot ||\theta||^{2k}d\theta})^{1/2}(\int_{||\theta||>2\pi}{||\theta||^{-2k}d\theta})^{1/2}  \leq ||f||_{\infty} \cdot (\frac{2R^2}{mp})^{m}\]
due to
\[\int_{||\theta||>2\pi}{||\theta||^{-2k}d\theta}=A(\mathbb{S}_m(1))\int_{a>2\pi}{a^{-2k+m-1}da} \leq A(\mathbb{S}_m(1)) \cdot \frac{(2\pi)^{-m}}{m} \leq 3^m,\]
where \(\mathbb{S}_m(1)=\{x: x \in \mathbb{R}^m,||x||=1\}\) with area \(A(\mathbb{S}_m(1))=\frac{(2\pi)^{m/2}}{\Gamma(\frac{m}{2})} \leq (2\pi)^{m/2} \leq 3^m.\)
\vspace{0.3cm}
\par
\textit{Proof of (\ref{p2}):} All constants in this part depend on \(\epsilon_0:\) let \(c_{10},c_{11}>0\) such that 
\[\overline{G}(p,\theta) \leq 2(1+\sigma_0\sum_{1 \leq i \leq n}{|\theta_{ii}|})^{Cp}\exp(-c_{10}p||\theta||^2) \leq 2^{Cp} \cdot [1+(n\sigma_0||\theta||)^{Cp/2}] \cdot \exp(-c_{10}p||\theta||^2)\] 
for \(||\theta||\leq 1, p \leq c_{11}(\log{n})^{1/45}\) (recall (\ref{defg}) and use \((1+a)^q \leq 2^{q-1}(1+a^q),\sum_{1 \leq i \leq n}{|\theta_{ii}|} \leq n^{1/2}||\theta||\)), whereby under the aforesaid conditions, 
\[\overline{G}(p,\theta) \leq n^{\tilde{C}p}\cdot \exp(-\tilde{c}_{10}p||\theta||^2).\]
Finally,
\[\int_{||\theta|| \leq 2\pi}{|\hat{f}_\epsilon(\theta)| \cdot \overline{G}(p,\frac{\theta}{2\pi})d\theta} \leq n^{\tilde{C}p}\int_{||\theta|| \leq 2\pi}{|\hat{f}_\epsilon(\theta)| \cdot \exp(-\tilde{c}_{10}p \cdot ||\theta||^2)d\theta},\]
along with
\[|\hat{f}_\epsilon(\theta)| \leq ||f_\epsilon||_{L_1}=||f||_{L_1} \leq R^m \cdot V(B_m(1)) \cdot ||f||_{\infty},\]
and
\[\int_{\mathbb{R}^m}{\exp(-c||\theta||^2)d\theta}=(\int_{\mathbb{R}}{\exp(-cx^2)dx})^m=(\frac{\sqrt{\pi}}{2\sqrt{c}})^m \leq c^{-m/2},\]
entails (\ref{p2}) (as \(V(B_m(1))=\frac{\pi^{m/2}}{\Gamma(\frac{m}{2}+1)}\)).

\end{proof}

\subsection{Two Implications}\label{s3}

This subsection contains the proofs of Corollary~\ref{cor1} and Theorem~\ref{th7}.

\begin{proof} (Corollary~\ref{cor1})
Suppose without loss of generality that \(||f||_{\infty} \leq 1,\) and \(||f_l||_{\infty} \leq 2\) for all \(l \in \mathbb{N}.\) Since \(\mu_{n,p}(\mathbb{R}^m)=\nu_{n,p}(\mathbb{R}^m)=1,\) and both measures are nonnegative,
\[\lim_{l \to \infty}{(\int{f_ld\mu_{n,p}}-\int{f_ld\nu_{n,p}})}=\int{fd\mu_{n,p}}-\int{fd\nu_{n,p}}.\]
Let \(\phi:\mathbb{R}^m \to [0,1],\phi|_{B_m(\frac{c_2\sqrt{mp}}{2})}=1,\phi|_{B^c_m(c_2\sqrt{mp})}=0\) continuous\footnote{\(S^c\) denotes the complement of \(S\subset \mathbb{R}^m.\)} (e.g., for \(K_1\) closed, \(K_1 \subset K_2,d(K_1,K_2^c)>~0,\)
\begin{equation}\label{fzeroone}
    \phi(x)=\frac{d(x,K^c_2)}{d(x,K_1)+d(x,K^c_2)},
\end{equation}
is continuous with \(\phi(\mathbb{R}^m) \subset [0,1], \phi(K_1)=1, \phi(K^c_2)=0,\) where
\[d(S_1,S_2)=\inf_{s_1 \in S_1,s_2 \in S_2}{||s_1-s_2||}, \hspace{0.5cm} S_1,S_2 \subset \mathbb{R}^m:\]
the denominator is always positive since \(d(x,K_1)=0 \Leftrightarrow x \in K_1,\) and \(d(K_1,K_2^c)>0\)). Recall that \(||f_l||_{\infty} \leq 2,\) from which \(\frac{\phi f_l}{2} \in \mathcal{R}_{n,p},\) with Theorem~\ref{th55} entailing
\[|\int{\phi f_ld\mu_{n,p}}-\int{\phi f_ld\nu_{n,p}}| \leq 3c^m_1+4C(\epsilon_0)n^{-\epsilon_0/8}.\]
Then
\[|\int{f_ld\mu_{n,p}}-\int{f_ld\nu_{n,p}}| \leq |\int{\phi f_ld\mu_{n,p}}-\int{\phi f_ld\nu_{n,p}}|+ \int_{(B(\frac{c_2\sqrt{mp}}{2}))^c}{|f_l|d\mu_{n,p}}+\int_{(B(\frac{c_2\sqrt{mp}}{2}))^c}{|f_l|d\nu_{n,p}} \leq \]
\[\leq 3c^m_1+4C(\epsilon_0)n^{-\epsilon_0/8}+\mathbb{P}_{\mu_{n,p}}(||X||>\frac{c_2\sqrt{mp}}{2})+\mathbb{P}_{\nu_{n,p}}(||X||>\frac{c_2\sqrt{mp}}{2}) \leq 3c^m_1+4C(\epsilon_0)n^{-\epsilon_0/8}+\frac{C}{m \cdot (\frac{pc_2^2}{8})^2},\]
for \(p \geq C',\) by Chebyshev's inequality, and %\textcolor{red}{CHECK CONSTANTS}
\[\mathbb{E}_{\mu_{n,p}}[||X||^2]=O(m), \hspace{0.3cm} Var_{\mu_{n,p}}(||X||^2)=O(m),\] 
consequences of the proof of Lemma~\ref{ll2}.
\end{proof}

\begin{proof}(Theorem~\ref{th7}) \(A \in Sym^d(n)\) if and only if \(F(A)=1,\) where
\[F(M)=\prod_{1 \leq i \leq n-1}{\chi_{\lambda_i(M) \ne \lambda_{i+1}(M)}}.\] 
The set of discontinuities of \(F,\) \(D:=\{M: F(M)=0\},\) is a null set under the Lebesgue measure \(\eta\) on \(\mathbb{R}^m\) because \(\eta\) is absolutely continuous with respect to the distribution of a \(n \times n\) Wigner matrix with i.i.d. entries that are centered standard normal random variables, and the latter assigns measure zero to the sets 
\[(\{M \in \mathbb{R}^{n \times n}, M=M^T, \lambda_i(M)=\lambda_{i+1}(M)\})_{1 \leq i \leq n}\] 
from the change of variables taking the entries of \(M\) to its eigenvalues and eigenvectors. For \(k \in \mathbb{N},\) let \(F_k:\mathbb{R}^{m} \to [0,1]\) be continuous with 
\[F_k(D)=1, \hspace{0.3cm} F_k(S_k)=0, \hspace{0.3cm}  S_k=\{M: \min_{1 \leq i \leq n-1}{(\lambda_i(M)-\lambda_{i+1}(M))}>\frac{1}{k^2}\}\] 
(use (\ref{fzeroone}) as well as the continuity of \(M \to \lambda_i(M),\) a consequence of Weyl's inequalities
\[\lambda_i(M_2)+\lambda_n(M_1-M_2) \leq \lambda_i(M_1) \leq \lambda_i(M_2)+\lambda_1(M_1-M_2),\]
which entail \(|\lambda_i(M_1) -\lambda_i(M_2)| \leq ||M_1-M_2|| \leq \sqrt{tr((M_1-M_2)^2)},\) the latter being the distance between \(M_1\) and \(M_2\) when viewed as vectors, i.e., elements of \(\mathbb{R}^{n^2}\)), \(\phi:\mathbb{R}^m \to [0,1]\) smooth, with \(\int_{\mathbb{R}^m}{\phi(x)dx}=1,\) and \(G_k:\mathbb{R}^{m} \to \mathbb{R}, G_k=F_k* \phi_{1/k}:\) i.e., %for \(x \in \mathbb{R}^m,\)
\[G_k(x)=\int_{\mathbb{R}^{m}}{F_k(y)k^{m}\phi(k(x-y))dy}, \hspace{0.5cm} x \in \mathbb{R}^m.\]
Then
\[|G_k(x)-(1-F)(x)|=|\int_{\mathbb{R}^{m}}{(F_k(y)+F(y)-1) \cdot k^{m}\phi(k(x-y))dy}| \leq\]
\begin{equation}\label{ptbd}
    \leq \int_{M:\min_{1 \leq i \leq n-1}{(\lambda_i(M)-\lambda_{i+1}(M))} \leq \frac{1}{k^2}}{k^{m}\phi(k(x-M))dM} \leq \eta(\{M:\min_{1 \leq i \leq n-1}{(\lambda_i(M)-\lambda_{i+1}(M))}) \leq \frac{1}{k}\}=c(k) \to 0
\end{equation}
using \(F_k(y)+F(y)-1=0+1-1=0\) for \(y \in S_k,\) \(|F_k(y)+F(y)-1| \leq 1\) for \(y \in \mathbb{R}^m,\) and the translation invariance of \(\eta\) yielding \(\lim_{k \to \infty}{c(k)}=\eta(D)=0.\) Corollary~\ref{cor1} renders
\[\mathbb{P}(A \not \in Sym^d(n)) \leq \int{(1-F)d\mu_{n,p}}=|\int{(1-F)d\mu_{n,p}}-\int{(1-F)d\nu_{n,p}}| \leq 6c^m_1+4C(\epsilon_0)n^{-\epsilon_0/8}+\frac{c_3}{mp^2}\]
(the integral with respect to \(\nu_{n,p}\) is zero because this measure is absolutely continuous with respect to the Lebesgue measure \(\eta\)), giving the first part of the theorem. %\textcolor{blue}{Pmu?}
\par
The second part follows because 
\[\int{f\chi_{F=0}d\nu_{n,p}}=\int{f\chi_{F=0}d\nu_{n}}=0, \hspace{0.2cm}  |\int{f\chi_{F=1}d\nu_{n,p}}-\int{f\chi_{F=1}d\nu_{n}}| \leq Cp^{-1/2},\]
since \(\frac{c(2p,\mathbb{E}[a^4_{11}])}{C_p}-\frac{\sigma+16}{8}=O(p^{-1/2})\) from the proof of Lemma~\ref{lems}, and for \(a>0,\) %\textcolor{blue}{c2 ratio/convergence}
\[|\exp(-x^2)-\frac{1}{a}\exp(-\frac{x^2}{a^2})|=|\exp(-x^2) \cdot (1-\exp(-\frac{x^2(1-a^2)}{a^2}))-(1-a) \cdot \frac{1}{a}\exp(-\frac{x^2}{a^2})| \leq\]
\[\leq \exp(-x^2) \cdot \frac{x^2|1-a^2|}{a^2}+|1-a| \cdot \frac{1}{a}\exp(-\frac{x^2}{a^2}).\]
Corollary~\ref{cor1} yields the desired convergence because \(f\chi_{F=1}=fF\) is the pointwise limit of \(f(1-G_k)\) (\(||f||_{\infty} \leq 1,\) \(f,G_k\) are continuous, and \(F-(1-G_k)=G_k-(1-F)\) together with (\ref{ptbd})). 
    
\end{proof}

\subsection{Haar Measures}\label{subsechaar2}

Let \(Z=(z_{ij})_{1 \leq i,j \leq n} \in \mathbb{R}^{n \times n}\) have i.i.d. entries with \(z_{11} \overset{d}{=} N(0,1).\) Since for \(1 \leq i \leq n,\)
\[||(\frac{1}{\sqrt{n}}Z)_i||^2=1+\frac{1}{n}\sum_{1 \leq i \leq n}{(z^2_{ij}-1)},\]
Bernstein's inequality (theorem \(2.8.1\) in Vershynin~\cite{vershynin}) entails
\[\mathbb{P}(\frac{1}{n}|\sum_{1 \leq i \leq n}{(z^2_{ij}-1)}| \geq t) \leq 2\exp(-K\min{(\frac{nt^2}{||z^2_{11}-1||^2_{\psi_1}},\frac{nt}{||z^2_{11}-1||_{\psi_1}}})),\]
for \(||X||_{\psi_1}=\inf{\{t>0: \mathbb{E}[\exp(-|X|/t)] \leq 2\}},\) whereby for some \(C,T>0,\) and all \(t \geq T,\)
\[\mathbb{P}(\frac{1}{n}\max_{1 \leq i \leq n}{|\sum_{1 \leq i \leq n}{(z^2_{ij}-1)}|} \geq t \cdot \sqrt{\frac{\log{n}}{n}}) \leq  2n^{1-Ct}.\]
Choose \(\alpha=\alpha_n=(\frac{2}{C}+T+1)\sqrt{\frac{\log{n}}{n}},\) for which 
\[\int_{Z \in (O(n))_\alpha}{dZ} \geq 1-2n^{-1}.\]
A union bound entails for \(t \geq T'=2 \cdot (\frac{2}{C}+T+1),\) %\textcolor{red}{n or n2?}
\[\mathbb{P}(U \in O(n): \max_{1 \leq i,j \leq n}{|u_{ij}|} \geq t\cdot \sqrt{\frac{\log{n}}{n}}) \leq n^2\mathbb{P}(\frac{1}{\sqrt{n}}\max_{1 \leq i \leq n}{|z_{1i}|} \geq \frac{t}{2} \cdot \sqrt{\frac{\log{n}}{n}}) \leq n^2(1-\exp(-cn^{c'(1-t^2/4)}))\]
as \(\mathbb{P}(\max_{1 \leq i \leq n}{|z_{1i}|} \leq t\sqrt{\log{n}}) \geq (1-(2\pi)^{-1/2}\exp(-\frac{t^2\log{n}}{4}))^n\) since when \(a>0,\) 
\[\frac{1}{\sqrt{2\pi}}\int_{a}^{\infty}{\exp(-\frac{y^2}{2})dy} \leq \frac{1}{\sqrt{2\pi}} \cdot \exp(-\frac{a^2}{2})\int_{0}^{\infty}{\exp(-ay)dy}=\frac{1}{a\sqrt{2\pi}} \cdot \exp(-\frac{a^2}{2}) \leq \frac{1}{\sqrt{2\pi}} \cdot \exp(-\frac{a^2}{4}).\]
The quantum unique ergodicity version (\ref{bouryau}) proved by Bourgade and Yau~\cite{bourgadeyau} and the strong version (\ref{cipo}) by Cipolloni, Erdös and Schröder~\cite{cipolini} can be justified by a similar rationale (Gaussianity allows to reduce (\ref{cipo}) to \(M\) diagonal with trace zero and \(\max_{1 \leq i \leq n}{|M_{ii}|} \leq C;\) for the dot products \(u_i^Tu_j,\) use Cauchy-Schwarz inequality to deal with the errors).

\section{Trace Expectations}\label{traccov}

This section contains the proofs of Lemma~\ref{l1} and Theorem~\ref{th1} (subsections~\ref{sect11} and \ref{4sect.1}, respectively).

\subsection{Proof of Lemma~\ref{l1}}\label{sect11}

Consider the power series 
\[S(X)=\sum_{k \geq 1}{\beta(k,\gamma)X^k},\]
whose radius of convergence is \(\frac{1}{b(\gamma)}=(1+\sqrt{\gamma})^{-2}\) (\(\beta(k,\gamma)\) is the \(k^{th}\) moment of a probability distribution supported on \([a(\gamma),b(\gamma)],\) with positive mass on \([b(\gamma)-\epsilon,b(\gamma)]\) for all \(\epsilon>0\)). Then
\[S^2(X)=\sum_{r \geq 2}{(\sum_{1 \leq k \leq r-1}{\beta(k,\gamma)\beta(r-k,\gamma)})X^r},\]
whereby (\ref{momrec}) is equivalent to
\[S^2(X)=\sum_{r \geq 2}{\frac{\beta(r+1,\gamma)-(1+\gamma)\beta(r,\gamma)}{\gamma} \cdot X^r}=\frac{1}{\gamma X}(S(X)-X-X^2(\gamma+1))-\frac{1+\gamma}{\gamma}(S(X)-X)\]
or
\begin{equation}\label{Seq}
    \gamma X S^2(X)+((\gamma+1)X-1)S(X)+X=0.
\end{equation}
When \(x \in \mathbb{C},||x||<(1+\sqrt{\gamma})^{-2},\)
\[S(x)=\sum_{k \geq 1}{\int_{a(\gamma)}^{b(\gamma)}{x^ky^kf_\gamma(y)dy}}=\int_{a(\gamma)}^{b(\gamma)}{\frac{xy}{1-xy} f_\gamma(y)dy}=-1+\int_{a(\gamma)}^{b(\gamma)}{\frac{1}{1-xy} f_\gamma(y)dy}=-1-\frac{1}{x}m(\frac{1}{x}),\]
where \(m(z)=\int_{\mathbb{R}}{\frac{1}{x-z}f_\gamma(x)dx}\) is the Stieltjes transform of \(f_\gamma,\) entailing (\ref{Seq}) as
\[\gamma z m^2(z)-(1-\gamma-z)m(z)+1=0\]
for \(z \in \mathbb{C}^{+},\) and \(S,m\) are analytic in \(\{z \in \mathbb{C}, ||z||<(1+\sqrt{\gamma})^{-2}\} \cap \mathbb{C}^{+}\) (lemma \(3.11\) in \cite{baisilvbook}).

\subsection{Proof of Theorem~\ref{th1}}\label{4sect.1}

Fix \(p:\) (\ref{trcov}) yields
\[\mathbb{E}[tr(A^p)]=\sum_{(i_0,k_0,i_1,k_1, \hspace{0.05cm}... \hspace{0.05cm},i_{p-1},k_{p-1})}{\mathbb{E}[x_{i_0k_0}x_{i_1k_0}x_{i_1k_1}x_{i_2k_1}...x_{i_{p-1}k_{p-1}}x_{i_0k_{p-1}}]}.\]
By convention, \(i_l \in \{1,2, \hspace{0.05cm} ... \hspace{0.05cm}, n\}, k_l \in \{1,2, \hspace{0.05cm} ... \hspace{0.05cm}, N\}\) for \(0 \leq l \leq p-1;\) denote by \(\mathcal{G}(p,n,N)\) the set of directed graphs with edges 
\[(i_0,k_0),(i_1,k_0),(i_1,k_1),(i_2,k_1), \hspace{0.05cm} ... \hspace{0.05cm}, (i_{p-1},k_{p-1}),(i_0,k_{p-1}),\] 
and vertices \(i_l \in \{1,2, \hspace{0.05cm} ... \hspace{0.05cm}, n\}, k_l \in \{1,2, \hspace{0.05cm} ... \hspace{0.05cm}, N\}\) for \(0 \leq l \leq p-1.\) Let \(\mathbf{g}=(\mathbf{i},\mathbf{j})\) be a generic element of \(\mathcal{G}(p,n,N)\) with \(\mathbf{i}=(i_0,i_1, \hspace{0.05cm} ... \hspace{0.05cm},i_{p-1}),\mathbf{j}=(k_0,k_1, \hspace{0.05cm} ... \hspace{0.05cm},k_{p-1}),\) and
\[x_{\mathbf{g}}:=x_{i_0k_0}x_{i_1k_0}x_{i_1k_1}x_{i_2k_1}...x_{i_{p-1}k_{p-1}}x_{i_0k_{p-1}}.\]  
\par
Consider the map taking elements of \(\mathcal{G}(p,n,N)\) and reversing the orientation of the even-indexed edges: \(\newline \mathcal{W}:\mathcal{G}(p,n,N) \to \mathcal{D}(p,n,N),\) for \(\mathcal{D}(p,n,N)\) the set of directed graphs with edges 
\[(i_0,k_0),(k_0,i_1),(i_1,k_1),(k_1,i_2), \hspace{0.05cm} ... \hspace{0.05cm}, (i_{p-1},k_{p-1}),(k_{p-1},i_0),\] 
and vertices \(i_l \in \{1,2, \hspace{0.05cm} ... \hspace{0.05cm}, n\}, k_l \in \{1,2, \hspace{0.05cm} ... \hspace{0.05cm}, N\}\) for \(0 \leq l \leq p-1.\)
\par
Take the expectations induced by the elements of \(\mathcal{D}(p,n,N):\) symmetry yields only those with undirected edges of even multiplicity have non-vanishing contributions, a property invariant under this mapping. Call the subset of such cycles \(\mathcal{E}(p,n,N),\) %which in turn is contained in \(\mathcal{E}(p,n+N),\)
%where \(\mathcal{E}(l,m)\) denotes 
i.e., the set of even cycles of size \(2l\) that belong to the range of \(\mathcal{W}.\) %and vertices in \(\{1,2, \hspace{0.05cm} ... \hspace{0.05cm},m\}.\) 
Return to the image of \(\mathcal{G}(p,n,N),\) and consider the dominant graphs in (\ref{trcov}). Since such graphs \(\mathbf{g}\) have edges solely of even multiplicity, their images under \(\mathcal{W}\) are elements of \(\mathcal{E}(p,n,N).\) Note
\begin{equation}\label{wineq}
    \mathbb{E}[x_{\mathbf{g}}] \leq \mathbb{E}[x_{\mathcal{W}(\mathbf{g})}],
\end{equation}
a consequence of independence and Hölder's inequality, \(\mathbb{E}[x_{11}^{2q}] \cdot \mathbb{E}[x_{11}^{2l}] \leq \mathbb{E}[x_{11}^{2q+2l}].\) If \(\mathcal{W}(\mathbf{g}) \in \mathcal{C}(p),\) then it has \(p\) pairwise distinct undirected edges, each of multiplicity \(2;\) using the recursion \(2.\) in section~\ref{treigen} that \((\mathcal{C}(l))_{l \in \mathbb{N}}\) satisfy, induction on \(l\) yields no element of \(\mathcal{C}(l)\) contains two copies of a directed edge, entailing equality occurs in (\ref{wineq}) (there is an even number of edges between two such copies as together they form an element of \((\mathcal{C}(l))_{l \in \mathbb{N}},\) and from this it ensues that these copies are identical in \(\mathbf{g}\)). It suffices to consider \(\mathcal{C}(p,n,N):=\mathcal{C}(p) \cap \mathcal{W}(\mathcal{D}(p,n,N)),\) i.e., the simple even cycles whose odd numbered vertices belong to \(\{1,2, \hspace{0.05cm} ... \hspace{0.05cm}, N\},\) and its even numbered to \(\{1,2, \hspace{0.05cm} ... \hspace{0.05cm}, n\}:\) all the expectations of such elements are \(1,\) and take 
\[\beta(p,n,N)=\sum_{0 \leq i \leq p-1}{|\mathcal{C}_{i}(p)| \cdot n^{i+1}N^{p-i}},\] 
where 
\[\mathcal{C}_{j}(p)=\{\mathbf{i}: \mathbf{i} \in \mathcal{C}(p), j+1=|\{i_{2k}, 0 \leq k \leq p-1\}|,p-j=|\{i_{2k+1}, 0 \leq k \leq p-1\}|\}.\] This constitutes a good estimate of the quantity of interest because for \(p=o(\sqrt{n}),\) 
\[|\mathcal{C}(p,n,N)| \in [\sum_{0 \leq i \leq p-1}{|\mathcal{C}_{i}(p)| \cdot \prod_{0 \leq j \leq p-i-1}{(N-j)}}\prod_{0 \leq j \leq i}{(n-p-j)},\sum_{0 \leq i \leq p-1}{|\mathcal{C}_{i}(p)| \cdot N^{p-i}n^{i+1}}],\]
whereby \(|\mathcal{C}(p,n,N)|=\beta(p,n,N)(1-o(1))\) using that for \(j \leq \sqrt{m},\)
\[m^j-(m-1)(m-2)...(m-j)=\sum_{1 \leq k \leq j}{m^{j-k}(-1)^{k-1}\sum_{1 \leq j_1<j_2<...<j_k \leq j}{j_1j_2...j_k}} \leq\]
\[\leq \sum_{1 \leq k \leq j}{m^{j-k} \cdot \frac{(1+2+...+j)^k}{k!}} \leq \sum_{1 \leq k \leq j}{m^{j-k}\cdot \frac{j^{2k}}{k!}} \leq m^{j}(\exp(\frac{j^2}{m})-1).\]
\par
The recurrent description \(2.\) in section~\ref{treigen} of \(\mathcal{C}(p)\) gives a partition of \(\mathcal{C}_j(p)\) into three subsets, implying
\[\beta(p+1,n,N)=nN^{p+1}+n^{p+1}N+\sum_{1 \leq i \leq p}{(|\mathcal{C}_{i}(p)|+|\mathcal{C}_{i-1}(p)|+\sum_{0 \leq k \leq i-1, 1 \leq a \leq p-1}{|\mathcal{C}_{k}(a)| \cdot |\mathcal{C}_{i-k-1}(p-a)|}) \cdot n^{i+1}N^{p+1-i}}.\]
Let \(\beta_0(p,X)=\sum_{0 \leq i \leq p-1}{|\mathcal{C}_{i}(p)| \cdot X^{i+1}},\) for which this, upon division by \(N^{p+2},\) entails when \(X=n/N,\)
\[\beta_0(p+1,X)=X+X^{p+1}+\sum_{1 \leq i \leq p-1}{(|\mathcal{C}_{i}(p)|+|\mathcal{C}_{i-1}(p)|+\sum_{0 \leq k \leq i-1, 1 \leq a \leq p-1}{|\mathcal{C}_{k}(a)| \cdot |\mathcal{C}_{i-k-1}(p-a)|}) \cdot X^{i+1}}=\]
\[=X^{p+1}+\beta_0(p,X)+X(\beta_0(p,X)-X^p)+\sum_{1 \leq i \leq p-1}{\sum_{0 \leq k \leq i-1, 1 \leq a \leq p-1}{|\mathcal{C}_{k}(a)| \cdot |\mathcal{C}_{i-k-1}(p-a)|} \cdot X^{i+1}}\]
or
\[\beta_0(p+1,X)=(1+X)\beta_0(p,X)+\sum_{1 \leq a \leq p-1}{\beta_0(a,X)\beta_0(p-a,X)}.\]
Induction and Lemma~\ref{l1} give
\[\beta_0(p,X)=X\beta(p,X)=X\sum_{0 \leq r \leq p-1}{\frac{1}{r+1}\binom{p}{r} \binom{p-1}{r}X^{r+1}},\] 
providing the leading term in (\ref{tracecovarnew}) (the last equality is a consequence of (\ref{betadef})). For the second-order terms, employ (\ref{wineq}) and the bound for the Wigner case, i.e.,
\[\mathbb{E}[tr(B^p)]=n(1+O(\frac{p^2}{n}))\]
for \(p=o(\sqrt{n}), B=n^{-1/2}(b_{ij})_{1 \leq i,j \leq n}, B=B^T, (b_{ij})_{1 \leq i \leq j \leq n}\) i.i.d., \(b_{11} \overset{d}{=} x_{11}\) (see \cite{sinaisosh}). %symmetric and subgaussian of finite norm (i.e., \(\mathbb{E}[b^{2k}_{11}] \leq (Ck)^k\) for all \(k \in \mathbb{N},\) and some \(C>0\)). 
This concludes the proof of Theorem~\ref{th1}.

\section*{Appendix}

This section looks into the asymptotic behavior of the functions \(f,s,\mathcal{A}_1,\mathcal{A}_2,\) defined by (\ref{keyeq}),(\ref{formm}), and (\ref{ref1}), respectively.

\begin{lemma}\label{fpklemma}
    For \(k,p \in \mathbb{N},1 \leq k \leq p,\)
    \begin{equation}\label{rbdratiodens}
    f_{p,k+1}=C_{p-1} \cdot \frac{k}{2^{k-1}}\prod_{0 \leq j \leq \lfloor \frac{k-1}{2} \rfloor-1}{(1-\frac{2\lfloor \frac{k}{2} \rfloor -1}{2p-3-2j})}.
\end{equation}
    In particular, there are absolute constants \(c_{10},c_{11},c_{12},c_{13}>0\) such that for all such \(k,p,\) 
\begin{equation}\label{rbdgrowthdpk}
    \exp(-\frac{c_{10}k^2}{p}-\frac{c_{11}}{p}) \leq \frac{f_{p,k+1}}{C_{p-1} \cdot \frac{k}{2^{k-1}}} \leq \exp(-\frac{c_{12}k^2}{p}+\frac{c_{13}}{p}).
\end{equation}  
\end{lemma}

\begin{proof}
By definition (\ref{keyeq}),
\begin{equation}\label{keyerq}
    f_{p,k+1}=\binom{2p-k-1}{p-k}-\binom{2p-k-1}{p-k-1}=\frac{k}{2p-k}\binom{2p-k}{p-k},
\end{equation}
for \(1 \leq k \leq p\) (with \(\binom{n}{-1}=0\)). This entails that for \(2 \leq k \leq p,\)
\begin{equation}\label{recc}
    f_{p,k+1}=f_{p,k}-f_{p-1,k-1}
\end{equation}
since
\[f_{p,k}-f_{p,k+1}=(\binom{2p-k-1}{p}-\binom{2p-k-1}{p})-(\binom{2p-k-2}{p}-\binom{2p-k-2}{p})=\binom{2p-k-2}{p-1}-\binom{2p-k-2}{p-1},\]
and (\ref{recc}) renders
\begin{equation}\label{alphaf}
    f_{p,k+1}=\tilde{\alpha}(p-1,k-1),
\end{equation}
where
\begin{equation}\label{alphatilde}
    \tilde{\alpha}(p,k):=C_{p} \cdot \frac{k+1}{2^{k}}\prod_{0 \leq j \leq \lfloor \frac{k}{2} \rfloor-1}{(1-\frac{2\lfloor \frac{k+1}{2} \rfloor -1}{2p-1-2j})}.
\end{equation}
To see (\ref{alphaf}), note the formula holds for \(k \in \{1,2\}\) (\(f_{p,2}=C_{p-1},f_{p,3}=\frac{2}{2p-2}\binom{2p-2}{p-2}=\frac{1}{p}\binom{2p-2}{p-1}=C_{p-1}\)), %for the definition of \(f\) above in terms of binomial coefficients), 
and it suffices to show \(\tilde{\alpha}\) also satisfies the recurrence encapsulated by (\ref{recc}) to conclude the claim via induction on \(k:\) this follows from the following identities for \(k \geq 1,\)
\[\tilde{\alpha}(p,2k)-\tilde{\alpha}(p,2k+1)=C_p \cdot [(2k+1) \prod_{0 \leq j \leq k-1}{\frac{p-k-j}{4p-2-4j}}-(k+1)\prod_{0 \leq j \leq k-1}{\frac{p-k-j-1}{4p-2-4j}}]=\]
\[=C_p \cdot \frac{\prod_{0 \leq j \leq k-2}{(p-k-j)}}{\prod_{0 \leq j \leq k-1}{(4p-2-4j)}} \cdot [(2k+1) \cdot (p-k)-(p-2k) \cdot (k+1)]=\]
\[=C_p \cdot \frac{\prod_{0 \leq j \leq k-2}{(p-k-j)}}{\prod_{0 \leq j \leq k-1}{(4p-2-4j)}} \cdot k(p+1)=C_{p-1}\cdot k \cdot \prod_{0 \leq j \leq k-2}{\frac{p-k-j}{4p-6-4j}}=\tilde{\alpha}(p-1,2k-1)\]
because \(\frac{C_p}{C_{p-1}}=\frac{4p-2}{p+1},\) and
\[\tilde{\alpha}(p,2k+1)-\tilde{\alpha}(p,2k+2)=
C_p \cdot [(k+1)\prod_{0 \leq j \leq k-1}{\frac{p-k-1-j}{4p-2-4j}-(2k+3)\prod_{0 \leq j \leq k}{\frac{p-k-1-j}{4p-2-4j}}}=\]
\[=C_p\cdot \prod_{0 \leq j \leq k-1}{\frac{p-k-1-j}{4p-2-4j}} \cdot [(k+1)-(2k+3) \cdot \frac{p-2k-1}{4p-2-4k}]
=C_p\cdot \frac{(2k+1)(p+1)}{4p-2-4k}\prod_{0 \leq j \leq k-1}{\frac{p-k-1-j}{4p-2-4j}}=\]
\[=C_{p-1}(2k+1)\prod_{0 \leq j \leq k-1}{\frac{p-k-j-1}{4p-6-4j}}=\tilde{\alpha}(p-1,2k).\]
\par
Showing \(\frac{\tilde{\alpha}(p,k)}{C_p \cdot \frac{k+1}{2^k}}\) admits lower and upper bounds of the type appearing in (\ref{rbdgrowthdpk}) is enough to derive it in light of (\ref{alphaf}): this follows from (\ref{alphatilde}) inasmuch as for \(2|k,P \geq \max{(k-1,5)},\)
\[\sum_{0 \leq j \leq k/2}{\frac{1}{P-j}}=-\log{(1-\frac{k/2+1}{P+1})}+\frac{k+2}{4P(P-k/2)}+O(\frac{k}{P(P-k/2)^2}),\]
\[-k\log{(1-\frac{k/2+1}{P+1})}+\frac{k(k+2)}{4P(P-k/2)}=\Theta(\frac{k^2}{P}), \hspace{0.5cm} \frac{k^2}{P(P-k/2)^2}=O(\frac{1}{P})\]
(use these results for \((2(\lfloor \frac{k}{2} \rfloor-1),p-\frac{1}{2}),\) and \(2(\lfloor \frac{k}{2} \rfloor-1) \leq 2(\frac{k}{2}-1) \leq p-2 \leq p-\frac{3}{2}\)): the second claim ensues from \(x \to \frac{-\log{(1-x)}}{x}=\sum_{n \geq 1}{\frac{x^{n-1}}{n}}\) increasing in \((0,\frac{2}{3}],\) giving for \(k \leq P+1,\)
\[-k\log{(1-\frac{k/2+1}{P+1})}+\frac{k(k+2)}{4P(P-k/2)} \leq \frac{3\log{3}}{2} \cdot \frac{k(k/2+1)}{P+1}+\frac{k(k+2)}{4P(P-k/2)} 
 \leq \frac{k(k+2)}{4(P+1)} \cdot (3\log{3}+\frac{2(P+1)}{P(P-1)})\] 
 \[-k\log{(1-\frac{k/2+1}{P+1})}+\frac{k(k+2)}{4P(P-k/2)} \geq 
\frac{k(k+2)}{4(P+1)} \cdot (2+\frac{P+1}{P(P-1)})\] 
while the first arises from
\[\sum_{0 \leq n \leq M-m}{\frac{1}{n+m}}=\log{\frac{M+1}{m}}+\frac{M-m+1}{2mM}+O(\frac{M-m}{m^2M})\]
for \(m=P-k/2 \geq \max{(k/2-1,5-k/2)} \geq 2,M=P:\)
\[\sum_{0 \leq n \leq M-m}{\frac{1}{n+m}}=\log{\frac{M+1}{m}}+\sum_{0 \leq n \leq M-m}{(\frac{1}{n+m}-\log{(1+\frac{1}{n+m})})}=\]
\[=\log{\frac{M+1}{m}}+\frac{1}{2}\sum_{0 \leq n \leq M-m}{\frac{1}{(n+m)^2}}+O(\sum_{0 \leq n \leq M-m}{\frac{1}{(n+m)^3}})=\]
\[=\log{\frac{M+1}{m}}+\frac{1}{2} \cdot [\frac{M-m+1}{mM}+O(\frac{M-m}{m^2M})]+O(\frac{M-m}{m^2M})=\log{\frac{M+1}{m}}+\frac{M-m+1}{2mM}+O(\frac{M-m}{m^2M}),\]
when \(m \in \mathbb{R}, m \geq 2,M-m \in \mathbb{N},\) using 
\[\frac{M-m+1}{(M+1)m}=\int_{m}^{M+1}{\frac{1}{x^2}dx} \leq \sum_{0 \leq n \leq M-m}{\frac{1}{(n+m)^2}} \leq \int_{m-1}^{M}{\frac{1}{x^2}dx}=\frac{M-m+1}{M(m-1)},\]
\[\frac{2}{m^2}-\frac{2}{(M+1)^2}=\int_{m}^{M+1}{\frac{1}{x^3}dx} \leq \sum_{0 \leq n \leq M-m}{\frac{1}{(n+m)^3}} \leq \int_{m-1}^{M}{\frac{1}{x^3}dx}=\frac{2}{(m-1)^2}-\frac{2}{M^2}.\]
\end{proof}

\begin{lemma}\label{lems}
    %For \(s\) %:\mathbb{Z}_{\geq 0} \to \mathbb{Z}_{\geq 0}\) 
    %given by (\ref{formm}),
    \[\lim_{p \to \infty}{\frac{s(p)}{C_{p-1}}}=12,\hspace{0.5cm} \lim_{p \to \infty}{\frac{\mathcal{A}_1(p)}{C_{p-1}}}=\frac{\sigma(\sigma+6)}{2}, \hspace{0.5cm} \lim_{p \to \infty}{\frac{\mathcal{A}_2(2p)}{C^2_{p-1}}}=6(2\sigma+9),\]
    where
    \begin{equation}\label{sigmadef}
        \sigma=\sum_{b \geq 0}{\frac{s(b)}{4^b}}.
    \end{equation}
    %where \(\mathcal{A}_1,\mathcal{A}_2\) are given by (\ref{ref1}).
\end{lemma}

\begin{proof}
A key ingredient behind these convergences is (\ref{rbdgrowthdpk}). Since
\[\frac{s(p)}{C_{p-1}}=\sum_{1 \leq k \leq p}{\frac{kf_{p,k+1}}{C_{p-1}}} \leq \exp(\frac{c_{13}}{p-1})\sum_{1 \leq k \leq p}{\frac{k^2}{2^{k-1}}},\]
\[\frac{s(p)}{C_{p-1}}=\sum_{1 \leq k \leq p}{\frac{kf_{p,k+1}}{C_{p-1}}} \geq \sum_{1 \leq k \leq (p-1)^{1/4}}{\frac{kf_{p,k+1}}{C_{p-1}}} \geq \exp(-\frac{c_{10}}{(p-1)^{1/2}}-\frac{c_{11}}{p-1})\sum_{1 \leq k \leq p^{1/4}}{\frac{k^2}{2^{k-1}}},\]
and for \(|x|<1,\)
\[\sum_{k \geq 1}{kx^{k-1}}=(\sum_{k \geq 0}{x^{k}})'=(\frac{1}{1-x})'=\frac{1}{(1-x)^2}, \hspace{0.2cm} \sum_{k \geq 1}{k(k+1)x^{k-1}}=(\sum_{k \geq 1}{x^{k+1}})''=(\frac{x^2}{1-x})''=(\frac{2x-x^2}{(1-x)^2})'=\frac{2}{(1-x)^3},\]
the first result follows from the two bounds above together with
\[\sum_{k \geq 1}{\frac{k^2}{2^{k-1}}}=\sum_{k \geq 1}{\frac{k(k+1)}{2^{k-1}}}-\sum_{k \geq 1}{\frac{k}{2^{k-1}}}=\frac{2}{1/8}-\frac{1}{1/4}=12.\] 
For \(\mathcal{A}_1,\mathcal{A}_2,\) the following result is vital:
\begin{equation}\label{partiala}
    \lim_{p \to \infty}{\frac{1}{C_p}\sum_{0 \leq b \leq p}{s(b)C_{p-b}}}=\sigma+6 \in (6,\infty)
\end{equation}
(note \(s(b) \leq 12\exp(c_{13})C_{b-1} \leq C \cdot \frac{4^b}{b^{3/2}}\) guarantees \(\sigma<\infty\)). An immediate consequence of this is
\[\lim_{p \to \infty}{\frac{\mathcal{A}_2(2p)}{C^2_{p-1}}}=2\cdot [2 \cdot \frac{12}{4}\cdot (\sigma+6)-(\frac{12}{4})^2]=6(2\sigma+9).\]
It remains to justify (\ref{partiala}) and use it to derive the asymptotic behavior of \(\mathcal{A}_1.\)
\par
For \(\epsilon>0\) and \(p \geq p_1(\epsilon),\)
\[\frac{1}{C_p}\sum_{0 \leq b \leq p}{s(b)C_{p-b}} \leq (1-\frac{3\sqrt{p}}{2(p-\sqrt{p})})\sum_{0 \leq b \leq \sqrt{p}}{\frac{s(b)}{4^b}}+\frac{\exp(c_4)}{C_p}\sum_{\sqrt{p}<b<p-\sqrt{p}}{C_bC_{p-b}}+\frac{1}{C_p}\sum_{p-\sqrt{p} \leq b \leq p}{s(b)C_{p-b}} \leq\]
\[\leq (1-\frac{3\sqrt{p}}{2(p-\sqrt{p})})\sum_{0 \leq b \leq \sqrt{p}}{\frac{s(b)}{4^b}}+2\exp(c_4)\sum_{b>\sqrt{p}}{\frac{C}{b^{3/2}}}+\frac{3+\epsilon}{C_p}\sum_{p-\sqrt{p} \leq b \leq p}{C_bC_{p-b}}\]
from 
\[\frac{C_{p+1}}{C_p}=\frac{p+1}{p+2} \cdot \frac{\binom{2p+2}{p+1}}{\binom{2p}{p}}=\frac{4p+2}{p+2}, \hspace{0.2cm} \frac{C_p}{4^bC_{p-b}}=\prod_{1 \leq k \leq b}{(1-\frac{3}{2(p-b+j)})} \geq (1-\frac{3}{2(p-b)})^{b} \geq 1-\frac{3b}{2(p-b)},\]
\(s(b) \leq 12\exp(c_{13})C_{b-1} \leq (3+\epsilon)C_{b}\) for \(b \geq b(\epsilon),\) as well as \(\frac{C_aC_b}{C_{a+b}} \leq \frac{C}{\min{((a+1)^{3/2},(b+1)^{3/2})}}\) for \(a,b \geq 0,\) whereby
\[\limsup_{p \to \infty}{\frac{1}{C_p}\sum_{0 \leq b \leq p}{s(b)C_{p-b}}} \leq \sigma+6\]
because when \(p \geq 5,\)
\begin{equation}\label{lb}
    \frac{1}{C_p}\sum_{p-\sqrt{p} \leq b \leq p}{C_bC_{p-b}} \leq \frac{1}{2C_p}\sum_{b \geq 0}{C_bC_{p-b}}=\frac{C_{p+1}}{2C_p} \leq 2;
\end{equation}
similarly, for arbitrary \(\epsilon>0\) and \(p \geq p_2(\epsilon),\)
\[\frac{1}{C_p}\sum_{0 \leq b \leq p}{s(b)C_{p-b}} \geq \sum_{0 \leq b \leq \sqrt{p}}{\frac{s(b)}{4^b}}+\frac{1}{C_p}\sum_{p-\sqrt{p} \leq b \leq p}{s(b)C_{p-b}} \geq \sum_{0 \leq b \leq \sqrt{p}}{\frac{s(b)}{4^b}}+\frac{3-\epsilon}{C_p}\sum_{p-\sqrt{p} \leq b \leq p}{C_bC_{p-b}},\]
from which 
\[\liminf_{p \to \infty}{\frac{1}{C_p}\sum_{0 \leq b \leq p}{s(b)C_{p-b}}} \geq \sigma+6,\]
due to
\[\frac{C_{p+1}}{C_p}-\frac{2}{C_p}\sum_{p-\sqrt{p} \leq b \leq p}{C_bC_{p-b}}=\frac{C_{p+1}}{C_p}-\frac{1}{C_p}\sum_{p-\sqrt{p} \leq b \leq p}{C_bC_{p-b}}-\frac{1}{C_p}\sum_{0 \leq b \leq \sqrt{p}}{C_bC_{p-b}}=\]
\begin{equation}\label{ub}
    =\frac{1}{C_p}\sum_{\min{(b,p-b)}>\sqrt{p}}{C_bC_{p-b}} \leq 2C\sum_{b>\sqrt{p}}{\frac{1}{b^{3/2}}}
\end{equation}
for \(p \geq 5.\) This completes the proof of (\ref{partiala}).
\par
Return now to \(\mathcal{A}_1:\) since for any \(c \in (0,1),\)
\[\lim_{p \to \infty}{\frac{1}{C_{p}}\sum_{a>cp}{s(a)C_{p-a}}}=6, \hspace{0.5cm} \lim_{p \to \infty}{\frac{1}{C_{p}}\sum_{a \leq cp}{s(a)C_{p-a}}}=\sigma\]
(via a similar rationale to the one above: the range \(a \in (cp,p-\sqrt{p})\) makes a negligible contribution, and  \(\lim_{p \to \infty}{\frac{1}{C_p}\sum_{p-\sqrt{p} \leq a \leq p}{C_aC_{p-a}}}=2\) from (\ref{lb}) and (\ref{ub})), these together with (\ref{partiala}) yield
\[\lim_{p \to \infty}{\frac{2}{C_{p-1}}\sum_{0 \leq a \leq \frac{p-1}{4},0 \leq b \leq \frac{p-1}{2}}{s(a)s(b)C_{p-2-a-b}}}=\frac{\sigma^2}{2},\]
\[\lim_{p \to \infty}{\frac{2}{C_{p-1}}\sum_{\frac{p-1}{4} < a \leq \frac{p-1}{4},0 \leq b \leq \frac{p-1}{2}}{s(a)s(b)C_{p-2-a-b}}}=3\sigma,\]
from
\[\frac{2}{C_{p-1}}\sum_{0 \leq a \leq \frac{p-1}{4}}{s(a)\sum_{b \leq \frac{p-1}{2}}{s(b)C_{p-2-a-b}} \geq \frac{2}{C_{p-1}}\sum_{0 \leq a \leq \frac{p-1}{4}}{(\sigma-\epsilon)s(a)C_{p-2-a}} \geq \frac{2C_{p-2}}{C_{p-1}}(\sigma-\epsilon)^2},\]
\[\frac{2}{C_{p-1}}\sum_{0 \leq a \leq \frac{p-1}{4}}{s(a)\sum_{b \leq \frac{p-1}{2}}{s(b)C_{p-2-a-b}} \leq \frac{2}{C_{p-1}}\sum_{0 \leq a \leq \frac{p-1}{4}}{(\sigma+\epsilon)s(a)C_{p-2-a}} \leq \frac{2C_{p-2}}{C_{p-1}}(\sigma+\epsilon)^2},\]
and
\[\frac{2}{C_{p-1}}\sum_{0 \leq b \leq \frac{p-1}{2}}{s(b)\sum_{a>\frac{p-1}{4}}{s(a)C_{p-2-a-b}} \geq \frac{2(6-\epsilon)}{C_{p-1}}}\sum_{0 \leq b \leq \frac{p-1}{2}}{s(b)C_{p-2-b}} \leq \frac{2(6-\epsilon)(\sigma-\epsilon) C_{p-2}}{C_{p-1}},\]
\[\frac{2}{C_{p-1}}\sum_{0 \leq b \leq \frac{p-1}{2}}{s(b)\sum_{a>\frac{p-1}{4}}{s(a)C_{p-2-a-b}}} \leq \frac{2(6+\epsilon)}{C_{p-1}}\sum_{0 \leq b \leq \frac{p-1}{2}}{s(b)C_{p-2-b}} \leq \frac{2(6+\epsilon)(\sigma+\epsilon) C_{p-2}}{C_{p-1}}\]
for \(p \geq p(\epsilon),\)  using \(\frac{p-2}{2} \leq p-2-a \leq p-2 \leq 4 \cdot \frac{p-1}{4}.\) This completes the proof of the lemma.
\end{proof}

\bibliographystyle{unsrtnat}
\bibliography{references}

\end{document}